\documentclass{amsbook}
\usepackage{amsmath}
\usepackage{amssymb}
\usepackage{graphicx}

\newtheorem{thm}{Theorem}[section]
\newtheorem{cor}[thm]{Corollary}
\newtheorem{lem}[thm]{Lemma}
\newtheorem{prop}[thm]{Proposition}
\theoremstyle{definition}
\newtheorem{defn}[thm]{Definition}
\theoremstyle{remark}
\newtheorem{rem}[thm]{Remark}
\newtheorem{ex}{Example}

\begin{document}
\frontmatter
\title[Symmetric matrices]{Symmetric Matrices: Theory \& Applications}
\author[Helmut Kahl]{Helmut Kahl}
\email{kahl@hm.edu}
\address{%
Munich University of Applied Sciences\\
Germany}
\date{January 2014+}
\maketitle
\tableofcontents

\mainmatter
\section{Introduction}
Matrices are of the most important objects of mathematical applications. This text restricts to the symmetric ones for moderating the scope of this text, for aesthetic reasons, because of a long history and because there are still many applications. Symmetric matrices $(\alpha_{i j}) \in \mathbb{O}^{n \times n}$ ($\alpha_{i j} = \alpha_{j i}$) over an integral domain $\mathbb{O}$ with $1+1 \ne 0$ are nothing else but quadratic forms $q:M \to \mathbb{O}$ with respect to a basis $e_1,...,e_n$ of the $\mathbb{O}$-module $M$ (cf. section \ref{sec_quadraticforms}); the bijective correspondence given by the equations $2 \alpha_{i j} = q(e_i+e_j)-q(e_i)-q(e_j)$.\footnote{Don't worry if you don't understand this yet. It will be explained by Proposition \ref{prop_quadraticforms}.} Quadratic forms are used to define quadrics (i.e. conics; cf. section \ref{sec_quadrics}) which represent the most simple non-linear algebraic varieties. In dimension two and three quadrics were investigated already in the Greek-Hellenistic antiquity (s. \cite{Scriba}, sect.2.2.2, p.42 and sect.2.5.10, p.92). For instance, the notions "ellipse, parabola, hyperbola" were already used by Apollonios of Perge (262?-190 B.C.) in his extensive examinations of conics \cite{Apollonios}. In his "Recherches d'Arithm\'etique" of 1773-1775 (s. \cite{Lagrange}, p.695-758) J.L. Lagrange investigated binary quadratic forms $a x^2 + b x y + c y^2$ with arbitrary integral coefficients $a,b,c$. In 1795-1800 C.F. Gauss revealed much deeper results on binary quadratic forms\footnote{and also on ternary quadratic forms} in his "Disquisitiones Arithmeticae" \cite{Gauss}. Analytic aspects of binary quadratic forms were founded by P.G.L. Dirichlet in the 1830s and 1840s (s. \cite{DD}). In order to get a rather\footnote{See the comment at the end of subsection \ref{subsec_rational}!} complete theory of quadratic forms, H. Hasse in the 1920s and E. Witt in 1936, restricted the category of underlying rings to the category of fields. Since then an immense amount of literature on quadratic forms has appeared. (See e.g. the references of \cite{Cassels} to receive an impression!) This text deals with non-symmetric quadratic matrices, too. But essentially, it describes basic theory of symmetric matrices and yields applications to various fields like Numerical Analysis, Geometry, Statistics and Cryptography. For moderation of the scope of this text, the theory of hermitian matrices, though also very important for applications, will be omitted. For the reader's convenience there is an appendix about basic analytical and algebraical facts. Nevertheless, knowledge of the real numbers and rudiments of linear algebra would be helpful. The examples are - not only but also - meant to be (implicit or explicit) exercises; The reader should verify their assertions or solve the problems posed there. The first draft of this paper was written for a summer school in Sao Joao Del Rei, Brazil, 2014. In the meantime the author has used revised versions of it as a lecture script at his home university during several summer semesters.

\section{Basic notions and notations}
Most of the conventions in this section are propably known to the reader. We discuss them in order to standardise the language of this text. The symbol $\mathbb{O}$ is used for integral domains (s. Def. \ref{def_ring}!) whose most important instance is the set $\mathbb{Z}$ of rational integers. But in the beginning of this section we restrict to the algebraic structure of a commutative ring $R$. The symbol $\mathbb{K}$ stands for fields like the set $\mathbb{Q}$ of rational numbers or the set $\mathbb{R}$ of real numbers.

\begin{defn}
For two sets $I,J$ the set $I \times J := \lbrace (i,j) : i \in I , j \in J \rbrace$ is called the \textit{cartesian product} of $I$ with $J$. For two discrete sets $I , J$ a function $A : I \times J \to M$ is called a \textit{matrix over a set} $M$. Its values $\alpha_{i j} := A(i,j)$ are called \textit{entries}. Often, the function is denoted by $(\alpha_{i j})$ or, more precisely, by $(\alpha_{i j})_{i \in I,j \in J}$. The first argument $i$ is called \textit{row index} and the second argument $j$ \textit{column index}. The matrix $A^t : J \times I \to M$ defined by $A^t(j,i) := A(i,j)$ (i.e. change of roles: row $\leftrightarrow$ column) is called the \textit{transpose of matrix} $A : I \times J \to M$. A matrix that coincides with its transpose is called \textit{symmetric}.
\end{defn}

\begin{ex}
A schedule is a non-symmetric matrix with row index set e.g.
\begin{equation*}
I := \lbrace \text{08:15-9:45},\text{9:45-10:00},\text{10:00-11:30},\text{11:45-13:15},\text{13:15-14:15} \rbrace
\end{equation*}
of time intervals, column index set
\begin{equation*}
J := \lbrace \text{Monday},\text{Tuesday},\text{Wednesday},\text{Thursday},\text{Friday} \rbrace
\end{equation*}
of weekdays and set $M := \lbrace \text{lecture},\text{break} \rbrace$ of entries.
\end{ex}

From now we restrict to finite index sets of the form $\mathbb{N}_n := \lbrace 1,2,...,n \rbrace$ or $\lbrace 0, 1,...,n \rbrace$ where $n$ is an element of the set $\mathbb{N}$ of all natural numbers. A matrix $(\alpha_{i j})_{i \in \mathbb{N}_m , j \in \mathbb{N}_n}$ is called a $m \times n$\textit{ matrix} and can be described by a rectangular table as follows.
\begin{equation*}
\left(\begin{matrix} \alpha_{1 1} &\ \alpha_{1 2} &\dots &\ \alpha_{1 n} \\ \alpha_{2 1} &\ \alpha_{2 2}  &\dots &\ \alpha_{2 n} \\ \vdots  &\vdots &\dots &\vdots \\  \alpha_{m 1} &\ \alpha_{m 2} &\dots &\ \alpha_{m n} \end{matrix}\right)
\end{equation*}
In case $m=n$ it is called \textit{quadratic}. Hence, a symmetric matrix is quadratic. A $1 \times n$ matrix is called \textit{row vector} and a $m \times 1$ matrix \textit{column vector}. In this text, but in section \ref{sec_analyticAppl}, an element of $M^n = M \times M \times ... \times M$ is considered as a row vector. To endow the set $M^{m \times n}$ of all ${m \times n}$ matrices over $M$ with an algebraic structure we require of $M$ having some algebraic structure: From now on $M$ is a commutative ring $R$ with additive neutral element $0$ and multiplicative neutral element $1$ (s. Definition \ref{def_ring}).

\begin{ex}\label{ex_Vandermonde}
The \textit{Vandermonde} matrix $(x_i^j)_{i,j \in \left\{0 , 1 , ... , n \right\}}$ over $R$ is symmetric if and only if there is some $\omega \in R$ with $x_i = \omega^i$ for all $i \in \left\{0 , 1 , ... , n \right\}$.
\end{ex}

\begin{defn}\label{def_matrixOps}
For two $m \times n$ matrices $A , B: \mathbb{N}_m \times \mathbb{N}_n \to R$ with entries  $\alpha_{i j}$ and  $\beta_{i j}$, respectively, we define its \textit{sum} $A + B := (\alpha_{i j}+\beta_{i j})$. For an $l \times m$ matrix $A = (\alpha_{i j})$ and an $m \times n$ matrix $B = (\beta_{j k})$ the matrix
\begin{equation*}
A B := \left(\sum\limits_{j=1}^{m} \alpha_{i j} \beta_{j k}\right)_{i \in \mathbb{N}_l,k \in \mathbb{N}_n}
\end{equation*}
is called the \textit{product} of $A$ with $B$. The \textit{Kronecker symbol} $\delta_{i j} := 1$ for $i = j$ and $\delta_{i j} := 0$ for $i \ne j$ defines the \textit{identity matrix} $E_n := (\delta_{i j})_{i,j \in \mathbb{N}_n}$. A matrix $(\alpha_{i j})_{i,j \in \mathbb{N}_n}$ with $\alpha_{i j} = 0$ for all $i \ne j$ is called \textit{diagonal}. It is obviously symmetric and will be denoted by $\textnormal{diag}(\alpha_{1 1},...,\alpha_{n n})$. A diagonal matrix $\textnormal{diag}(\delta_1,...,\delta_n)$ with equal \textit{main diagonal entries} $\delta_i$ is called a \textit{scalar matrix}, e.g. $\textnormal{diag}(1,...,1) = E_n$. The scalar matrix $(0) = \textnormal{diag}(0,...,0)$ is called the \textit{zero matrix}. For a scalar $\lambda \in R$ and a matrix $A$ over $R$ we write $\lambda A := \textnormal{diag}(\lambda,...,\lambda) A$.
\end{defn}

\begin{ex}
For two column vectors $x,y \in R^{m \times 1}$ the product $x \circ y := x^t y \in R$ is called the \textit{scalar product} of $x$ and $y$. Then it holds $A B = (a_i \circ b_j)_{i,j}$ for a matrix $A$ with rows $a_i$ and a matrix $B$ with columns $b_j$.
\end{ex}

\begin{rem}\label{rem_matrixOps}
a) With addition and multiplication above, the set $R^{n \times n}$ of all $n \times n$-matrices becomes a non-commutative (s. next example!) ring with neutral element $(0)$ w.r.t. addition and neutral element $E_n$ w.r.t. multiplication. The mulplication with scalars makes this ring even an algebra over $R$.

b) For $\lambda \in R$ and $A \in R^{m \times n}$ it holds $(\lambda \delta_{i j})_{i,j \in \mathbb{N}_m} A = \lambda A = A (\lambda \delta_{i j})_{i,j \in \mathbb{N}_n}$. In particular, any quadratic matrix commutes with every scalar matrix. For an element $\lambda$ of the multiplicative group $R^{\times}$ of all units (s. Remark \ref{rem_ring}d)!) we write $A/\lambda := \frac{1}{\lambda} A$.
\end{rem}

\begin{ex}\label{ex_nonCommutativity}
Over $R$ it holds
\begin{equation*}
\left(\begin{matrix} 1 &\ 1 \\ 0 &\ 1 \end{matrix}\right) \left(\begin{matrix} 0 &\ 1 \\ 1 &\ 0 \end{matrix}\right) = \left(\begin{matrix} 1 &\ 1 \\ 1 &\ 0 \end{matrix}\right) \ne \left(\begin{matrix} 0 &\ 1 \\ 1 &\ 1 \end{matrix}\right) = \left(\begin{matrix} 0 &\ 1 \\ 1 &\ 0 \end{matrix}\right) \left(\begin{matrix} 1 &\ 1 \\ 0 &\ 1 \end{matrix}\right) .
\end{equation*}
\end{ex}

\begin{ex}\label{ex_angularInertia}
For the \textit{vector product}
\begin{equation*}
(\alpha,\beta,\gamma) \times (\delta,\varepsilon,\xi) := (\beta \xi - \gamma \varepsilon , \gamma \delta - \alpha \xi , \alpha \varepsilon - \beta \delta)
\end{equation*}
it holds $(a \times b) \times a = a \times (b \times a) = b (a a^t E_3 - a^t a)$ for all row vectors $a , b \in R^3$. The first equation follows from $a \times b = - b \times a$ and the second from the \textit{Grassmann-identity} $a \times (b \times c) = a c^t b - a b^t c$ for $a , b , c \in R^3$. In physics the symmetric matrix $I := m (a a^t E_3 - a^t a)$ is the inertia tensor of a particle of mass $m > 0$ at point $a \in \mathbb{R}^3$ (up to physical units). Then $b I = m a \times b \times a$ is the corresponding angular momentum where $b \in \mathbb{R}^3$ denotes the angular velocity.
\end{ex}

\begin{prop}\label{prop_determinant}
For a quadratic matrix $A = (a_{i j}) \in R^{n \times n}$ the element, recursively well defined by\footnote{Laplace's expansion by $j$-th column} \begin{equation*} |A|:= \sum\limits_{i=1}^{n} (-1)^{i+j} a_{i j} |A_{i j}| , |(a)|:= a \end{equation*}
with $A_{i j}$ evolving from $A$ by deleting the $i$-th row and $j$-th column, is the same for all $j \in \mathbb{N}_n$ and equals\footnote{Laplace's expansion by $i$-th row} \begin{equation*} \sum\limits_{j=1}^{n} (-1)^{i+j} a_{i j} |A_{i j}| \end{equation*} for all  $i \in \mathbb{N}_n$. In other words:
\begin{equation*} \textnormal{adj}(A) A = A \; \textnormal{adj}(A) = |A|E_n \end{equation*}
for $\textnormal{adj}(A) := ((-1)^{i+j} |A_{i j}|)_{j,i}$. For $A,B \in R^{n \times n}$ we have $|A B| = |A||B|$.
\end{prop}
\begin{proof}
See e.g. \cite{SerreD}, ch.3.1\&2 where $R$ is an integral domain. But the proofs work also for commutative rings like in Def. \ref{def_matrixOps}.
\end{proof}

\begin{rem}\label{rem_det_tr}
a) The \textit{determinant} $|A|$ is a very important notion. The Proposition shows that $\textnormal{adj}(A) / |A|$ is the inverse of $A$ if $|A|$ is a unit. And vice versa: If $A \in R^{n \times n}$ is invertible, i.e. $A B = E_n$ for some $B \in R^{n \times n}$, then $|A| |B| = |A B| = |E_n| = 1$ shows that $|A|$ is a unit.

b) Another somewhat less important notion is the {\it trace}
\begin{equation*}
\textnormal{tr}((\alpha_{i j})):= \alpha_{1 1} + ... + \alpha_{n n}
\end{equation*}
of $(a_{i j}) \in R^{n \times n}$ with the obvious property $\textnormal{tr}(A+B) = \textnormal{tr}(A)+ \textnormal{tr}(B)$.

c) Because of $(A^t)_{i j} = (A_{j i})^t$ it holds $\textnormal{adj}(A^t) = \textnormal{adj}(A)^t$. So $\textnormal{adj}(A)$ is symmetric when $A$ is so.
\end{rem}

\begin{ex}\label{ex_det_tr}
a) Compute the determinant and the trace of
\begin{equation*}
\left(\begin{matrix} \alpha &\ \beta \\ \gamma &\ \delta \end{matrix}\right) \in R^{2 \times 2} .
\end{equation*}

b) Prove $\textnormal{tr}(A B) = \textnormal{tr}(B A)$ for all $A,B \in R^{n \times n}$.

c) Show that the determinant of the Vandermonde matrix $(x_i^j)_{i,j \in \left\{0 , 1 , ... , n \right\}}$ is the product of all $x_j - x_i$ with $i < j$. Hint: Use Laplace's expansion by the $n$-th row and induction on $n \in \mathbb{N}$.
\end{ex}

\begin{prop}\label{prop_GL}
The \textit{general linear group}
\begin{equation*}
\textnormal{GL}_n(R) := \lbrace A \in R^{n \times n} : |A| \in R^{\times} \rbrace
\end{equation*}
of invertible $n \times n$-matrices is a group under matrix multiplication. The determinant function $\textnormal{det}: R^{n \times n} \to R , A \mapsto |A|$ induces a group epimorphism $\textnormal{det}:\textnormal{GL}_n(R) \to R^{\times}$. Its kernel $\textnormal{SL}_n(R) := \lbrace A \in R^{n \times n} : |A| = 1 \rbrace$ is a normal subgroup of $\textnormal{GL}_n(R)$. The elements of the factor group $\textnormal{GL}_n(R) / \textnormal{SL}_n(R)$ are the cosets $\textnormal{diag}(\varepsilon,1,...,1) \textnormal{SL}_n(R)$, $\varepsilon \in R^{\times}$.
\end{prop}
\begin{proof}
The determinant function is surjective because for every $\varepsilon \in R$ we have $|D_\varepsilon|=\varepsilon$ with $D_\varepsilon := \textnormal{diag}(\varepsilon,1,...,1)$. So the first two assertions follow by Remark \ref{rem_det_tr}. That $\textnormal{SL}_n(R)$ is a normal subgroup is due to Proposition \ref{prop_homomorphism}. For a matrix $A$ be of determinant $\varepsilon \in R^{\times}$ we have $|A D_\varepsilon^{-1}| = \varepsilon \varepsilon^{-1} = 1$ and thus $A \: \textnormal{SL}_n(R) = D_\varepsilon \textnormal{SL}_n(R)$. This proves the last assertion.
\end{proof}

\begin{ex}\label{ex_invertibility}
a) The matrices of Example \ref{ex_nonCommutativity} are invertible.

b) Let $\mathbb{K}$ be a field and $\omega \in \mathbb{K}$ a \textit{primitive} $n$-\textit{th root of unity}, i.e. $\omega , \omega^2 , ... , \omega^n = 1$ are pairwise different. Show that $\omega$ and $n$ (= $n$-th sum of $1$) are units in $\mathbb{K}$ and that the inverse of the symmetric Vandermonde matrix\footnote{Around 1805 Gauss found an algorithm, called \textit{Fast Fourier-Transformation}, for computing $V(\omega) x$ for any $x \in \mathbb{K}^n$ very efficiently if $n$ was a power of two. For details see \cite{GG}, ch. 8.2!} $V(\omega) := (\omega^{i+j})_{i,j \in \lbrace 0 , 1 , ... , n-1 \rbrace}$ (s. Example \ref{ex_Vandermonde}) is $V(\omega^{-1})/n$. Hint: Use Example \ref{ex_det_tr}c) and compute $V(\omega) V(\omega^{-1})$.
\end{ex}

\begin{rem}
It holds $(A B)^t = B^t A^t$. But in general, the product of symmetric matrices is not symmetric as shown by Example \ref{ex_nonCommutativity}.
\end{rem}

Nevertheless, for some commutative rings $R$ the set
\begin{equation*}
\textnormal{Sym}_n(R) := \lbrace Q \in R^{n \times n} : Q^t = Q \rbrace
\end{equation*}
of symmetric $n \times n$-matrices will reveal some interesting invariants under certain right actions of $\textnormal{GL}_n(R)$ (s. section \ref{sec_classification}). The special case $n=2$ has been studied most intensively. Thereby, the identity
\begin{equation*}
\left(\begin{matrix} \alpha &\ \beta \\ \gamma &\ \delta \end{matrix} \right) J = J \left(\begin{matrix} \delta &\ -\gamma \\ -\beta &\ \alpha \end{matrix} \right) \textnormal{ with } J := \left(\begin{matrix} 0 &\ 1 \\ -1 &\ 0 \end{matrix} \right)
\end{equation*}
and arbitrary $\alpha,\beta,\gamma,\delta \in R$ is of interest. For $A \in \textnormal{GL}_2(R)$ it means
\begin{equation}\label{eq_equivConnection}
A^t J = |A| J A^{-1} \textnormal{ and } J A^t / |A| = A^{-1} J .
\end{equation}
Here and in what follows we denote by $A^{-1}$ the inverse matrix of $A$.

\begin{defn}\label{def_similarity}
Two quadratic matrices $A , B \in R^{n \times n}$ are called \textit{similar} (\textit{over the ring} $R$) when there is some $T \in \textnormal{GL}_n(R)$ with $T^{-1} A T = B$.
\end{defn}

\begin{rem}\label{rem_similarity}
This defines an important equivalence relation on $R^{n \times n}$, since the linear map $x \mapsto y := A x$ is described by $B := T^{-1} A T$ with respect to the basis that consists of the columns of $T$; i.e. $x = T x'$ and $y = T y'$ imply $y' = B x'$. The equivalence class of a scalar matrix is the set formed of that single matrix. The investigation of the equivalence class of a non-scalar matrix is one of the major tasks of linear algebra. Many great mathematicians like Cauchy, Cayley, Frobenius, Gauss, H.G. Grassmann, Hamilton, Hermite, Jacobi, C. Jordan, Minkowski, Perron, E. Schmidt, Schur, Smith, Sylvester, Vandermonde attended and contributed to this task. An important invariant under \textit{conjugation} $A \mapsto T^{-1} A T$ is the \textit{characteristic polynomial} ({\it function}) $x \mapsto |A - x E_n| , x \in R$ (s. \cite{SerreD}, prop.3.11). In particular, the determinant, the trace and the eigenvalues are invariant under conjugation. 
\end{rem}

\begin{ex}\label{ex_similarity}
The following three matrices fulfill the equation of Definition \ref{def_similarity} over an arbitrary ring.
\begin{equation*}
A := \left(\begin{matrix} 1 &\ 0 \\ 0 &\ -1 \end{matrix} \right) \:,\: B := \left(\begin{matrix} 1 &\ 2 \\ 0 &\ -1 \end{matrix} \right) \:,\: T := \left(\begin{matrix} 1 &\ 1 \\ 0 &\ 1 \end{matrix} \right)
\end{equation*}
Hence $A$ and $B$ are similar. Over a field, two non-scalar $2 \times 2$-matrices are similar if and only if they have same trace and determinant (s. Corollary \ref{cor_similarity}).
\end{ex}

The following algebraic notions are essential for understanding symmetric matrices (not only over fields). From now we presuppose $R = \mathbb{O}$ being an integral domain.

\begin{defn}\label{def_basis}
For an integral domain $\mathbb{O}$ an $\mathbb{O}$-module $M$ is called \textit{finitely generated} when $M$ consists of its (additive) zero element $o$ only or when there are $e_1,...,e_n \in M$ s.t. every element $e$ of $M$ is a {\it linear combination} of the $e_i$, i.e. $e = \alpha_1 e_1 + ... + \alpha_n e_n$ for some $\alpha_i \in \mathbb{O}$. It is called {\it free} when the $e_i$ can be chosen s.t. they are {\it linearly independent} over $\mathbb{O}$, i.e. $\alpha_1 e_1 + ... + \alpha_n e_n = o$ implies that all the $\alpha_i$ vanish. In this case  $e_1,...,e_n$ is called a {\it basis} of $M$.
\end{defn}

\begin{rem}\label{rem_module}
a) Two bases of a free finitely generated $\mathbb{O}$-module $M \ne \lbrace o \rbrace$ have the same number of (generating) elements. This number $\dim M$ is called the {\it rank} or {\it dimension} of $M$. For a proof of this fact see e.g. \cite {vdWaerden2}, art.132. Hence such a module $M$ is nothing else but the image of the \textit{arithmetic module} $\mathbb{O}^n$ under an injective linear map. A basis of $M$ is given by the images of the \textit{canonical unit vectors} $(\delta_{1 j},...,\delta_{n j})$, $j \in \mathbb{N}_n$.

b) Over a field $\mathbb{K}$ the maximal number of linearly independent rows of a matrix coincides with the maximal number of linearly independent columns (s. \cite{SerreD}, prop.2.7). This number $\textnormal{rk}(A)$ is called the \textit{rank} of the matrix $A \in \mathbb{K}^{m \times n}$. It is the dimension of the vectorspace $\lbrace A x \in \mathbb{K}^m : x \in \mathbb{K}^n \rbrace$. A matrix $A$ is called of \textit{full rank} when $\textnormal{rk}(A)$ equals its number of rows or columns.

c) For a linear map $l:M \to N$ between $\mathbb{O}$-modules $M,N$ and finite bases $e_1,...,e_m$ of $M$ and $f_1,...,f_n$ of $N$ there is a unique matrix $A \in \mathbb{O}^{n \times m}$ s.t. the coordinate vector of $l(\alpha_1 e_1 + ... + \alpha_m e_m)$ with respect to the basis $f_1,...,f_n$ is given by $A (\alpha_1, ... ,\alpha_m)^t$. In case $M = N$ and $f_i = e_i$ for all $i \in \mathbb{N}_n$ we call $A$ the \textit{mapping matrix of} $l:M \to M$ \textit{with respect to the basis} $e_1,...,e_n$ of $M$. A linear map $l:M \to M$ is an \textit{automorphism}, i.e. bijective, if and only if the quadratic mapping matrix $A$ w.r.t. any basis is invertible. And this is the case if and only if $A$ has full rank, i.e. if and only if $|A|$ is a unit. These facts hold for same reasons as in standard linear algebra over fields; s. e.g. \cite {vdWaerden2}, art.132.

d) The set $\textnormal{Sym}_n(\mathbb{O})$ is an $\mathbb{O}$-module of dimension $n^2$.
\end{rem}

\begin{ex}\label{ex_arithmeticModule}
Show that a set of $n$ elements of the \textit{arithmetic module} $\mathbb{O}^n$ over an integral domain $\mathbb{O}$ conform an $\mathbb{O}$-basis if and only if they are the rows or columns of an invertible $n \times n$-matrix over $\mathbb{O}$. Hint: Represent the canonical unit vectors as linear combinations of the given vectors.
\end{ex}

The following {\it basis theorem} is fundamental in theory.\footnote{Recall Definition \ref{def_ideal} of a principal ideal domain!}

\begin{thm}\label{thm_basis}
For a free module $M$, finitely generated over a principal ideal domain $\mathbb{O}$, and a submodule $N \ne \lbrace o \rbrace$ of $M$ there are a basis $e_1,...,e_n$ of $M$ and elements $\alpha_1,...,\alpha_k \in \mathbb{O} \: (k \in \mathbb{N}_n)$ s.t. $\alpha_i$ divides $\alpha_{i+1} \: (i \in \mathbb{N}_{k-1})$ and $\alpha_1 e_1,...,\alpha_k e_k$ is a basis of $N$. \end{thm}
\begin{proof}
See \cite{Cassels}, ch.11, thm.5.1 and \cite{O'Meara}, Thm.81:11!
\end{proof}

\begin{ex}\label{ex_quadrOrder}
a) Let $\sqrt{\Delta}$ denote a solution $z \in \mathbb{C}$ of $z^2 = \Delta$ for a non-square integer $\Delta \equiv 0 \textnormal{ or } 1 \mod{4}$. Then for $\omega_\Delta := (\Delta + \sqrt{\Delta})/2$ the integral domain $\mathbb{O}_\Delta := \mathbb{Z}[\omega_\Delta] = \lbrace x + y \omega_\Delta : x,y \in \mathbb{Z} \rbrace$ is a $\mathbb{Z}$-module of rank two. It is called the \textit{quadratic order} of discriminant $\Delta$. According to the Theorem every non-zero ideal $I$ of $\mathbb{O}_\Delta$ is a submodule of rank one or two, i.e. $I = \alpha \mathbb{Z}$ or $I = \alpha \mathbb{Z} + \beta \mathbb{Z}$ for some $\alpha,\beta \in \mathbb{O}_\Delta$ that are linearly independent over $\mathbb{Z}$. Let us assume the first case. If $\alpha$ were an integer we would have $I \subset \mathbb{Z}$ in contradiction to $\alpha \omega_\Delta \in I \setminus \mathbb{Z}$. So it holds $\alpha \notin \mathbb{Z}$, therefore $I \cap \mathbb{Z} = \lbrace 0 \rbrace$. But for $\alpha' \in \mathbb{O}_\Delta$ defined by $(x + y \omega_\Delta)' = x + y (\Delta - \sqrt{\Delta})/2$ holds $\alpha \alpha' \in I \cap \mathbb{Z} \setminus \lbrace 0 \rbrace$, a contradiction. So $I$ must have rank two.

b) But not every submodule of maximal rank is an ideal. For example $M := 2 \mathbb{Z} + \omega_5 \mathbb{Z} \subset \mathbb{O}_5$ is not an ideal of $\mathbb{O}_5$ since $\omega_5^2 = 5 (3 + \sqrt{5}) / 2 = 5 \omega_5 - 5 \notin M$.

c) The quadratic order $\mathbb{O}_{-20} = \left\lbrace x + i \sqrt{5} y : x,y \in \mathbb{Z} \right\rbrace$ is not a principal ideal domain. Hint: S. Proposition \ref{prop_pid} and Example \ref{ex_nonPID}!
\end{ex}

\begin{cor}\label{cor_basis}
For a module $M$ over a principal ideal domain $\mathbb{O}$ with finite $\mathbb{O}$-basis $e_1,...,e_n$ a vector $\beta_1 e_1 + ... + \beta_n e_n \in M$ is a member of some basis of $M$ if and only if $(\beta_1,...,\beta_n) = \mathbb{O}$. In particular, this property of so-called {\it primitivity} of the vector does not depend on the chosen basis of $M$.
\end{cor}
\begin{proof}
If the vectors $\beta_{i 1} e_1 + ... + \beta_{i n} e_n , i \in \mathbb{N}_n$ conform also a basis for some $(\beta_{i j}) \in \mathbb{O}^{n \times n}$ then the determinant of this matrix must be unit $\varepsilon$ of $\mathbb{O}$. On expanding this determinant (according to Laplace's formula) by the first row we obtain $\varepsilon = \beta_{1 1} \alpha_1 + ... + \beta_{1 n} \alpha_n$ for some $\alpha_j \in \mathbb{O}$. This shows $(\beta_{1 1}, ... ,\beta_{1 n}) = (\varepsilon) = \mathbb{O}$, i.e. one direction of the assertion. For the other direction we use the theorem: It exists a (basis element) $e \in M$ and an $\alpha \in \mathbb{O}$ s.t. $\alpha e = b$ for the given $b = \beta_1 e_1 + ... + \beta_n e_n$ with $\alpha_1 \beta_1 + ... + \alpha_n \beta_n = 1$ for some $\alpha_i \in \mathbb{O}$. By representing $e$ also as a linear combination of the basis elements $e_i$ it follows that $\alpha$ divides the $\beta_i$ and therefore $1$. Hence $b$ is also a basis element.
\end{proof}

\begin{rem}\label{rem_basis}
For a principal ideal domain $\mathbb{O}$ an element of $\mathbb{O}^n$ is primitive if and only if it is a row or column of an element of $\textnormal{GL}_n(\mathbb{O})$. This follows from Example \ref{ex_arithmeticModule} and Corollary \ref{cor_basis}.
\end{rem}

\section{Quadratic forms}\label{sec_quadraticforms}
In this section $M$ denotes a module over an integral domain $\mathbb{O}$ with $1+1 \ne 0$ and with a finite $\mathbb{O}$-basis $e_1,...,e_n$.

\begin{defn}\label{def_quadrForm}
A function $q : M \to \mathbb{O}$ is called ($n$-\textit{ary}) \textit{quadratic form} (\textit{on} $M$) if $q(\lambda x) = \lambda^2 q(x)$ for all $\lambda \in \mathbb{O}$, $x \in M$ and if its \textit{polar form} \begin{equation*}\label{polarform}
\varphi(x,y) := \frac{1}{2}\left(q(x+y) - q(x) - q(y)\right)
\end{equation*}
is a bilinear function $\varphi : M \times M \to \mathbb{O}$.\footnote{This is the classical definition. The non-classical definition allows $\varphi$-values $\kappa$ in the quotient field (s. Definition \ref{def_quotField}) of $\mathbb{O}$ s.t. $2\kappa \in \mathbb{O}$. See also Definition \ref{def_intBinForm}! Hence in case of $\mathbb{O}$ being a field there is no difference between those two definitions.} In case $n=2$ it is called \textit{binary}, in case $n=3$ \textit{ternary}.\footnote{In any case it holds $q(x+y)=q(x)+2\varphi(x,y)+q(y)$. This is the well-known \textit{binomial formula} in case $M = \mathbb{O}$ ($n=1$) because then $q$ means squaring and $\varphi$ multiplication.} A quadratic form is also called \textit{quadratic module/space} for emphasis on the underlying module/vectorspace $M$, respectively.
\end{defn}

\begin{ex}
a) The product of two linear forms is a quadratic form.

b) For $P \in \textnormal{Sym}_n(\mathbb{O})$ the function $q(x) := x P x^t$ of row vectors $x \in M := \mathbb{O}^n$ is a quadratic form. Its polar form is given by $\varphi(x,y) = x P y^t$.
\end{ex}

\begin{rem}\label{rem_polarform}
It holds $\varphi(x,x)=q(x)$ for all $x \in M$ and
\begin{equation}\label{polarmatrix}
q(x_1 e_1+...+x_n e_n) = (x_1,...,x_n) P (x_1,...,x_n)^t = \sum\limits_{i,j=1}^{n} p_{i j} x_i x_j
\end{equation}
for all $x_1,...,x_n \in \mathbb{O}$ with $P := (p_{i j}):=(\varphi(e_i,e_j))$. The entries $p_{i j} \in \mathbb{O}$ of $P$ are called the \textit{coefficients} of $q$ (with respect to the basis $e_1,...,e_n$). Since $\varphi$ is symmetric in its arguments $P$ is symmetric. The right side of equation \ref{polarmatrix} defines a quadratic form $\tilde{q}:\mathbb{O}^n \to \mathbb{O}$ (in $(x_1,...,x_n)$). The polar form $\tilde{\varphi}:\mathbb{O}^n \times \mathbb{O}^n \to \mathbb{O}$ of $\tilde{q}$ is given by $\tilde{\varphi}(x,y) := x P y^t$.
\end{rem}

\begin{prop}\label{prop_quadraticforms}
The map $q \mapsto P$ described in the remark defines a bijective correspondence between the set of quadratic forms $q : M \to \mathbb{O}$ and $\textnormal{Sym}_n(\mathbb{O})$.
\end{prop}
\begin{proof}
The equation $\varphi(x,x)=q(x)$ shows that the map  $q \mapsto \varphi$ defines an injective map into the set $S(M)$ of symmetric bilinear forms. It is also surjective since $x  \mapsto \varphi(x,x)$ defines a quadratic form of $M$ for every $\varphi \in S(M)$. Now, it suffices to show that $\varphi \mapsto P$ defines a bijective map $S(M) \to \textnormal{Sym}_n(M)$. Since $e_1,...,e_n$ is a basis of $M$ a bilinear function $\varphi$ is determined by the values $\varphi(e_i,e_j)$ ($i,j \in \mathbb{N}_n$). This shows injectivity. For a given $P = (p_{i j}) \in \textnormal{Sym}_n(M)$ the bilinear form $\varphi:M \times M \to \mathbb{O}$ defined (with $\tilde{\varphi}$ of Remark \ref{rem_polarform}) by
\begin{equation*}
(x,y) = (x_1 e_1+...+x_n e_n , y_1 e_1+...+y_n e_n) \mapsto \tilde{\varphi}((x_1,...,x_n),(y_1,...,y_n))
\end{equation*}
is symmetric, i.e. $\varphi \in S(M)$. And it holds $\varphi(e_i,e_j) = \tilde{\varphi}(\tilde{e}_i,\tilde{e}_j) = \tilde{e}_i P \tilde{e}_j^t = p_{i j}$ where $\tilde{e}_i$ denotes the $i$-th canonical unit vector of $\mathbb{O}^n$. This shows surjectivity.
\end{proof}

\begin{rem}\label{rem_quadraticforms}
The bijective correspondence via polar forms is established with respect to a fixed basis of $M$. With another basis given by the vectors
\begin{equation*}
e'_i := \sum\limits_{j=1}^{n} \alpha_{i j} e_j
\end{equation*}
for some $A = (\alpha_{i j}) \in \textnormal{GL}_n(\mathbb{O})$ (s. Example \ref{ex_arithmeticModule}) the original matrix $P$ corresponding with the quadratic form $q$ changes to $A P A^t$ since then
\begin{equation*}
q(y_1 e'_1+...+y_n e'_n) = (y_1,...,y_n) A P A^t (y_1,...,y_n)^t
\end{equation*}
for all $y_1,...,y_n \in \mathbb{O}$. Because of invertibility of $A$ and $A^t$ the rank of the matrix (s. Remark \ref{rem_module}b)!) corresponding with $q$ does not change under basis change.
\end{rem}

\begin{ex}\label{ex_quadrForm}
The quadratic form $q:\mathbb{Q}^2 \to \mathbb{Q}$ defined by
\begin{equation*}
q(x,y) := 6 x^2 + 5 x y + 8 y^2
\end{equation*}
corresponds (according to Proposition \ref{prop_quadraticforms}) to the symmetric matrix
\begin{equation*}
P := \left(\begin{matrix} 6 &\ 5/2 \\ 5/2 &\ 8 \end{matrix} \right)
\end{equation*}
with respect to the canonical basis $(1,0),(0,1)$. Why can $q$ not be regarded as a quadratic form on $\mathbb{Z}^2$ (in the classical Definition \ref{def_quadrForm}), although $q(x,y) \in \mathbb{Z}$ for all $x,y \in \mathbb{Z}$? What is the symmetric matrix that corresponds to $q$ with respect to the basis $(2,1),(0,1)$ of $\mathbb{Q}^2$?
\end{ex}

\begin{defn}\label{def_orthogonalSplitting}
A basis $e_1,...,e_n$ of $M$ is called \textit{orthogonal with respect to a quadratic form} $q:M \to \mathbb{O}$ if the corresponding matrix $(\varphi(e_i,e_j))$ is diagonal. A quadratic form $q:M \to \mathbb{O}$ is called \textit{regular} when $M$ is injectively mapped into its vectorspace of linear forms by $x \mapsto (y \mapsto \varphi(x,y))$. A module $N$ is called the \textit{direct sum} of the submodules $M_1,...,M_k \subseteq N$ ($k \in \mathbb{N}$) if $M_i \cap M_j = \lbrace o \rbrace$ for all $i \ne j$ and for every $n \in N$ there exist unique $m_i \in M_i$ s.t. $n = m_1 + ... + m_k$. For a quadratic form $q:N \to \mathbb{O}$ and a submodule $M \subseteq N$ the module\footnote{Obviously, it is a submodule of $N$.}
\begin{equation*}
M^{\perp} := \lbrace x \in N : \varphi(x,y)=0 \textnormal{ for all } y \in M \rbrace
\end{equation*}
is called the \textit{orthogonal complement} of $M$. In case of $M=N$ it is called also the \textit{radical}. A direct sum of $L \subseteq N$ and $M \subseteq N$ with $\varphi(x,y)=0$ for all $x \in L, y \in M$ is called an \textit{orthogonal splitting} of $N$. Then we write $N = L \perp M$. More general, we write $N = M_1 \perp ... \perp M_k$ and call it an \textit{orthogonal splitting} when $N$ is the direct sum of the $M_i$ and the $M_i$ conform pairwise orthogonal splittings.
\end{defn}

Regularity of $q:M \to \mathbb{O}$ does not mean $q(x) \ne 0$ for all $x \in M \setminus \lbrace o \rbrace$.

\begin{ex}\label{ex_orthogSplitting}
The binary quadratic form $q(x,y) := xy$ on $\mathbb{O}^2$ is regular and vanishes on the submodules $\lbrace 0 \rbrace \times \mathbb{O}$ and $\mathbb{O} \times \lbrace 0 \rbrace$. Hence, this example shows also that a regular form is not necessarily regular when restricted to submodules. Since $\mathbb{O}^2$ is the direct sum of $\lbrace 0 \rbrace \times \mathbb{O}$ and $\mathbb{O} \times \lbrace 0 \rbrace$ all the quadratic forms $q(x,y)=\alpha x^2 + \gamma y^2$ with $\alpha,\gamma \in \mathbb{O}$ make an orthogonal splitting out of it. Show that there are no other quadratic forms that do this job.
\end{ex}

\begin{defn}\label{def_representation}
We say that a quadratic form $q:M \to \mathbb{O}$ \textit{represents} an element $\alpha \in \mathbb{O}$ (\textit{primitively}) when there is some (primitive\footnote{For definition of \textit{primitivity} of a module element see Corollary \ref{cor_basis} and Remark \ref{rem_basis}.}) $a \in M \setminus \lbrace o \rbrace$ with $q(a) = \alpha$.
\end{defn}

\begin{ex}\label{ex_representation}
Over $\mathbb{Z}$ the quadratic form $q(x,y) := x^2 + 4 x y +y^2$ represents four since $q(2,0)=4$. But it does not primitively represent four. Why? Hint: Consider $q(x,y) \equiv 0 \mod{4}$.
\end{ex}

Over a field $\mathbb{K}$ with $1+1 \ne 0$ every binary quadratic form of determinant equal to the negative of a square in $\mathbb{K}$ there is basis $e_1,e_2$ of $\mathbb{K}^2$ s.t. the corresponding matrix $(\varphi(e_i,e_j))_{i,j \in \mathbb{N}_2}$ equals (s. \cite{O'Meara}, prop.42:9)
\begin{equation*}
\left(\begin{matrix} 0 &\ 1 \\ 1 &\ 0 \end{matrix}\right).
\end{equation*}
Such a quadratic form is called a \textit{hyperbolic plane}. A \textit{hyperbolic space} is the orthogonal splitting of hyperbolic planes. Hence it has even dimension. Obviously, a hyperbolic plane/space represents every field element.

\begin{prop}\label{prop_universal}
Over a field with $1+1 \ne 0$, a regular quadratic form that represents zero represents all field elements.
\end{prop}
\begin{proof}
By hopothesis there is an $a \ne o$ with $\varphi(a,a) = 0$. Because of regularity there is a vector $b$ with $\varphi(a,b) \ne 0$. By dividing all coefficients by that element we may assume without loss of generality $\varphi(a,b) = 1$. For $c := b -\varphi(b,b) a / 2$ it holds $\varphi(c,c) = 0$ and $\varphi(a,c) = 1$. Hence for an arbitrary field element $\kappa$ we have $\varphi(a + \kappa c,a + \kappa c) = 2 \kappa$. Since $1+1 \ne 0$ this proofs the assertion. 
\end{proof}

\begin{ex}\label{ex_universal}
The quadratic form $x^2 + x y - 2 y^2$ represents zero for $x := y := 1$. Hence it represents every field element. Study the proof of Proposition \ref{prop_universal} in order to find a representation of two.
\end{ex}

\begin{rem}\label{rem_directsum}
a) A quadratic form $q:M \to \mathbb{O}$ is regular if and only if its corresponding matrix is invertible with respect to any basis of $M$. And this is equivalent with $M^{\perp} = \lbrace o \rbrace$.

b) For modules $K \subseteq L$ with the direct sum of $K$ and $M$ equal to the direct sum of $L$ and $M$ it follows $K=L$. Because, for $l \in L$ there are $k \in K, m \in M$ s.t. $l = k + m$, hence $o + o = o = (k-l) + m$. Since $k-l \in L$ the uniqueness requires $k=l$. Thus we have shown $L \subseteq K$.\footnote{In fact, it underlies the more fundamental principle that for a submodule $K \subseteq L$ with $\dim(K) = \dim(L)$ we have $K = L$.}

c) For a quadratic form $q:M \to \mathbb{O}$ and an isomorphism $l:M \to N$ it holds $l(L^{\perp}) = (l(L))^{\perp}$ for every submodule $L$ of $M$. Hereby, the latter orthogonal complement is meant with respect to the quadratic form $q \circ l^{-1}:N \to \mathbb{O}$. This is clear since the polar form of  $q \circ l^{-1}$ is defined by $(x,y) \mapsto \varphi(l^{-1}(x),l^{-1}(y))$.
\end{rem}

Now, we shall see that every orthogonal splitting of $N$ is of the form $M \perp M^{\perp}$ where the restriction of $q:N \to \mathbb{O}$ to $M$ is regular.

\begin{lem}\label{lem_directsum}
For a quadratic form $q:N \to \mathbb{O}$ that is regular on the submodule $M \subseteq N$ we have $N = M \perp M^{\perp}$. If $N = M \perp L$ for another submodule $L \subseteq N$ then $L = M^{\perp}$.
\end{lem}
\begin{proof}
The proof of the first assertion in \cite{Cassels}, ch.2, lem.1.3 for $\mathbb{O}$ being a field carries over to the more general case of $M$ being a free module with finite basis over any integral domain $\mathbb{O}$ (s. \cite {vdWaerden2}, art.132 for justification of matrix representation of $q$ restricted to $M$). The second assertion follows from Remark \ref{rem_directsum}b) by observing $L \subseteq M^{\perp}$.
\end{proof}

An important property of quadratic forms (and therefore of symmetric matrices) over fields $\mathbb{K}$ with $1 + 1 \ne 0$ is the following.

\begin{cor}\label{cor_orthogonalisation}
Every quadratic form $q:V \to \mathbb{K}$ of a finite-dimensional vectorspace $V$ over $\mathbb{K}$ has an orthogonal basis. Every $v \in V$ with $q(v) \ne 0$ can be completed to an orthogonal basis with respect to $q$.
\end{cor}
\begin{proof} (according to \cite{Cassels}, ch.2, lem.1.4)
In case $q$ is the zero-form the first assertion is clear. Otherwise there is an $e_1 \in V$ s.t. $\alpha := q(e_1) \ne 0$. The one-dimensional space $U$ spanned by $e_1$ is thus regular. Therefore $V$ is the direct sum of $U$ and $U^{\perp}$ according to Lemma \ref{lem_directsum}. By induction on the dimension there is an orthogonal basis $e_2,...,e_n$ of $U^{\perp}$. Then $e_1,...,e_n$ is an orthogonal basis of $V$. The second assertion is clear since $e_1 := v$ can be chosen.
\end{proof}

\begin{rem}\label{rem_orthogonalisation}
Corollary \ref{cor_orthogonalisation} means that for a $P \in \textnormal{Sym}_n(\mathbb{K})$ there is an $A \in \textnormal{GL}_n(\mathbb{K})$ s.t. $A^t P A$ is diagonal. Namely, for the quadratic form corresponding to $P$ with respect to the canonical unit basis the columns of $A$ conform an orthogonal basis of $\mathbb{K}^n$. 
\end{rem}

\begin{ex}\label{ex_orthogonalisation}
Let $a, b, c$ elements of a field $\mathbb{K}$ with $1 + 1 \ne 0$. In order to find a matrix $A \in \textnormal{GL}_2(\mathbb{K})$ s.t. $A^t P A$ becomes diagonal for
\begin{equation*}
P := \left(\begin{matrix} a &\ b \\ b &\ c \end{matrix} \right) \in \textnormal{Sym}_2(\mathbb{K})
\end{equation*}
we differentiate between three cases. In case $a \ne 0$ we choose $x, y \in \mathbb{K}$ s.t. $a x + b y = 0$, e.g. $x := b , y := -a$. This consideration yields
\begin{equation*}
A := \left(\begin{matrix} b &\ 1 \\ -a &\ 0 \end{matrix} \right)
\end{equation*}
as suitable. In case $a = 0, c \ne 0 $ a similar argumentation reveals
\begin{equation*}
A := \left(\begin{matrix} 0 &\ -c \\ 1 &\ b \end{matrix} \right)
\end{equation*}
as suitable. In case $a = c = 0$ one may take
\begin{equation*}
A := \left(\begin{matrix} 1 &\ -1 \\ 1 &\ 1 \end{matrix} \right) .
\end{equation*}
The quadratic form of Example \ref{ex_universal} corresponds to
\begin{equation*}
P := \left(\begin{matrix} 1 &\ 1/2 \\ 1/2 &\ -2 \end{matrix} \right) \in \textnormal{Sym}_2(\mathbb{K})
\end{equation*}
with respect to the canonical unit basis. Find an $A \in \textnormal{GL}_2(\mathbb{K})$ s.t. $A^t P A$ is diagonal.
\end{ex}

\begin{rem}\label{rem_orthogonalisation_real}
In case $\mathbb{K} = \mathbb{R}$ it can be shown (e.g. in \cite{SerreD}, cor.5.3 with a somewhat more elaborate argumentation of linear algebra) that the transformation matrix $Q := A$ of Remark \ref{rem_orthogonalisation} can be chosen \textit{orthogonal}, i.e. s.t. $Q^t Q = E_n$ (s. Example \ref{ex_orthogGroup}). Hence $P \in \textnormal{Sym}_n(\mathbb{R})$ is similar to a diagonal matrix. This is not true for general fields: In fact, every diagonal matrix with \textit{complex} entries $\in \mathbb{C} := \mathbb{R}^2$ that is similar to $P \in \textnormal{Sym}_n(\mathbb{C})$ must be of the form $Q^t P Q$ for an orthogonal $Q \in \mathbb{C}^{n \times n}$ (s. \cite{HJ}, thm.4.4.13). But such a $Q$ exists if and only if $P \bar{P}$ has real entries (s. \cite{HJ}, thm.4.4.7) where the bar over $P$ means \textit{complex conjugation} $x + i y \mapsto x - i y$ ($x,y \in \mathbb{R} \cong \mathbb{R} \times \lbrace 0 \rbrace, i := (0,1) \in \mathbb{C}$) of every entry of $P$. Hence a counterexample is
\begin{equation*}
P := \left(\begin{matrix} 1 &\ i \\ i &\ 0 \end{matrix} \right) \in \textnormal{Sym}_2(\mathbb{C}) \textnormal{ because } P \bar{P} = \left(\begin{matrix} 2 &\ -i \\ i &\ 1 \end{matrix} \right) \notin \mathbb{R}^{2 \times 2}.
\end{equation*}
\end{rem}

The following is a rather general fact about quadratic forms.

\begin{prop}\label{prop_radical}
A finitely generated, free quadratic module $M$ over a principal ideal domain $\mathbb{O}$ is the orthogonal splitting of $M^{\perp}$ and another submodule $L$ of $M$. In case $0 < k := \dim M^{\perp} < n:= \dim M$ there is a basis $e_1,...,e_n$ of $M$ s.t. $e_1,...,e_k$ is a basis of $M^{\perp}$ and $e_{k+1},...,e_n$ is a basis of $L$.
\end{prop}
\begin{proof}
The case $M^{\perp} = \lbrace o \rbrace$ or $M^{\perp} = M$ is trivial. Otherwise, according to Theorem \ref{thm_basis}, there is some basis $e_1,...,e_n$ of $M$ and there are some $\alpha_1,...,\alpha_k \in \mathbb{O}$ ($k \in \mathbb{N}_{n-1}$) s.t. $\alpha_1 e_1,...,\alpha_k e_k$ is a basis of $M^{\perp}$. Since $\mathbb{O}$ does not have any zero divisors it follows $e_i \in M^{\perp}$. Therefore $e_1,...,e_k$ is a basis of $M^{\perp}$. The other elements $e_{k+1},...,e_n$ generate a submodule $L$ of $M$. Because of linear independence of any basis they form even a basis of $L$. That $M$ is the direct sum of $M^{\perp}$ and $L$ is clear from the definition of a basis. Orthogonality is shown by $\varphi(e_i,e_j)=0$ for the given symmetric bilinear form $\varphi: M \times M \to \mathbb{O}$ and $i \le k < j$.
\end{proof}

\begin{ex}\label{ex_radical}
It holds $\mathbb{Z}^2 = \mathbb{Z}(2,-1) \perp \mathbb{Z}(1,-1)$ with respect to the (non-regular) quadratic form $x^2 + 4 x y + 4 y^2$. Hereby the first submodule is the radical.
\end{ex}

Proposition \ref{prop_radical} implies that the number of variables of a quadratic form of the arithmetic module $\mathbb{O}^n$ can be reduced s.t. it becomes regular.

\begin{cor}\label{cor_rank}
For every non-zero quadratic form  $q : M \to \mathbb{O}$ of a finitely generated, free module $M$ over a principal ideal domain $\mathbb{O}$ there is a basis $e_1,...,e_n$ of $M$ and an invertible symmetric $r \times r$-matrix $R$ s.t.
\begin{equation*}
q(\alpha_1 e_1 + ... + \alpha_n e_n) = (\alpha_1,...,\alpha_r) R (\alpha_1,...,\alpha_r)^t
\end{equation*}
for all $\alpha_1,...,\alpha_n \in \mathbb{O}$. The number $r \in \mathbb{N}_n$ is the rank of the symmetric matrix corresponding with $q$ w.r.t any basis.
\end{cor}
\begin{proof}
In Remark \ref{rem_quadraticforms} we observed that the rank is independent of the choice of a basis. In case $q$ is regular any basis of $M$ will do, and $R$ is the matrix corresponding with $q$ w.r.t the chosen basis. Otherwise, due to Proposition \ref{prop_radical}, $M$ is the direct sum of $M^{\perp}$ with basis $e_1,...,e_k$ and some $L$ with basis $e_{k+1},...,e_n$, $1 \le k < n$. Then $q(\alpha_1 e_1 + ... + \alpha_n e_n) = q(\alpha_{k+1} e_{k+1} + ... + \alpha_n e_n)$ according to Definition \ref{def_quadrForm} of $q$ (by its polar form). The restriction of $q$ to $L$ with $\dim(L) = r = n-k$ is regular. By reversing the enumeration of the basis we obtain $R = (\varphi(e_i,e_j))_{i,j \in \mathbb{N}_r}$.
\end{proof}

\begin{ex}
a) With respect to the basis $(1 , -1),(2 , -1)$ of the quadratic module $q : \mathbb{Z}^2 \to \mathbb{Z}$ in Example \ref{ex_radical} the corresponding symmetric matrix is
\begin{equation*}
\left(\begin{matrix} 1 &\ -1 \\ 2 &\ -1 \end{matrix} \right) \left(\begin{matrix} 1 &\ 2 \\ 2 &\ 4 \end{matrix} \right) \left(\begin{matrix} 1 &\ 2 \\ -1 &\ -1 \end{matrix} \right) = \left(\begin{matrix} 1 &\ 0 \\ 0 &\ 0 \end{matrix} \right)
\end{equation*}
according to Remark \ref{rem_quadraticforms}; i.e. $q(x(1,-1)+y(2,-1))=x^2$.

b) In this example from \cite{Seysen}, Appendix B we determine the probability that a random symmetric $n \times n$-matrix $P$ over a finite principal ideal domain $\mathbb{O}$ has rank $0 , 1 , ... , n$, respectively. Therefore we need the (transition) probabilities that $P$ with \textit{corank} $k := n - \textnormal{rk}(P)$ obtains corank $k-1 , k , k+1$, respectively, by adding a random column $b \in \mathbb{O}^{n+1}$ (and row $b^t$). By Corollary \ref{cor_rank} we may asumme without loss of generality that this symmetric $(n+1) \times (n+1)$-matrix is in block form
\begin{equation*}
\left(\begin{matrix} O &\ O &\ b_0 \\ O &\ R &\ b_1 \\ b_0^t &\ b_1^t &\ \beta \end{matrix} \right) , R \in \textnormal{Sym}_{n-k}(\mathbb{O}) , b_0 \in \mathbb{O}^k , b_1 \in \mathbb{O}^{n-k} , \beta \in \mathbb{O}
\end{equation*}
with $\textnormal{rk}(R) = n-k$ and zero matrices $O$ of suitable dimensions. For the number $q \in \mathbb{N}$ of elements of $\mathbb{O}$ the case $b_0 = 0$ happens with probability $q^{-k}$. So the other case $b_0 \ne 0$ happens with probability $1-q^{-k}$. In this case the rank has increased by two, i.e. the corank has become $k-1$ (when $k > 0$). In case $b_0 = 0$ some linear algebra shows that the rank has not changed if and only if $\alpha = b_1 R^{-1} b_1^t$. This happens with probability $1/q$. So the corank has become $k+1$ with probability $q^{-k}/q = q^{-k-1}$ and is still $k$ with probability $q^{-k}(1-1/q) = q^{-k}-q^{-k-1}$. So we obtain a tridiagonal, stochastic (transition) matrix $Q = (q_{i j})_{i , j \in \mathbb{N}_0}$ (of infinite dimension) by defining
\begin{equation*}
q_{i j} = \left\lbrace \begin{matrix} 1 - q^{-i} &\ \textnormal{ for } j = i-1 \\ q^{-i}-q^{-i-1} &\  \textnormal{ for } j=i \hfill\null \\ q^{-i-1} &\ \textnormal{ for } j = i+1 \end{matrix}\right .
\end{equation*}
So with $Q^n = (q_{i j}^{(n)})$ the probability that $P$ has rank $k \in \lbrace 0 , 1 , ... , n \rbrace$ is $q_{0 , n-k}^{(n)}$. E.g. for $n = 3$ and $q = 3$ we obtain approximately the distribution
\begin{equation*}
0.001 , 0.036 , 0.321 , 0.642 .
\end{equation*}
\end{ex}

Corollary \ref{cor_orthogonalisation} about orthogonal bases of a vectorspace does not hold for every finitely generated, free module over a principal ideal domain.

\begin{ex}\label{ex_nonprimitve}
For the quadratic form $q : \mathbb{Z}^2 \to \mathbb{Z}$ of Example \ref{ex_representation} it holds $q(2,0) \ne 0$. But $(2,0)$ can not be a basis element of $\mathbb{Z}^2$. Why? Hint: Consider the parity of coordinates and consult Corollary \ref{cor_basis}.
\end{ex}

This motivates the following fact which will be useful for classification of symmetric matrices (s. Remark \ref{rem_equiv}d)!).

\begin{lem}\label{lem_primitiveRepr}
For a quadratic form $q:M \to \mathbb{O}$ of a finitely generated, free module $M$ over a principal ideal domain $\mathbb{O}$ and an element $\alpha \in \mathbb{O}$ there is a basis $e_1,...,e_n$ of $M$ with $q(e_1)=\alpha$ if and only if there is some primitive $a \in M$ with $q(a)=\alpha$.
\end{lem}
\begin{proof}
By identifying $a$ with $e_1$ this follows from Corollary \ref{cor_basis}.
\end{proof}

\begin{ex}\label{ex_primitiveRepr}
For the quadratic form $q$ of Example \ref{ex_nonprimitve} there is no basis element $a \in \mathbb{Z}^2$ with $q(a)=4$. This shows again that $(2,0)$ is not a basis element.
\end{ex}

\section{Quadrics with external symmetry centre}\label{sec_quadrics}
A bijective correspondence between certain symmetric matrices and certain geometric objects will be explained. Therefore, we require from our integral domain $\mathbb{O} = \mathbb{K}$ to be a commutative field with $1+1 \ne 0$. In this section $V$ denotes a finite-dimensional vector space over $\mathbb{K}$. We consider the affine space with point set $V$ and with the cosets $v + U$ ($v \in V$) of one-dimensional subspaces $U$ of $V$ as lines.\footnote{Every affine plane fulfilling the axiom of Desargues and every at least 3-dimensional affine space can be represented this way (cf. \cite{KSW}, Satz(10.1)\&(11.20)).} A \textit{translation} $t:V \to V$ is given by $t(x):=x+c$ for some vector $c \in V$ (s. \cite{KSW}, Satz 12.2). A \textit{linear affinity} is a composition of an isomorphism with a translation.\footnote{We forego more general \textit{affinities} like defined in \cite{KSW}, ch.II.} A point $c \in V$ is called a \textit{centre} of a set $X \subseteq V$ if $2 c - x \in X$ for all $x\in X$. It is called \textit{internal} in case $c \in X$ and \textit{external} otherwise. A non-empty subset of $V$ is called a \textit{quadric} (\textit{of} $V$) when it is the set of all points $x\in V$ satisfying the equation
\begin{equation}\label{quadric}
q\left(x\right)+l\left(x\right)+\gamma=0
\end{equation}
for a non-zero quadratic form $q:V \to \mathbb{K}$, a linear form $l:V \to \mathbb{K}$ and a scalar $\gamma \in \mathbb{K}$. Sometimes we use the notation $Q:$ \eqref{quadric} for a quadric $Q$ defined by equation \eqref{quadric}. A quadric of $V$ with $\dim(V)=2$ is called \textit{planar}. For an isomorphism $\Phi:V \to V'$ the set $\Phi(Q)$ is also a quadric since $q'(y) := q(\Phi^{-1}(y))$ defines a quadric $q':V' \to \mathbb{K}$. And also every translation maps a quadric onto a quadric. Both facts are seen by help of the polar form (s. Remark \ref{rem_polarform}). Hence a linear affinity maps a quadric onto a quadric.

\begin{lem}\label{lem_affine}
The set of all centres of a quadric is an affine subspace of $V$. It consists of exactly one point if and only if the defining quadratic form is regular.
\end{lem}
\begin{proof}
See \cite{KahlQuadSect}, lem. 2.3a) \& rem. 2.2c)!
\end{proof}

\begin{ex}
The affine subspace $C$ of centres of an elliptic or hyperbolic cylinder $Q$ is a line. Here for all $p\in Q$ the line through $p$ parallel to $C$ is contained in $Q$.
\end{ex}

The subspace $C$ may be empty, e.g. when $Q$ is a parabola. Anyway it holds $C \cap Q=\emptyset$ (empty intersection) or $C\subseteq Q$ (s. \cite{KahlQuadSect}, lem.2.3b)!). From now we restrict to quadrics with external centre\footnote{Quadrics with internal centre show a very special geometry: s. \cite{KahlQuadSect}, lem.2.3b)!}, i.e.
\begin{equation*}
C \neq \emptyset \; \wedge \; C \cap Q = \emptyset\;.
\end{equation*}

\begin{lem}\label{lem_normalisation}
A quadric $Q$ with external centre $c \notin Q$ is defined by equation
\begin{equation}\label{normalisation}
q(x-c)=1
\end{equation}
in $x$ where $q$ is a scalar multiple of a $Q$ defining quadratic form like in equation \eqref{quadric}.
\end{lem}
\begin{proof}
See \cite{KahlQuadSect}, lem. 2.3c)!
\end{proof}

\begin{ex}
Determine the centres of the quadric $Q:x^{2}=1, (x,y)\in \mathbb{K}^{2}$.
\end{ex}

In order to establish the announced bijective correspondence, we need some mild condition on $\mathbb{K}$ for avoiding the "pathologic" situation that a quadric with external centre is contained in a proper affine subspace of $V$; e.g. $Q:x^{2}-y^{2}=1$ with centre $(0,0) \notin Q$ consists only of the two linearly dependent vectors $(-1,0),(+1,0) \in V:=\mathbb{K}^2$ when $\mathbb{K}$ is the field of three elements.

\begin{prop}\label{prop_external}
For $|\mathbb{K}| > 5$ (i.e. more than five field elements) we have the following properties of a quadric $Q$ of $V$ with external centre $c$:

a) There are $n = \dim(V)$ points $p_{1},...,p_{n} \in Q$ s.t. $p_{1}-c,...,p_{n}-c$ are linearly independent.

b) For two different points $p_1,p_2 \in Q$ with $2c - p_1 \neq p_2$ there is a point $p_3 \in Q$ s.t. $p_1-c,p_2-c,p_3-c$ are linearly dependent but pairwise linearly independent.
\end{prop}
\begin{proof}
The assertions follow from the fact (see proof of \cite{KahlQuadSect}, prop.2.10) that for every $ \alpha \in \mathbb{K}$ there are $\lambda , \mu \ne 0$ s.t. $\lambda^2 + \alpha \mu^2 = 1$.

a) This is due to \cite{KahlQuadSect}, prop.2.10.

b) The linear independence of $p_1-c$ and $p_2-c$ follows from Lemma \ref{lem_normalisation}, since $1 = q(\lambda (p_1 - c)) = \lambda^2$ implies $\lambda = 1$ or $\lambda = -1$. For the injective affine map $\Phi(x,y) := c + x (p_1 - c) + y (p_2 - c)$ from $\mathbb{K}^2$ into $V$ there is some (unique) $\beta \in \mathbb{K}$ s.t. $P := \lbrace (x,y) \in \mathbb{K}^2 \mid x^2 + \beta x y + y^2 = 1 \rbrace$ is the preimage of $Q$ under $\Phi$: $\Phi^{-1}(Q) = P$ (see \cite{KahlQuadSect}, prop.2.6).\footnote{but not necessarily $Q = \Phi(P)$} In case $\beta = 0$ there exists $(x,y) \in P$ with $x y \ne 0$ by the fact above. Otherwise $(x,y) := (\beta,-1)$ does it. Then $p_3 := \Phi(x,y)$ is the requested point.
\end{proof}

\begin{rem}
a) Proposition \ref{prop_external}a) shows that the "pathologic" situation described above may occur only for very small numbers of field elements. For $|\mathbb{K}| > 5$ the defining vectorspace $V$ of a quadric $Q$ with external centre is uniquely determined by $Q$; i.e. there is no "quadratic polynomial" $q+l+\gamma:V' \to \mathbb{K}$ (like in \eqref{quadric}) that defines $Q$ on any other vector space $V' \supset V$.

b) In case $|\mathbb{K}|>5$, for linearly independent $a,b$ of a quadric at zero centre $o \notin Q$ there is a point $c = x a + y b \in Q$ with $xy \neq 0$ according to Proposition \ref{prop_external}b). Counter-example: Over the field $\mathbb{K}$ of five elements we have $Q =\lbrace a,b,-a,-b \rbrace$ with centre $(0,0)$ for $Q:x^2+y^2=1$, $a:=(1,0),b:=(0,1)$. Here is no point $c \in Q$ linearly independent from $a$ and from $b$.
\end{rem}

\begin{thm}\label{thm_sectorCoef}
a) For pairwise linearly independent vectors $a,b,c$ of a two-dimensional vectorspace $V$ over $\mathbb{K}$ there is only one planar quadric $Q \subset V\setminus \left\{o \right\}$ with centre at the zero-vector $o$ and with $a,b,c \in Q$. For $\Phi(x,y) := x a + y b$ and for
\begin{equation*}
C : x^2 + \left(\frac{1}{\alpha \beta}-\frac{\alpha}{\beta}-\frac{\beta}{\alpha}\right) x y + y^2 = 1 \; (x,y \in \mathbb{K})
\end{equation*}
with $c = \Phi(\alpha,\beta)$ it is $Q=\Phi(C)$.

b) In case $|\mathbb{K}|>5$, for a quadric $Q \subset \mathbb{K}^{n}$ with zero-vector $o$ as an external centre there is only one symmetric $n\times n$ matrix $M$ with $Q: x \: M \: x^t = 1$.

c) In case $|\mathbb{K}|>5$, for a basis $\beta_{1},...,\beta_{n}$ of $\mathbb{K}^{n}$ and vectors $\alpha_{i j}\in \mathbb{K}^{n}$ with $\alpha_{i j} = x_{i} \beta_{i} + y_{j} \beta_{j}$ for some $x_{i},y_{j}\in \mathbb{K}\setminus \{0\}$ ($1\leqslant i<j\leqslant n$) there is exactly one quadric $Q \subset \mathbb{K}^{n}\setminus \left\{o\right\}$ centred at zero-vector $o$ that contains all the $\left(n^{2}+n\right)/2$ vectors $\beta_{i}$ and $\alpha_{i j}$.
\end{thm}
\begin{proof}
See \cite{KahlQuadSect}, prop.2.6, thm.2.9b)\&c), cor.2.11!
\end{proof}

\begin{rem}
a) Theorem \ref{thm_sectorCoef}c) might be useful in the field of image data processing: By spherical triangulation of a spatial region with respect to some centre of perspective a surface in this region can be approximated by certain partial surfaces that need $O\left(n^{2}\right)$ instead of $O\left(n^{3}\right)$ storage space in a computer.

b) Theorem \ref{thm_sectorCoef}b) establishes a one-one-correspondence between quadrics of $\mathbb{K}^{n}$ externally centred in $o$ and its defining quadratic forms of $n$ variables (in equation \eqref{normalisation} of Lemma \ref{lem_normalisation}).
\end{rem}

Because of the uniqueness of matrix $M$ in Theorem \ref{thm_sectorCoef}b) the following notion is well defined.

\begin{defn}\label{def_determinant}
In case $|\mathbb{K}|>5$ the determinant $|M|$ is called \textit{the determinant of the quadric} $Q: x \: M \: x^t = 1 \: (x \in \mathbb{K}^n)$.
\end{defn}

\begin{rem}\label{rem_determinant}
a) According to Theorem \ref{thm_sectorCoef}b) it holds $E^{-t} P E^{-1} = M$ for the Matrix $P := (\varphi(e_i,e_j))$ in equation \eqref{polarmatrix} when $E$ has columns $e_i$. Hence $|M| \ne 0$ if and only if $o$ is the only centre of $Q$ (s. Lemma \ref{lem_affine}), and $|M|=|P|$ in case $|E| = \pm 1$. So the \textit{determinant} of $Q:q(x)=1$ with a quadratic form $q:V \to \mathbb{R}$ of some linear subspace $V$ of $\mathbb{R}^N$ can be defined as $|(\varphi(e_i,e_j))|$ where $e_1,...,e_n$ is an orthonormal basis of $V$.

b) For an automorphism $\Phi$ of $\mathbb{K}^n$ the determinant of a quadric $Q$ of $\mathbb{K}^n$ is $\det^2(\Phi)$ times the determinant of $\Phi(Q)$. This follows from a) by help of a coefficient matrix of $\Phi$.
\end{rem}

\begin{ex}\label{ex_ellipseArea}
What is the determinant of $Q:\alpha x^{2}+ \beta xy + \gamma y^{2} = 1 \: (x,y \in \mathbb{K})$? What is the area of $\lbrace \alpha x^{2}+ \beta xy + \gamma y^{2} \le 1 \rbrace$ in case of positive determinant?
\end{ex}

Not all symmetric matrices correspond to quadrics with external centre, e.g. the equations $-x^2 = 1$ and $-x^2 - y^2 = 1$ do not have any real solutions $(x,y)$.

\begin{rem}\label{rem_quadricCorresp}
Due to Remark \ref{rem_quadraticforms} and Lemma \ref{lem_primitiveRepr} a matrix $M \in \textnormal{Sym}_n(\mathbb{K})$ corresponds to a quadric with external centre if and only if there is some $A \in \textnormal{GL}_n(\mathbb{K})$ s.t. the first entry of $A M A^t$ equals one. See e.g. Remark \ref{rem_diag} for the case $n=2$ and $\mathbb{K}=\mathbb{R}$! Over $\mathbb{R}$ the remedy mentioned above can be removed: Theorem \ref{thm_diag} will show that every non-zero quadratic form represents $1$ or $-1$. So Theorem \ref{thm_sectorCoef}b) yields a one-one-correspondence between $\textnormal{Sym}_n(\mathbb{R})$ and the (geometric) sets $\lbrace x \in \mathbb{R}^n : |q(x)|=1 \rbrace$ where $q$ is some quadratic form on $\mathbb{R}^n$. For $n=2$ we still have an ellipse in case of positive determinant and a pair of parallel lines in case of zero determinant. But in case of negative determinant we have a quadruple of hyperbola branches (each lying in a sector inbetween a pair of intersecting asymptotes).
\end{rem}

\begin{ex}
The hyperbolic plane $q(x,y):=xy$ corresponds with the four \textit{unit hyperbola branches} defined by $|xy|=1$.
\end{ex}

\section{Classification}\label{sec_classification}
Classification of symmetric matrices goes back more than 200 years: J.L. Lagrange (1736-1813) defined the following notion of equivalence for binary quadratic forms over the integral domain $\mathbb{Z}$ of rational integers. In this section $\mathbb{O}$ and $\mathbb{K}$ denote an integral domain and a field, respectively, with $1+1 \ne 0$.

\begin{defn}\label{def_equiv}
Two symmetric matrices $P,Q \in \textnormal{Sym}_n(\mathbb{O})$ are called \textit{equivalent} (\textit{in the classical sense}) if there is some $A \in \textnormal{GL}_n(\mathbb{O})$ with $P = A^t Q A$. They are also called \textit{congruent}.
\end{defn}

\begin{rem}\label{rem_equiv}
a) Indeed, this defines an equivalence relation. Its classes are the orbits of right action of $\textnormal{GL}_n(\mathbb{O})$ on $\textnormal{Sym}_n(\mathbb{O})$ defined by $Q \mapsto A^t Q A$. For equivalent $P,Q$ like in the definition it holds $|P| = |A^t Q A| = |A|^2 |Q|$. Hence according to Remark \ref{rem_det_tr}a) their determinants differ by a square in $\mathbb{O}^{\times}$.

b) According to Proposition \ref{prop_quadraticforms} and Remark \ref{rem_quadraticforms} two symmetric matrices are equivalent if and only if they correspond to some common quadratic form with respect to some bases. By the same reason we can well-define \textit{equivalence of quadratic forms} by requiring of their coefficient matrices, with respect to some fixed basis, to be equivalent.

c) Two quadrics are linearly affine (in the sense declared at the beginning of section \ref{sec_quadrics}) if and only if its defining quadratic forms (of equation \eqref{quadric}) are equivalent.

d) For a primitive $a \in \mathbb{O}^n$ over a principal ideal domain $\mathbb{O}$ every $Q \in \textnormal{Sym}_n(\mathbb{O})$ is equivalent to some $(r_{i j}) \in \textnormal{Sym}_n(\mathbb{O})$ with $r_{1 1} = a Q a^t$. This is due to Lemma \ref{lem_primitiveRepr} since the latter equation shows a primitive representation of $r_{1 1}$ by the quadratic form $x \mapsto x^t Q x$. In case of $\mathbb{O}$ being a field it can be achieved more according to Corollary \ref{cor_orthogonalisation}: If $a Q a^t \ne 0$ for some $a$ then, additionly, $r_{1 j} = 0$ for all $j > 1$.

e) Equivalent quadratic forms represent the same set of elements. This is clear since $A \in \textnormal{GL}_n(\mathbb{O})$ yields an automorphism $x \mapsto A x^t$ of $\mathbb{O}^n$. Equivalent quadratic forms over a principal ideal domain $\mathbb{O}$ represent the same set of primitively represented elements. This is also clear since $A b^t$ is primitive if $b \in \mathbb{O}^n$ is so. The latter fact follows from Remark \ref{rem_basis} because for $b^t$ being a column of $B \in \textnormal{GL}_n(\mathbb{O})$ the vector $A b^t$ is a column of $A B$.

f) In case $|A Q| \ne 0$ the equation $P = A^t Q A$ imlies the equation
\begin{equation*}
\textnormal{adj}(P) = \textnormal{adj}(A) \textnormal{adj}(Q) \textnormal{adj}(A)^t ,
\end{equation*}
since $\textnormal{adj}(A) = |A| A^{-1}$ according to Prop. \ref{prop_determinant}. In case $|A Q| = 0$ the assertion holds also over an infinite integral domain $\mathbb{O}$ by Weyl's principle of irrelevance.
\end{rem}

\begin{ex}\label{ex_equiv}
The form $x^2 + 2 x y - 2 y^2$ over $\mathbb{Z}$ does not primitively represent four because of Example \ref{ex_representation} and
\begin{equation*}
\left(\begin{matrix} 3 &\ -1 \\ 1 &\ 0 \end{matrix} \right) \left(\begin{matrix} 1 &\ 1 \\ 1 &\ -2 \end{matrix} \right) \left(\begin{matrix} 3 &\ 1 \\ -1 &\ 0 \end{matrix} \right) = \left(\begin{matrix} 1 &\ 2 \\ 2 &\ 1 \end{matrix} \right) .
\end{equation*}
\end{ex}

The following equivalence relation on $\textnormal{Sym}_n(\mathbb{O})$ requires even $n$. For binary quadratic forms over $\mathbb{Z}$ it can be found e.g. in \cite{Zagier}, ch.II.8. For $\mathbb{O}:=\mathbb{R}$ it is interesting in the context of area or volume measurements (s. Example \ref{ex_sectorArea} c)!) since the determinant is an invariant under the corresponding group action.

\begin{defn}\label{def_geomEquiv}
In case $\mathbb{O}$ has unique $n$-th roots (e.g. $n$ odd, $\mathbb{O}= \mathbb{R}$) two symmetric matrices $P,Q \in \textnormal{Sym}_{2 n}(\mathbb{O})$ are called \textit{geometrically equivalent} if \begin{equation*} P = Q.A := A^t Q A / \sqrt[n]{|A|} \end{equation*} for some $A \in \textnormal{GL}_{2n}(\mathbb{O})$. This kind of equivalence is also called \textit{twisted equivalence}.
\end{defn}

\begin{ex}\label{ex_geomEquiv}
The two symmetric $2 \times 2$ matrices of the following equation are geometrically equivalent over $\mathbb{Q}$:
\begin{equation*}
\left(\begin{matrix} 1 &\ 2 \\ 2 &\ 1 \end{matrix}\right).\left(\begin{matrix} 2 &\ 1 \\ 0 &\ 1 \end{matrix}\right) = \left(\begin{matrix} 2 &\ 3 \\ 3 &\ 3 \end{matrix}\right).
\end{equation*}
But they are not classically equivalent over $\mathbb{Q}$. Otherwise, due to Remark \ref{rem_equiv}e), there would be rational numbers $p,q$ s.t. $2 p^2 + 6 p q + 3 q^2 = 1$, whence rational integers $a,b,c$ s.t. $\textnormal{gcd}(a,b)=1$ and $2 a^2 + 6 a b + 3 b^2 = c^2$. But this equation would imply that three is a divisor of $a$ and $c$ (since $3$ never divides $x^2 - 2$ for $x \in \mathbb{Z}$) and hence also of $b$.
\end{ex}

In 1933 C.G. Latimer and C.C. MacDuffee showed the link between certain classes of ideals of orders of algebraic number fields and conjugation classes of symmetric matrices, a theory worked up by O. Taussky from 1949 to 1977 (see the second appendix 'Introduction into connections between algebraic number theory and integral matrices' by Taussky in \cite{Cohn}). The following remark is in the spirit of \cite{Taussky}.

\begin{rem}\label{rem_geomEquiv}
For $n=1$ every integral domain trivially fulfills the condition in Definition \ref{def_geomEquiv} which can then be interpreted as follows: Two non-zero symmetric matrices
\begin{equation*}
\left(\begin{matrix} \alpha &\ \beta \\ \beta &\ \gamma \end{matrix}\right) , \left(\begin{matrix} \alpha' &\ \beta' \\ \beta' &\ \gamma' \end{matrix}\right)
\end{equation*}
are geometrically equivalent if and only if
\begin{equation*}
\left(\begin{matrix} \beta + \delta &\ \gamma \\ -\alpha &\ -\beta + \delta \end{matrix}\right) , \left(\begin{matrix} \beta' + \delta &\ \gamma' \\ -\alpha' &\ -\beta' + \delta \end{matrix}\right)
\end{equation*}
are similar for a/all $\delta \in \mathbb{O}$. And conversely, two non-scalar matrices
\begin{equation*}
\left(\begin{matrix} \alpha &\ \beta \\ \gamma &\ \delta \end{matrix}\right) , \left(\begin{matrix} \alpha' &\ \beta' \\ \gamma' &\ \delta' \end{matrix}\right)
\end{equation*}
are similar if and only if they have same trace and
\begin{equation*}
\left(\begin{matrix} -\gamma &\ (\alpha - \delta)/2 \\ (\alpha - \delta)/2 &\ \beta \end{matrix}\right) , \left(\begin{matrix} -\gamma' &\ (\alpha' - \delta')/2 \\ (\alpha' - \delta')/2 &\ \beta' \end{matrix}\right)
\end{equation*}
are geometrically equivalent. Here we use the non-classical definition\footnote{s. the first footnote of Definition \ref{def_quadrForm}!} of a quadratic form $\sum \alpha_{i j} x_i x_j$ of variables $x_i$ that requires only $2 \alpha_{i j} \in \mathbb{O} , i \ne j$.
\end{rem}

\begin{ex}
Show both facts of Remark \ref{rem_geomEquiv} with help of equation \eqref{eq_equivConnection}.
\end{ex}

\begin{prop}\label{prop_classCorresp}
For every  $\delta \in \mathbb{O}$ the functions
\begin{equation*}
\left(\begin{matrix} \alpha &\ \beta \\ \beta &\ \gamma \end{matrix}\right) \mapsto \left(\begin{matrix} \beta + \delta &\ \gamma \\ -\alpha &\ -\beta + \delta \end{matrix}\right) , \left(\begin{matrix} \delta - \beta &\ -\gamma \\ \alpha &\ \delta + \beta \end{matrix}\right)
\end{equation*}
are bijective between the set of geometric equivalence classes of non-zero binary quadratic forms (in the non-classical sense) and the set of conjugation classes of non-scalar $2 \times 2$ matrices of trace $2 \delta$. The two maps increase the determinant by $\delta^2$.
\end{prop}
\begin{proof}
The assertion about the first map follows from Remark \ref{rem_geomEquiv}. The other assertion follows from the first one by altering the first map's sign and by taking $-\delta$ instead of $\delta$.
\end{proof}

\begin{rem}\label{rem_classCorresp}
a) The second map of Proposition \ref{prop_classCorresp} will be used to describe the orthogonal group (s. subsection \ref{subsec_automIntBin}) of a $Q \in \textnormal{Sym}_2(\mathbb{Z})$.

b) The assertions of Proposition \ref{prop_classCorresp} remain true when equivalence and similarity, respectively, are declared with  $\textnormal{SL}_n(\mathbb{O})$ instead of  $\textnormal{GL}_n(\mathbb{O})$ (s. Proposition \ref{prop_GL} for definition of $\textnormal{SL}$!).
\end{rem}

\begin{lem}\label{lem_equiv}
a) Two symmetric matrices $Q,Q'$ over $\mathbb{O}$ are (geometrically) equivalent if and only if $\lambda Q, \lambda Q'$ are (geometrically) equivalent for $\lambda \in \mathbb{O}^{\times}$.

b) The (symmetric) $n \times n$ matrix $A_n := (\alpha_{i j})_{i,j \in \mathbb{N}_n}$ defined by \begin{equation*}
\alpha_{i j} := \left \lbrace \begin{matrix}
1 \hfill\null \; \textnormal{ if } i+j = n+1 \\ 0 \; \textnormal{ otherwise}\hfill\null
\end{matrix} \right. \end{equation*}
has determinant $(-1)^{n-1}$. It holds $A_n (\beta_{i j}) A_n = (\beta_{n+1-i \; n+1-j})$ for arbitrary $\beta_{i j} \in \mathbb{O} ,  i,j \in \mathbb{N}_n.$ I.e.: $A_n$ acts (in the classical sense of Definition \ref{def_equiv}) on a symmetric matrix like 'reflection across the counterdiagonal'.

c) For all $\alpha,\beta,\gamma \in \mathbb{O}$ and all $\delta,\varepsilon \in \mathbb{O}^{\times}$ it holds
\begin{equation*}
\left(\begin{matrix} \alpha &\ \beta \\ \beta &\ \gamma \end{matrix}\right).\left(\begin{matrix} \delta &\ 0 \\ 0 &\ \varepsilon \end{matrix}\right) = \left(\begin{matrix} \delta \alpha / \varepsilon &\ \beta \\ \beta &\ \varepsilon \gamma / \delta \end{matrix}\right).
\end{equation*}
\end{lem}
\begin{proof}
a) This follows from the definition of (geometric) equivalence.

b) The first assertion is seen by induction on $n$ when using any of the well-known determinant formulas, e.g. Laplace expansion along the first row or column (s. Prop. \ref{prop_determinant}). The second assertion follows by the fact that multiplication with $A_n$ from the left or from the right inverts the order of the rows or colums, respectively.

c) This is due to Definition \ref{def_geomEquiv} of geometric equivalence.
\end{proof}

\begin{ex}
The matrices \begin{equation*}
A_2 = \left(\begin{matrix} 0 &\ 1 \\ 1 &\ 0 \end{matrix}\right) \textnormal{ and } A_3 = \left(\begin{matrix} 0 &\ 0 &\ 1 \\ 0 &\ 1 &\ 0 \\ 1 &\ 0 &\ 0 \end{matrix}\right)
\end{equation*}
have determinant $-1$ and $+1$, respectively.
\end{ex}

Since for $n > 1$ the $n$-th root does not exist in every field, like e.g. in $\mathbb{Q}$, we restrict our investigation of geometrical equivalence to binary quadratic forms.

\begin{thm}\label{thm_geomEquiv}
For every field $\mathbb{K}$ with $1+1 \ne 0$ all non-zero symmetric $2 \times 2$ matrices of same determinant are geometrically equivalent over $\mathbb{K}$.
\end{thm}
\begin{proof}
We distinguish between three cases of the determinant. In all cases we show first that we may assume, without loss of generality, $\alpha \ne 0$ for two given matrices
\begin{equation*}
\left(\begin{matrix} \alpha &\ \beta \\ \beta &\ \gamma \end{matrix}\right) \in \textnormal{Sym}_2(\mathbb{K})
\end{equation*}
of same determinant.

First case: $\beta^2-\alpha\gamma = 0$. Since the given matrices must not be zero we may assume $\alpha \ne 0$ because of Lemma \ref{lem_equiv}b). This implies\footnote{Recall Definition \ref{def_geomEquiv} of the right operation indicated by the dot between the matrices.}
\begin{equation*}
\left(\begin{matrix} \alpha &\ 0 \\ 0 &\ 0 \end{matrix}\right).\left(\begin{matrix} \alpha &\ \beta \\ 0 &\ \alpha \end{matrix}\right) = \left(\begin{matrix} \alpha &\ \beta \\ \beta &\ \gamma \end{matrix}\right).
\end{equation*}
Now, the assertion follows by Lemma \ref{lem_equiv}c).

Second case: $\beta^2-\alpha\gamma = \delta^2$ for some $\delta \in \mathbb{K^{\times}}$. In case the given forms do not equal already
\begin{equation*}
\left(\begin{matrix} 0 &\ \delta \\ \delta &\ 0 \end{matrix}\right)
\end{equation*}
we may assume $\alpha \ne 0$ again by Lemma \ref{lem_equiv}b). But then we have
\begin{equation*}
\left(\begin{matrix} 0 &\ \delta \\ \delta &\ 0 \end{matrix}\right).\left(\begin{matrix} \alpha &\ \beta-\delta \\ \alpha &\ \beta+\delta \end{matrix}\right) = \left(\begin{matrix} \alpha &\ \beta \\ \beta &\ \gamma \end{matrix}\right).
\end{equation*}

Third case:  $\beta^2-\alpha\gamma$ is not a square in $\mathbb{K}$. Then $\alpha \ne 0$ is obvious. Because of
\begin{equation*}
\left(\begin{matrix} \alpha &\ \beta \\ \beta &\ \gamma \end{matrix}\right).\left(\begin{matrix} 1 &\ \delta \\ 0 &\ 1 \end{matrix}\right) = \left(\begin{matrix} \alpha &\ \alpha \delta + \beta \\ \alpha \delta + \beta &\ \alpha \delta^2 + 2 \beta \delta + \gamma \end{matrix}\right)
\end{equation*}
for all $\delta \in \mathbb{K}$ the given matrices can be assumed to be
\begin{equation*}
\left(\begin{matrix} \alpha &\ \beta \\ \beta &\ \alpha' \gamma' \end{matrix} \right) \textnormal{ and } \left(\begin{matrix} \alpha' &\ \beta \\ \beta &\ \alpha \gamma' \end{matrix} \right)
\end{equation*}
for some $\alpha,\alpha',\beta,\gamma' \in \mathbb{K}$ (with $\alpha \alpha' \ne 0$). Because of
\begin{equation*}
\left(\begin{matrix} \alpha &\ \beta \\ \beta &\ \alpha' \gamma' \end{matrix} \right).\left(\begin{matrix} \alpha' &\ 0 \\ 0 &\ 1 \end{matrix} \right) = \left(\begin{matrix} \alpha \alpha' &\ \beta \\ \beta &\ \gamma' \end{matrix} \right) = \left(\begin{matrix} \alpha' &\ \beta \\ \beta &\ \alpha \gamma' \end{matrix} \right).\left(\begin{matrix} \alpha &\ 0 \\ 0 &\ 1 \end{matrix} \right)
\end{equation*}
this implies the assertion.
\end{proof}

\begin{ex}\label{ex_intGeomEquiv}
Now the geometric equivalence of the symmetric matrices of Example \ref{ex_geomEquiv} can be seen without finding the transformation. But they are even geometrically equivalent over $\mathbb{Z}$ (which is not foresaid by the Theorem) since
\begin{equation*}
\left(\begin{matrix} 1 &\ 2 \\ 2 &\ 1 \end{matrix}\right).\left(\begin{matrix} 3 &\ 2 \\ -1 &\ -1 \end{matrix}\right) = \left(\begin{matrix} 2 &\ 3 \\ 3 &\ 3 \end{matrix}\right) .
\end{equation*}
\end{ex}

\begin{cor}\label{cor_similarity}
Two non-scalar $2 \times 2$ matrices are similar over $\mathbb{K}$ if and only if they have the same characteristic polynomial.
\end{cor}
\begin{proof}
This follows from Proposition \ref{prop_classCorresp} and Theorem \ref{thm_geomEquiv}.
\end{proof}

\begin{rem}
The characteristic polynomial $x^2 - t x + d, t := \textnormal{tr}(A), d := |A|$ of a $2 \times 2$ matrix $A$ is irreducible over $\mathbb{K}$ if and only if the negative $t^2/4 - d$ of the determinant of the corresponding quadratic form is not a square. For such matrices $A$ Corollary \ref{cor_similarity} follows also by reduction of $A$ to the \textit{canonical form}
\begin{equation*}
\left(\begin{matrix} 0 &\ -d \\ 1 &\ t \end{matrix}\right) ,
\end{equation*}
i.e. to the \textit{companion matrix} of $A$ (s. e.g. \cite {vdWaerden2}, art.137).
\end{rem}

\begin{ex}
The matrices
\begin{equation*}
\left(\begin{matrix} 2 &\ -1 \\ 1 &\ -2 \end{matrix}\right) , \left(\begin{matrix} 3 &\ -2 \\ 3 &\ -3 \end{matrix}\right) \textnormal{ have the same companion matrix } \left(\begin{matrix} 0 &\ 3 \\ 1 &\ 0 \end{matrix}\right).
\end{equation*}
The negative of their common determinant is $3$. This is a square in $\mathbb{R}$ but not in $\mathbb{Q}$. Therefore the matrices are similar even over $\mathbb{Q}$.
\end{ex}

\subsection{Classification over the reals}\label{subsec_real}
It is well known that the 'type' (ellipse, hyperbola,...) of a quadric in the real plane is an affine invariant (s. e.g. \cite{Audin}, ch.VI.2, cor.2.5). In case of an external centre it is characterised by the sign of its determinant. We shall clarify these assertions under the light of Remark \ref{rem_equiv}c). In 1829 A.-L. Cauchy (1789-1857), motivated by his teaching activity in Paris, found the following fact\footnote{It is known as the "principal axes theorem" or the "inertia law" named after J.J. Sylvester (1814-1897). For $n=2$ it was shown already by Lagrange.} by help of his determinant theory (s. \cite{AltenEtAl}, Kap.7.6, p.402):

\begin{thm}\label{thm_diag}
Every symmetric matrix $P$ over $\mathbb{R}$ is equivalent with \begin{equation*}\textnormal{diag}(1,...,1,-1,...,-1,0,...,0)\end{equation*} for some unique numbers $r,s \in \mathbb{N}_0$ of $1$ and $-1$, respectively. The number of positive and negative eigenvalues of $P$ is $r$ and $s$, respectively. The rank of $P$ is $r+s$.
\end{thm}
\begin{proof}
The equivalence follows from Corollary \ref{cor_orthogonalisation} by taking square roots. The uniqueness follows from the fact (s. Remark \ref{rem_orthogonalisation_real}) that a real symmetric matrix is also similar to a diagonal matrix, and the fact (s. Remark \ref{rem_similarity}) that similar matrices have same eigenvalues. For the same reason holds the assertion about the number of positive and negative eigenvalues. Equivalent symmetric matrices have the same rank. Since the shown diagonal matrix has obviously rank $r+s$ that proofs the last assertion.
\end{proof}

\begin{rem}\label{rem_diag}
An equivalence class is characterised by its \textit{signature} $(r,s)$ as defined in the Theorem. Hence, for $n=2$ there are six classes. Three of them correspond to classes of quadrics with external centre (s. section \ref{sec_quadrics}), in dependence on the sign of the determinant. The other classes are characterised by its rank $r+s \in \lbrace 0 , 1 , 2 \rbrace$. According to Remark \ref{rem_quadricCorresp} they are determined by $r=0$. We list them first:
\begin{itemize}
\item{$(0,0)$ - only the zero matrix with empty geometric set}
\item{$(0,1)$ - rank one matrices with empty geometric set}
\item{$(0,2)$ - rank two matrices with empty geometric set}
\item{$(1,0)$ - determinant zero matrices corresponding to pairs of parallel lines}
\item{$(2,0)$ - positive determinant matrices corresponding to ellipses}
\item{$(1,1)$ - negative determinant matrices corresponding to pairs of hyperbola branches}
\end{itemize}
\end{rem}

\begin{ex}\label{ex_equivReals}
a) Show that the two symmetric matrices of Example \ref{ex_geomEquiv} are (classically) equivalent over $\mathbb{R}$ by investigating the signs of their eigenvalues.

b) Two quadrics $Q,Q' \subset \mathbb{R}^2 \setminus \lbrace (0,0) \rbrace$ with symmetry centre in the origin $(0,0)$ have the same non-zero determinant (s. Definition \ref{def_determinant}) if and only if there is some isomorphism $f:\mathbb{R}^2 \to \mathbb{R}^2$ of determinant $\pm 1$ with $f(Q)=Q'$.
\end{ex}

Since the determinant of a symmetric $2n \times 2n$ matrix ($n$ odd) is invariant under the action described in Definition \ref{def_geomEquiv} geometrically equivalent matrices must have the same determinant. The rank $k := r+s$ is also an invariant. We discard the zero matrix which implies $k > 0$. Then, according to Theorem \ref{thm_diag} and Lemma \ref{lem_equiv}b), there are $\left\lfloor (k+1)/2 \right\rfloor$ geometric equivalence classes comprising of the class(es) in classical sense with signature $(r,s)$ and $(s,r)$. This reassures Theorem \ref{thm_geomEquiv} for $\mathbb{K} = \mathbb{R}$.

\subsection{Classification over the rationals}\label{subsec_rational}
The investigation of classical equivalence over the rationals, based on ideas of Gauss, Hensel (about \textit{local fields}), Hasse and Witt, is more complicated than that over the reals (cf. \cite{Cassels}, ch.6). Especially for investigation of the whole set of equivalence classes, the literature usually restricts to symmetric matrices with integral coprime entries (cf. \cite{Cassels}, ch.6, thm.1.3 \& ch.9, thm.1.2). In 1801 C.F. Gauss (1777-1855) published the seven sections of his 'Disquisitiones Arithmeticae' \cite{Gauss} giving a profound investigation of binary and ternary quadratic forms over $\mathbb{Z}$ (in section V). His \textit{genus theory} yields a simple formula (s. Remark \ref{rem_genusNumber}) for the number of equivalence classes of elements of $\textnormal{Sym}_2(\mathbb{Z})$ with given squarefree (s. Definition \ref{def_fundamental}!) determinant under the action of $\textnormal{SL}_2(\mathbb{Q})$ (s. Proposition \ref{prop_GL} for definition of $\textnormal{SL}$!).

\begin{ex}\label{ex_ratEquiv}
The rational equivalence classes of all symmetric $2 \times 2$ matrices with coprime integral entries and with determinant $-3$ are represented by
\begin{equation*}
\left(\begin{matrix} 1 &\ 2 \\ 2 &\ 1 \end{matrix}\right) , \left(\begin{matrix} 2 &\ 3 \\ 3 &\ 3 \end{matrix}\right) \:.
\end{equation*}
That they are classically inequivalent over $\mathbb{Q}$ is already shown in Example \ref{ex_geomEquiv}. That there are not more than two classes follows from the fact that even under the narrower notion of integral (proper) equivalence the class number is just two; s. Example \ref{ex_content}!
\end{ex}

But not every rational symmetric matrix is rationally equivalent to a matrix with coprime integral entries.

\begin{ex}\label{ex_rational_transf}
We apply rational transformation matrices \begin{equation*}
A := \left(\begin{matrix} \alpha &\ \beta \\ \gamma &\ \delta \end{matrix}\right) \textnormal{ to } Q := \left(\begin{matrix} 6 &\ 3 \\ 3 &\ 3 \end{matrix}\right)
\end{equation*}
and consider different cases of the maximum power $3^{\nu(\kappa)}$ of three that divides an entry $\kappa \in \mathbb{Q}$ of $A$; whereby $\nu(0) := \infty$. The exponent $\nu(\kappa)$ may be an arbitrary integral number; e.g. $\nu(2/3) = -1$. It fulfills the three properties of Remark \ref{rem_local}c) that can be readily verified. Due to the first property the three numbers $6 \alpha^2 , 6 \alpha \gamma , 3 \gamma^2$ have mutually different exponents $\nu$ if and only if $\nu(\alpha) \ne \nu(\gamma)$. Hence, for $a := 6 \alpha^2 + 6 \alpha \gamma + 3 \gamma^2$ to be integral we must have $\nu(\alpha) , \nu(\gamma) \ge 0$ according to all three properties. The same argumentation applies to $c := 6 \beta^2 + 6 \beta \delta + 3 \delta^2$. Since \begin{equation*}
A^t Q A = \left(\begin{matrix} a &\ b \\ b &\ c \end{matrix}\right)
\end{equation*} with $b := 12 \alpha \beta + 6 (\alpha \delta + \beta \gamma) + 6 \gamma \delta$ it follows $\nu(\kappa) \ge 0$ for all entries $\kappa$ of $A$ if $a,b,c$ are integral. But then $\nu(a),\nu(b),\nu(c) > 0$ as the defining equations of $a,b,c$ and the first and second property show. Hence $A^t Q A \in \mathbb{Z}^{2 \times 2}$ can not have coprime entries.
\end{ex}

The restriction to integral symmetric matrices would be less serious if their greatest common divisor was invariant under rational transformations of determinant $\pm 1$. But this illusion is already destroyed by the simple example \begin{equation}\label{eq_gcd}
\left(\begin{matrix} 1/2  &\ 0 \\ 0 &\ 2 \end{matrix}\right) \left(\begin{matrix} 8 &\ 0 \\ 0 &\ 1 \end{matrix}\right) \left(\begin{matrix} 1/2 &\ 0 \\ 0 &\ 2 \end{matrix}\right) = \left(\begin{matrix} 2  &\ 0 \\ 0 &\ 4 \end{matrix}\right) \: .
\end{equation}
Nevertheless we shall investigate integral symmetric matrices a little bit further; let us construct an infinite family of such $2 \times 2$-matrices so that they are pairwise rationally inequivalent. Therefore we use the following

\begin{defn}\label{def_fundamental}
An integer is called \textit{squarefree} when for all its prime divisors $p$ the square of $p$ is not a divisor of it. An integer congruent $0$ or $1$ modulo $4$ is called a \textit{discriminant}. A discriminant $\Delta$ is called \textit{fundamental} when it is squarefree or - in case $\Delta \equiv 0 \mod{4}$ - the integer $\delta := \Delta/4$ is squarefree and fulfills $\delta \equiv 2 \textnormal{ or } 3 \mod{4}$.
\end{defn}

\begin{prop}
The elements of the following infinite set of integral diagonal matrices are pairwise rationally inequivalent.
\begin{equation*}
\lbrace \textnormal{diag}(1,\Delta) : \Delta \textnormal{ fundamental} \rbrace
\end{equation*}
The same holds for the set of $\textnormal{diag}(1,-\Delta)$.
\end{prop}
\begin{proof}
For fundamental discriminants $\Gamma$ and $\Delta$ we presuppose the rational equivalence of $\textnormal{diag}(1,\Gamma)$ and $\textnormal{diag}(1,\Delta)$. Then by Remark \ref{rem_equiv}a) it follows that the determinants $\Gamma$ and $\Delta$ of the two given matrices differ by a rational square. So fundamentality implies $\Gamma = \Delta$ and therefore equality of the two matrices. The assertion for $-\Delta$ instead of $\Delta$ follows analogously.
\end{proof}

The following Remarks are concerned with non-square discriminants.

\begin{rem}\label{rem_normform}
a) For arbitrary discriminants $\Delta$ the quadratic form $x^2 - \Delta y^2$ is rationally equivalent with the quadratic form
\begin{equation*}
n(x,y) := x^2 + \Delta x y + \frac{\Delta^2 - \Delta}{4} y^2 = \left(x + \frac{\Delta}{2} y\right)^2 - \Delta \left(\frac{y}{2}\right)^2 .
\end{equation*}
The equation shows also that there is a 1-1-correspondence between integral (or rational) solutions $t := 2 x + \Delta y , u := y$ of $|t^2 - \Delta u^2| = 4$ and integral (or rational, respectively) solutions $x,y$ of $|n(x,y)| = 1$.\footnote{For verifying the integral case observe $t \equiv \Delta u \mod{2}$ for $t,u \in \mathbb{Z}$ with $t^2 \equiv \Delta u^2 \mod{4}$. It is clear that the 1-1-correspondence holds also without the absolute value function.} It will turn out that these equations over $\mathbb{Z}$ characterises the units of the quadratic order $\mathbb{O} := \mathbb{O}_\Delta$ of Example \ref{ex_quadrOrder} with non-square $\Delta$. With $\omega := (\Delta + \sqrt{\Delta})/2$ and $\omega' := (\Delta - \sqrt{\Delta})/2 = \Delta - \omega$ it holds
\begin{equation}\label{eq_normform}
n(x,y) = (x + y \omega)(x + y \omega') \textnormal{ for } x,y \in \mathbb{Q}.
\end{equation}
Since $1,\omega$ is a basis of the $\mathbb{Z}$-module $\mathbb{O}$ the function $x + y \omega \mapsto x + y \omega'$ is well-defined on $\mathbb{O}$. A straightforward computation shows that it is a ring endomorphism of $\mathbb{O}$. Hence the \textit{norm function} $N(x + y \omega) := n(x,y) \in \mathbb{Z}$ is multiplicative on $\mathbb{O}$. It follows $|N(u)| = 1$ for $u \in \mathbb{O}^{\times}$ because $u v = 1$ implies $N(u) N(v) = N(1) = 1$ with factors in $\mathbb{Z}$. So Remark \ref{rem_ring}f) shows
\begin{equation*}
\mathbb{O}^{\times} = \lbrace x + y \omega : x,y \in \mathbb{Z}, |n(x,y)| = |N(x + y \omega)| = 1 \rbrace .
\end{equation*}
The quotient field (s. Definition \ref{def_quotField}!) of $\mathbb{O}$ is isomorphic to the \textit{quadratic number field} $\mathbb{K} := \mathbb{Q}(\sqrt{\Delta}) = \mathbb{Q}(\omega) = \lbrace x + y \omega : x,y \in \mathbb{Q} \rbrace$. To see this first observe that $\mathbb{K}$ is a vector field over $\mathbb{Q}$ with basis $1,\omega$. So the norm function $N$ is also well-defined on $\mathbb{K}$. Equation \ref{eq_normform} shows us $N(\kappa) \ne 0$ for all $\kappa \in \mathbb{K} \setminus \lbrace 0 \rbrace = \mathbb{K}^{\times}$ and
\begin{equation*}
\frac{x+y\omega}{z+w\omega} = \frac{(x+y\omega)(z+w\omega')}{N(z+w\omega)} \in \mathbb{K}
\end{equation*}
for all $x,y,z,w \in \mathbb{Z}$ with $(z,w) \ne (0,0)$. By computing the coordinates of the latter vector w.r.t. basis $1,\omega$ we find the required isomorphy. The same argumentation like above shows that $N$ is multiplicative on $\mathbb{K}$. So we obtain a group homorphism $N:\mathbb{K}^{\times}$ $\to$ $\mathbb{Q}^{\times}$.

b) By Remark \ref{rem_diag} the quadric
\begin{equation*}
Q := \lbrace (t,u) \in \mathbb{R} \times \mathbb{R} : t^2 - \Delta u^2 = 4 \rbrace
\end{equation*}
is an ellipse in case $\Delta < 0$ and a hyperbola in case $\Delta > 0$. In case of non-square $\Delta$ the subset of points with integral or rational coordinates of $Q$ becomes a group via multiplication in $\mathbb{O}$ or $\mathbb{K}$, respectively. This is because $N(t + u \sqrt{\Delta}) = t^2 - \Delta u^2$ for the norm function $N$ in Remark a) and because the kernels of $N:\mathbb{O}^{\times}$ $\to$ $\mathbb{Z}^{\times}$ and of $N:\mathbb{K}^{\times}$ $\to$ $\mathbb{Q}^{\times}$ are (commutative) groups according to Proposition \ref{prop_homomorphism}. Hereby the group operation on two rational points $(t_1,u_1),(t_2,u_2)$ of $Q$ is defined as $((t_1 t_2 + \Delta u_1 u_2)/2 , (u_1 t_2 + u_2 t_1)/2)$. The rational points on $Q$ can be constructed by Bachet's secant method: For every $m \in \mathbb{Z}$ and every $n \in \mathbb{N}$ with $n^2 \ne \Delta m^2$ we obtain the rational point $(2+\lambda n,\lambda m) \in Q$ with $\lambda := 4 n / (\Delta m^2-n^2)$ as a straightforward calculation shows. Since $\Delta$ is not a square the equation $n^2 = \Delta m^2$ is impossible.\footnote{This equation would mean that the direction vector $(n,m)$ would be parallel to one of the asymtotes $t = \pm \sqrt{\Delta} u$ of the hyperbola.} Since the slope of a line through $(2,0)$ and any other rational point must be rational we obtain - by Bachet's construction - all rational points of $Q$ as intersection points with all the lines through $(2,0)$ with rational slope $m/n$. Since these slopes are infinitely many it follows that the group of rational points on $Q$ is infinite. And since $\mathbb{Q}$ is countable the group is also. For a point $(2+s,r) \in Q$ with $r,s \in \mathbb{R}, s \ne 0$ there is a sequence $(m_k/n_k)_k$ of rational "slopes" like $m/n$ above converging towards $r/s$ since $\mathbb{Q}$ is dense in $\mathbb{R}$. For the corresponding sequence $(\lambda_k)_k$ the sequence of rational points $(\lambda_k n_k,\lambda_k m_k) \in Q$ converges towards $(s,r) = \lambda (1,r/s) = \lambda (1,\tan(\alpha))$ with
\begin{equation*}
\alpha := \arctan(r/s) \textnormal{ and } \lambda := 4 \tan^2(\alpha)/(\Delta - \tan^2(\alpha))
\end{equation*}
because '$\tan$' and '$\arctan$' are continuous functions. Thus we have shown that the countably infinite group of rational points on $Q$ lies dense in $Q$. By expanding the group operation in Remark a) to points with real coordinates $t_1,u_1,t_2,u_2$ the whole quadric $Q$ becomes a group. So due to Theorem \ref{thm_diag}, Lemma \ref{lem_normalisation} and Remark a) every ellipse and every hyperbola can be equipped with a group operation and contains some countably infinite and dense subgroup (of points not necessarily with rational coordinates) .
\end{rem}

The author neither knows how the investigation of the whole set of equivalence classes of rational symmetric matrices can be restricted to integral symmetric matrices without loosing generality (in due consideration of Examples \ref{ex_rational_transf} and \ref{eq_gcd}); nor does he know any applications of rational transformations to other areas than number theory.  Therefore we shall be content with settling the classical problem of deciding whether two arbitrary rational symmetric matrices are rationally equivalent. This will be achieved by Theorem \ref{thm_ratequiv}. Therefore we need the following theorem of Hasse \cite{Hasse} about local fields (s. Remark \ref{rem_local}) that uses also ideas of Minkowski in \cite{Minkowski}.\footnote{A further tool in the theory of classification over the rationals is Corollary \ref{cor_witt}.}

\begin{thm}\label{thm_minkowski_hasse}
A quadratic form with rational coefficients represents zero if and only if it does so over $\mathbb{R}$ and over the local field $\mathbb{Q}_p$ for every prime $p$.
\end{thm}
\begin{proof}
See \cite{Cassels}, ch.6, thm.1.1!
\end{proof}

\begin{cor}\label{cor_minkowski_hasse}
A quadratic form with rational coefficients represents a given rational number if and only if it does so over $\mathbb{R}$ and every local field.
\end{cor}
\begin{proof}
This follows from the Theorem and Proposition \ref{prop_universal}.
\end{proof}

\subsection{Classification over the integers / group structure}\label{subsec_integral}
In contrast with $\mathbb{Q}$ the ring $\mathbb{Z}$ of integers as ground domain can be applied in the field of information security. The set of geometric\footnote{Classical equivalence is not that useful because of lacking group structure.} equivalence classes of binary\footnote{The literature tells us analogous results for more than two variables. But the theory is less complete and more complicated than in the binary case. In view of the cryptographic application in section \ref{sec_crypto} we may restrict to the latter case.} quadratic forms and of given determinant is finite, hence accessible by computing machines. Thanks to Gauss' \textit{composition} (\cite{Gauss}, art.234-251) a certain subset of it can be endowed with a group structure, so that it becomes useful to cryptographic algorithms under certain security requirements. Composition has been varied after Gauss: e.g. by Dirichlet \cite{DD}, art.56/146 and, more general, by Kneser \cite{Kneser}. The book \cite{BV} accounts for the algorithmic aspects of Dirichlet's variant (which corresponds to multiplication of $\mathbb{Z}$-modules; s. \cite{BV}, ch. 7.3.4).

\begin{defn}\label{def_intBinForm}
A non-zero (\textit{integral binary}) \textit{form} $[\alpha,\beta,\gamma] := \alpha x^2 + \beta x y + \gamma y^2$ with \textit{coefficients} $\alpha,\beta,\gamma \in \mathbb{Z}$ is called \textit{primitive} when $\gcd(\alpha,\beta,\gamma)=1$.\footnote{Note that integral coefficients do not guarantuee integrality of the corresponding symmetric matrix since $\beta /2$ is one of its entries.} The number $\beta^2 - 4 \alpha \gamma$ is called its \textit{discriminant}. In case it is negative the form is called \textit{definite}. In case it is positive the form is called \textit{indefinite}. Two forms $q,q'$ are called \textit{properly equivalent} when $q' = q.A$ for some $A \in \textnormal{SL}_2(\mathbb{Z})$.
\end{defn}

\begin{rem}\label{rem_content}
The \textit{content} $\gcd(\alpha,\beta,\gamma)$ of a non-zero form $[\alpha,\beta,\gamma]$ does not change under the action of $\textnormal{GL}_2(\mathbb{Z})$. Hence (properly) equivalent forms have the same content, and all forms of the (proper) class of a primitive form are primitive. The theory represented below deals with primitive forms only although it may be formulated analogously for non-primitve forms too.
\end{rem}

\begin{ex}\label{ex_content}
a) All integral binary forms of discriminant $12$ have content one. Otherwise there would be a (primitive) form of discriminant $3$. But a discriminant of an integral form is either congruent to zero or congruent to one modulo four as the definition shows. With the reduction theory of \cite{Zagier}, ch.13, thm.1 it can be shown that every indefinite form of non-square discriminant is properly equivalent to a form $[\alpha,\beta,\gamma]$ with $\alpha,\gamma > 0 , \beta > \alpha + \gamma$. The only such forms of discriminant $12$ are $[1,4,1] , [2,6,3]$. Hence there are at most two proper equivalence classes of discriminant $12$.\footnote{Example \ref{ex_ratEquiv} now shows that there are exactly two.}

b) For discriminant $20$ there are less classes of primitive forms than classes of all forms.
\end{ex}

For explaining the group structure we follow \cite{KahlDiplThesis}, art.2 which simplifies Dirichlet's exposition of composition in \cite{DD}.\footnote{and any other descriptions of composition I know; An interesting account on this is also \cite{Bhargava}.} Remind the notation $a \equiv b \mod{m}$ for integers $a,b,m$ when $m$ divides $a-b$ (s. Remark \ref{rem_ideal}c)!).

\begin{lem}\label{lem_composition}
A primitive form $\alpha x^2 + \beta x y + \gamma y^2$ primitively represents a non-zero integer coprime with $n \in \mathbb{N}$. In case $\alpha \ne 0$ every primitive form of same discriminant is equivalent with $[\alpha',\beta + 2 \alpha n,m \alpha]$ for some $m , n \in \mathbb{Z}$ and some $\alpha' \in \mathbb{Z} \setminus \lbrace 0 \rbrace$ coprime with $\alpha$. And then $[\alpha,\beta,\gamma]$ is equivalent with $[\alpha,\beta + 2 \alpha n,m \alpha']$. In case $\gcd(\alpha,\alpha')=1$ for a form $[\alpha',\beta',\gamma']$ of same discriminant one may choose $n$ s.t. $2 \alpha n \equiv \beta'-\beta \mod{\alpha'}$. Every primitive form is properly equivalent to a form $[\alpha,\beta,\gamma]$ with $\alpha \gamma \ne 0$ and $\gcd(\alpha,\gamma)=1$.
\end{lem}
\begin{proof}
The first assertion is due to \cite{Gauss}, art.228. Hence, for another primitive form $[\alpha',\beta',\gamma']$ we may assume $\alpha' \ne 0$ and $\gcd(\alpha,\alpha') = 1$ according to Lemma \ref{lem_primitiveRepr}. The definition of discriminant $\Delta := \beta^2 - 4 \alpha \gamma$ shows that $\beta$ and $\beta'$ have same parity. So there is some $n \in \mathbb{Z}$ s.t. $(\beta' - \beta)/2 \equiv \alpha n \mod{\alpha'}$, i.e. $\beta' \equiv \beta + 2 \alpha n \mod{\alpha'}$. It follows also $(\beta')^2 \equiv \Delta \mod{4 \alpha \alpha'}$. Since equivalent forms have same discriminant and
\begin{equation*}
[\alpha,\beta,\gamma].\left(\begin{matrix} 1 &\ n \\ 0 &\ 1 \end{matrix}\right) = [\alpha,\beta + 2 \alpha n, \alpha n^2 + \beta n + \gamma]
\end{equation*}
this shows all other assertions but the last one. We see also that a form $[\alpha,\beta,\gamma]$ is properly equivalent with $[\alpha,\beta+2\alpha,\alpha+\beta+\gamma]$ and analogously (by reasons of symmetry) with $[\alpha+\beta+\gamma,\beta+2\gamma,\alpha]$. Hence for showing the last assertion we may assume $\alpha \gamma \ne 0$ already. But then it follows also from the latter equation by choosing $n \in \mathbb{N}$ as the product\footnote{$n=1$ in case there is no such prime} of all primes that divide $\alpha \gamma$ but not $\gcd(\alpha,\gamma)$.
\end{proof}

\begin{defn}\label{def_composition}
The form $[\alpha \alpha',\beta,\gamma]$ is called the \textit{composition} of two primitive forms $[\alpha,\beta,\alpha' \gamma] , [\alpha',\beta,\alpha \gamma]$ with $\alpha \alpha' \ne 0$ and $\gcd(\alpha,\alpha') = 1$.
\end{defn}

\begin{ex}\label{ex_composition}
a) Show that the composition of primitive forms is primitive.

b) For composing a form in the proper class of $[1,4,1]$ with a form in the proper class of $[2,6,3]$ we solve the congruence $(4 - 6)/2 \equiv -n \mod{2}$ in $n \in \mathbb{N}$. A solution is $n := 1$. Hence the form
\begin{equation*}
[1,4,1].\left(\begin{matrix} 1 &\ 1 \\ 0 &\ 1 \end{matrix}\right) = [1,\beta,2 \gamma] = [1,6,6]
\end{equation*}
with $\beta := 4 + 2 \cdot 1 \cdot 1$ and $\gamma := 3$ (determined by the discriminant $12$) can be composed with $[2,6,3]$ itself; the composition is $[1 \cdot 2,\beta,\gamma]=[2,6,3]$.
\end{ex}

\begin{thm}\label{thm_composition}
Composition induces commutative group structure on the set of all proper classes of primitive forms of given discriminant $\Delta$. The neutral element is the proper class of $[1,\beta,\alpha \gamma]$ for any form $[\alpha , \beta, \gamma]$ of discriminant $\Delta$. The inverse of the proper class of $[\alpha,\beta,\gamma]$ is the proper class of $[\gamma,\beta,\alpha]$, i.e the proper class of $[\alpha,-\beta,\gamma]$.
\end{thm}
\begin{proof}
For two proper classes $F,G$ of primitive forms of discriminant $\Delta$ there are forms $[\alpha,\beta,\alpha' \gamma] \in F$ and $[\alpha',\beta,\alpha \gamma] \in G$ like in Definition \ref{def_composition} due to Lemma \ref{lem_composition}. For another such pair $[\alpha_1,\beta_1,\alpha_1' \gamma_1] \in F, [\alpha_1',\beta_1,\alpha_1 \gamma_1] \in G$ we have to show that  $[\alpha \alpha',\beta,\gamma]$ is equivalent with  $[\alpha_1 \alpha_1',\beta_1,\gamma_1]$. Two forms $[\alpha,\beta,\gamma]$ and $[\alpha_1,\beta_1,\gamma_1]$ with $\alpha_1 = \alpha r^2 + \beta r t + \gamma t^2$ for some $r,t \in \mathbb{Z}$ with $\gcd(r,t)=1$ are properly equivalent if and only if they have same discriminant and there are $s,u \in \mathbb{Z}$ s.t.
\begin{equation*}
\left(\begin{matrix} -t &\ r \\ 2 \alpha r + \beta t &\ \beta r + 2 \gamma t \end{matrix} \right) \left(\begin{matrix} s \\ u  \end{matrix}\right) = \left(\begin{matrix} 1 \\ \beta_1 \end{matrix}\right) .
\end{equation*}
The latter condition is equivalent to
\begin{equation*}
\left(\begin{matrix} (\beta-\beta_1) r / 2 + \gamma t \\ \alpha r + (\beta+\beta_1) t / 2 \end{matrix} \right) \equiv \left(\begin{matrix} 0 \\ 0 \end{matrix}\right) \mod{\alpha_1} .
\end{equation*}
So with notation as above we have this congruence with $\alpha' \gamma$ instead of $\gamma$ and $\alpha_1 = \alpha r^2 + \beta r t + \alpha' \gamma t^2$ for some $r,t \in \mathbb{Z}$ with $\gcd(r,t)=1$. By the same reason we have also
\begin{equation*}
\left(\begin{matrix} (\beta-\beta_1) v / 2 + \alpha \gamma x \\ \alpha' v + (\beta+\beta_1) x / 2 \end{matrix} \right) \equiv \left(\begin{matrix} 0 \\ 0 \end{matrix}\right) \mod{\alpha'_1}
\end{equation*}
for some $v,x \in \mathbb{Z}$ with $\gcd(v,x)=1$ and $\alpha'_1 = \alpha' v^2 + \beta v x + \alpha \gamma x^2$. Straightforward calculations show
\begin{equation*}
\alpha_1 \alpha'_1 = \alpha \alpha' X^2 + \beta X Y + \gamma Y^2 ,\footnote{a special case, known already to Lagrange, of an identity in \cite{Gauss}, art.235}
\end{equation*}
\begin{equation*}
\left(\alpha r + \frac{\beta+\beta_1}{2} t\right) \left(\alpha' v + \frac{\beta+\beta_1}{2} x \right) \equiv \alpha \alpha' X + \left(\frac{\beta+\beta_1}{2}\right) Y \mod{\alpha_1 \alpha'_1} ,
\end{equation*}
\begin{equation*}
\left(\alpha r + \frac{\beta+\beta_1}{2} t\right) \left(\alpha \gamma x + \frac{\beta-\beta_1}{2} v \right) \equiv \alpha \left(\gamma Y + \frac{\beta-\beta_1}{2} X \right) \mod{\alpha_1 \alpha'_1}
\end{equation*}
for $X := r v - \gamma t x$ and $Y := \alpha r x + \alpha' t v + \beta t x$. Because of the above congruences modulo $\alpha_1$ and $\alpha'_1$ the latter two congruences read
\begin{equation*}
\alpha \alpha' X + \left(\frac{\beta+\beta_1}{2}\right) Y \equiv 0 \mod{\alpha_1 \alpha'_1} ,
\end{equation*}
\begin{equation*}
\alpha \left(\gamma Y + \frac{\beta-\beta_1}{2} X \right) \equiv 0 \mod{\alpha_1 \alpha'_1} .
\end{equation*}
Analogously the latter congruence holds also for $\alpha'$ instead of $\alpha$. Because of $\gcd(\alpha,\alpha')=1$ it then holds even with factor one. Hence the latter two congruences imply the claimed equivalence since the equation $W X + Z Y = 1$ for some $W,Z \in \mathbb{Z}$ imply also $\gcd(X,Y)=1$. Therefore the \textit{composition} $F G$ as the class of $[\alpha \alpha',\beta,\gamma]$ is well defined. Obviously, it holds $F G = G F$, i.e. commutativity. For showing associativity $(F_1 F_2) F_3 = F_1 (F_2 F_3)$, we may act on three forms $q_1=[\alpha_1,\beta_1,\gamma_1] \in F_1 , q_2=[\alpha_2,\beta_2,\gamma_2] \in F_2 , q_3=[\alpha_3,\beta_3,\gamma_3] \in F_3$ with pairwise coprime $\alpha_1, \alpha_2, \alpha_3 \ne 0$ according to Lemma \ref{lem_composition}. By the chinese remainder theorem there is a $\beta \equiv \beta_j \mod{2 \alpha_j}$ for $j = 1 , 2 , 3$. Since the third coefficient $*$ of a form with leading coefficient $\ne 0$ is determined by the discriminant $\Delta$ and the other two coefficients the class of $[\alpha_1 \alpha_2 \alpha_3,\beta,*]$ conincides with both sides of the equation to be shown. The assertion about the neutral element is clear. According to Lemma \ref{lem_composition} every class has an element $[\alpha,\beta,\gamma]$ with $\alpha \gamma \ne 0$ and $\gcd(\alpha,\gamma)=1$. Therefore its inverse is the class of $[\gamma,\beta,\alpha]$ since the composition of these two forms is
\begin{equation*}
[\alpha \gamma , \beta, 1] = [1 , -\beta, \alpha \gamma].\left(\begin{matrix} 0 &\ 1 \\ -1 &\ 0 \end{matrix}\right) .
\end{equation*}
That the latter assertion is also correct without the conditions on $\alpha,\gamma$ follows from Lemma \ref{lem_composition} and from the equivalence of the two equations
\begin{equation*}
[\alpha',\beta',\gamma'] = [\alpha,\beta,\gamma].\left(\begin{matrix} r &\ s \\ t &\ u \end{matrix}\right) , [\gamma',\beta',\alpha'] = [\gamma,\beta,\alpha].\left(\begin{matrix} u &\ t \\ s &\ r \end{matrix}\right)
\end{equation*}
for arbitrary numbers $r,s,t,u$ s.t. $r u - s t \ne 0$. The last assertion of the theorem follows by
\begin{equation*}
[\gamma, \beta, \alpha].\left(\begin{matrix} 0 &\ 1 \\ -1 &\ 0 \end{matrix}\right) = [\alpha, -\beta, \gamma] .
\end{equation*}
\end{proof}

\begin{rem}
For $\alpha \ne 0$ the composition of $[\alpha,\beta,\gamma]$ and $[-1,\beta,\alpha \gamma]$ is
\begin{equation*}
[-\alpha,\beta,-\gamma] = [\alpha,\beta,\gamma].\left(\begin{matrix} 1 &\ 0 \\ 0 &\ -1 \end{matrix}\right) .
\end{equation*}
Hence a geometric class is the union of a proper class $F$ with the composition $F J$ of $F$ and the proper class $J$ of $[-1,*,*]$.\footnote{For positive discriminants $\Delta$ the proper class $J$ is the neutral element, i.e. equal to the proper class of $[1,*,*]$, if and only if the equation $x^2 - \Delta y^2 = -4$ has a solution $(x,y) \in \mathbb{Z}^2$; s. Example \ref{ex_units} and Remark \ref{rem_SO}!} For two proper classes $F,G$ of same discriminant the union of $F G$ with $F G J$ does not change by taking $F J$ for $F$. Therefore the \textit{composition} $F G \cup F G J$ of two geometric classes $F \cup F J$ and $G \cup G J$ is well defined, and the group structure carries over to  the set $Cl(\Delta)$ of geometric classes of primitive forms of discriminant $\Delta$.
\end{rem}

Cryptographic algorithms (s. section \ref{sec_crypto}) with binary forms are implemented mainly for definite forms. This is because the class group of a negative discriminant is usually much larger than those of positive discriminants of about the same absolute value (cf. \cite{BV}, ch.12). Therefore we restrict to definite forms.

\begin{rem}\label{rem_definiteCase}
In the definite case every geometric equivalence class equals the union of two proper equivalence classes which 'differ only by sign': one proper class with \textit{positive definite}\footnote{See also Definition \ref{def_posdef}!} forms, i.e. representing only positive numbers, and the other with \textit{negative definite} forms, i.e. representing only negative numbers.
\end{rem}

\begin{defn}\label{def_normalisation}
A form $[\alpha,\beta,\gamma]$ of negative discriminant is called \textit{reduced} when $-\alpha < \beta \le \alpha < \gamma$ or $0 \le \beta \le \alpha = \gamma$. For a positive definite form $q = [\alpha,\beta,\gamma]$ we call
\begin{equation*}
q.\left(\begin{matrix} 1 &\ n \\ 0 &\ 1 \end{matrix} \right) = [\alpha , \beta + 2 \alpha n , \alpha n^2 + \beta n + \gamma]
\end{equation*}
with $n \in \mathbb{Z}$ defined by $-\alpha < \beta + 2 \alpha n \le \alpha$ the \textit{normalisation} of $q$. (The number $n$ is the greatest integer smaller or equal to $(\alpha-\beta)/(2 \alpha)$.)
\end{defn}

\begin{rem}\label{rem_reduction}
According to \cite{Zagier}, ch.13 every positive definite form is properly equivalent with exactly one reduced form. That means that we have a one-one-correspondence between proper classes of positive definite forms and reduced forms. In particular, the reduced forms are mutually inequivalent. The reduction algorithm is as follows: Substitute $[\alpha , \beta , \gamma]$ by the normalisation of $[\gamma , -\beta , \alpha]$ until it is reduced. After first normalisation the loop may be terminated if $\alpha \le \gamma$. In case $\alpha < \gamma$  the form is already reduced. Otherwise $\beta$ must be substituted by $|\beta|$.
\end{rem}

\begin{ex}\label{ex_classComposition}
The form $q:=[2,1,21]$ is a primitive, positive definite form of discriminant $\Delta := -167$. It is already reduced. For composing the proper class of $q$ with itself we take $x := 2 / \gcd(2,21) = 2 , y:= 2 / \gcd(2, 2 x) = 1$ as in Lemma \ref{lem_composition}. Then we find by help of the extended euclidean algorithm $w := 1, z := 1$ s.t. $w x - y z = 1$. Then
\begin{equation*}
q' := q.\left(\begin{matrix} x &\ z \\ y &\ w \end{matrix} \right) = [31,53,24]
\end{equation*}
can still not be composed with $q$ since its third coefficient is not a divisor of the first coefficient of $q$. So we compute (in general by the extended euclidean algorithm) $n := 0$ s.t. $31 n \equiv (1 - 53))/2 \mod{2}$ (s. Lemma \ref{lem_composition} again!). Then we may take $[2,53+2 \cdot 31 n, *] = [2,53,31 \cdot 12]$ instead of $q$ for composition with $q'$. That gives the form $[62,53,12]$ which is not reduced. The normalisation of $[12,-53,62]$ is $[12,-53 + 2 \cdot 2 \cdot 12,12 \cdot 2^2 - 53 \cdot 2 + 62] = [12,-5,4]$ which is still not reduced. But the normalisation of $[4,5,12]$ is the reduced form $[4,-3,11]$. Since its first coefficient is different from one we conclude $h(-167) > 2$. Indeed by iterative computation of $F ^ k = F F ... F , k \in \mathbb{N}_{11}$ for the proper class $F$ of $q$ we obtain the following corresponding sequence of reduced forms:
\begin{equation*}
\begin{matrix}
[2,1,21],[4,-3,11],[6,-5,8],[3,1,14],[6,1,7], \\ [6,-1,7],[3,-1,14],[6,5,8],[4,3,11],[2,-1,21], \\ [1,1,42] .\end{matrix}
\end{equation*}
Therefore $F^{11}$ is the neutral element. Since there are no other reduced forms we have $h(-167)=11$. Hence $Cl(-167)$ is a cyclic group of order $11$ generated by any of its non-neutral elements.
\end{ex}

\begin{rem}
For a fixed number $\Delta \in \mathbb{Z}$ there are only finitely many triples $(\alpha,\beta,\gamma) \in \mathbb{Z}^3$ with $|\beta| \le \alpha \le \gamma$ and $\beta^2 - 4 \alpha \gamma = \Delta$ because
\begin{equation*}
-\Delta = 4 \alpha \gamma - \beta^2 \ge 4 \alpha^2 - \beta^2 \ge 3 \alpha^2 .
\end{equation*}
So there are only finitely many reduced forms of fixed negative discriminant. I.e. for $\Delta < 0$ the \textit{class} number $h(\Delta)$ of elements of $Cl(\Delta)$ is finite.\footnote{This also true for positive discriminants; s. e.g. \cite{Zagier}, ch.8, thm.1.} 
\end{rem}

For fundamental discriminants $\Delta$ (s. \ref{def_fundamental}!) the class number can be described by the \textit{Jacobi symbol} $(\Delta/n)$, multiplicatively in $n \in \mathbb{N}$ defined by
\begin{equation*}
\left(\frac{\Delta}{2}\right) := \left\lbrace \begin{matrix} \textnormal{ } 0 &\ \textnormal{ if } \Delta \equiv 0 \mod{4} \\ \textnormal{ } 1 &\ \textnormal{ if } \Delta \equiv 1 \mod{8}  \\ -1 &\ \textnormal{ if } \Delta \equiv 5 \mod{8}  \end{matrix} \right.
\end{equation*}
and
\begin{equation*}
\left(\frac{\Delta}{p}\right) := \left\lbrace \begin{matrix} \textnormal{ } 0 &\ \textnormal{ if } \Delta \equiv 0 \mod{p} \hfill\null \\ \textnormal{ } 1 &\ \textnormal{ if } \Delta \equiv x^2 \mod{p} \textnormal{ for some } x \in \mathbb{N} \textnormal{ coprime with } p \\ -1 &\ \textnormal{ otherwise} \hfill\null \end{matrix} \right.
\end{equation*}
for odd primes $p$. In case $\Delta < -4$ it holds (s. \cite{BS}, ch.5.4, thm.1 or \cite{Zagier}, ch.9, thm.3 or \cite{Cassels}, app.B, thm.2.1)
\begin{equation*}
h(\Delta) = \frac{1}{\Delta} \sum\limits_{n=1}^{|\Delta|-1} \left(\frac{\Delta}{n}\right) n .
\end{equation*}

\begin{ex}\label{ex_classNo}
$h(-7) = -(1+2-3+4-5-6)/7 = 1$. That means: All positive definite integral forms of discriminant $-7$ are properly equivalent.
\end{ex}

\section{Orthogonal group}\label{sec_orthogGroup}
The importance of the orthogonal group for symmetric matrices was pointed out first by M. Eichler \cite{Eichler}. Like in section \ref{sec_quadraticforms}, $M$ denotes a module over an integral domain $\mathbb{O}$ with $1+1 \ne 0$ and with finite $\mathbb{O}$-basis $e_1,...,e_n$.

\begin{defn}\label{def_automorph}
An automorphism $l:M \to M$ is called an \textit{automorph} of a quadratic form $q:M \to \mathbb{O}$ when $q \circ l = q$. The set $\textnormal{O}(q)$ of all automorphs of $q$ is called the \textit{orthogonal group}\footnote{Indeed, it is a group under composition of automorphisms.} of $q$.
\end{defn}

\begin{rem}\label{rem_automorph}
When $A$ denotes the matrix that represents $l \in \textnormal{O}(q)$ with respect to $e_1,...,e_n$ (s. Remark \ref{rem_module}c)) and $P$ denotes the symmetric matrix that represents $q$ with respect to the same basis (s. Remark \ref{rem_polarform}) then it holds $A^t P A = P$. Conversely, every such $A \in \textnormal{GL}_n(\mathbb{O})$ corresponds with an automorph of $q$. Therefore, the orthogonal group $\textnormal{O}(q)$ corresponds to the subgroup $\lbrace A \in \textnormal{GL}_n(\mathbb{O}) : A^t P A = P \rbrace$ with respect to the basis $e_1,...,e_n$. According to Remark \ref{rem_directsum}a) and Remark \ref{rem_equiv}a) the determinant $|A|$ of an automorph of a regular quadratic form fulfills $|A|^2 = 1$.
\end{rem}

\begin{ex}\label{ex_orthogGroup}
The elements of the orthogonal group
\begin{equation*}
\textnormal{O}_n(\mathbb{O}) := \left \lbrace Q \in \textnormal{GL}_n(\mathbb{O}) : Q^t Q = E_n \right \rbrace
\end{equation*}
of the quadratic form $x_1^2 + ...+ x_n^2$ on $\mathbb{O}^n$ are called \textit{orthogonal}. They are very helpful for solving linear equation systems (s. subsection \ref{subsec_lsf}) since the inverse of such a matrix is just its transpose.
\end{ex}

\subsection{Orthogonal matrices}\label{subsec_orthogMatrix}
In the following two subsections we restrict to fields $\mathbb{K}$ with $1+1 \ne 0$.

\begin{prop}\label{prop_symmetry}
If $q:V \to \mathbb{K}$ is a regular quadratic form with polar form $\varphi$ on a finite-dimensional vectorspace $V$ over $\mathbb{K}$ then every $l \in \textnormal{O}(q)$ is a composition of \textit{symmetries} \begin{equation*}
s_y(x) := x - 2 \frac{\varphi(x,y)}{q(y)} y \:,\: q(y) \ne 0 \:,
\end{equation*}
i.e. $l = s_{y_1} \circ ... \circ s_{y_n}$ for some $y_1,...,y_n \in V$ ($n \in \mathbb{N}$) with $q(y_i) \ne 0$.
\end{prop}
\begin{proof}
See \cite{Cassels}, ch.2, lem.4.3.
\end{proof}

\begin{ex}\label{ex_symmetric}
a) Due to Proposition \ref{prop_symmetry} the group $\textnormal{O}_n(\mathbb{K})$ of orthogonal matrices (s. Example \ref{ex_orthogGroup}) is generated by the symmetric matrices
\begin{equation*}
S_y := E_n - \frac{2}{y y^t} y^t y
\end{equation*}
where $y \in \mathbb{K}^n$ denotes a row vector with $y y^t \ne 0$. The linear map $x \mapsto S_y x^t$ represents the reflection in the hyperplane perpendicular to $y$. Its determinant is $-1$. The symmetric elements of $\textnormal{O}_n(\mathbb{K})$ are called \textit{Householder matrices} (cf. \cite{Householder}). They have the nice property that they are invariant under inversion.

b) Show that the Householder matrix $H := \textnormal{diag}(1,1,-1,-1) \in \textnormal{O}_4(\mathbb{R})$ is not a symmetry by showing that the rank of $E_4 - H$ differs from one.
\end{ex}

The following assertion characterises Householder matrices over $\mathbb{K}$ as reflections in subspaces of $\mathbb{K}^n$.

\begin{prop}\label{prop_householder}
For all $Q \in \textnormal{Sym}_n(\mathbb{K})$ the equation $Q^2 = E_n$ holds if and only if $Q$ is similar to a diagonal matrix $\textnormal{diag}(1,...,1,-1,...,-1)$ with entries $+ 1$ or $-1$ of arbitrary number.
\end{prop}
\begin{proof}
See \cite{O'Meara}, 42:14.
\end{proof}

\begin{rem}\label{rem_rotation}
In general, orthogonal matrices are not symmetric as shown by the \textit{rotation} matrix
\begin{equation*}
R_\alpha := \left(\begin{matrix} \cos \alpha &\ -\sin \alpha \\ \sin \alpha &\ \cos \alpha \end{matrix}\right) \in \textnormal{O}_2(\mathbb{R})
\end{equation*}
with $\alpha \ne k \pi , k \in \mathbb{Z}$. Hence, the proposition does not apply to all orthogonal matrices. But over $\mathbb{R}$, for every orthogonal $R$ there is an orthogonal $S$ s.t. $S^t R S = S^{-1} R S$ is a \textit{generalised diagonal matrix}
\begin{equation*}
\textnormal{diag}(1,...,1,-1,...,-1,R_{\alpha_1},...,R_{\alpha_k})
\end{equation*}
in so far as some pairs of diagonal elements of the diagonal matrix in Proposition \ref{prop_householder} must be 'replaced' by rotation matrices $R_{\alpha_i}$ for some $\alpha_i \in \mathbb{R}$. This follows from the fact (cf. \cite{SerreD}, cor.5.2) that for every \textit{unitary} $R \in \mathbb{C}^{n \times n}$, i.e. $R^t \bar R = E_n$ (s. Remark \ref{rem_orthogonalisation_real}), there is a unitary $S$ s.t. $S^t R S$ is diagonal. The number of diagonal entries equal to one and minus one, respectively, and the $R_{\alpha_i}$ are uniquely determined by $R$.
\end{rem}

\begin{ex}
a) For every $R \in \textnormal{O}_3(\mathbb{R})$ there is an orthogonal $S$ s.t.
\begin{equation*}
S^t R S =  \left(\begin{matrix} |R| &\ 0 &\ 0 \\ 0 &\ \cos \alpha &\ -\sin \alpha \\ 0 &\ \sin \alpha &\ \cos \alpha \end{matrix}\right)
\end{equation*}
with $\alpha \in \rbrack -\pi , \pi \rbrack$ uniquely defined by $2 \cos \alpha = \textnormal{tr}(R) - |R|$. Remind $|R| = \pm 1$! Hence in case $|R| = 1$ the linear map represented by $R$ acts as a rotation by angle $\alpha$ and in case $|R| = -1$ as a composition of a rotation and a symmetry (reflection in some plane through the origin).

b) Show that for the symmetry
\begin{equation*}
S := \left(\begin{matrix} \cos(2 \alpha) &\ \sin(2 \alpha) \\ \sin(2 \alpha) &\ -\cos(2 \alpha) \end{matrix}\right)
\end{equation*}
the linear map $x \mapsto S x$ is a reflection in the line of polar angle $\alpha$.
\end{ex}

\subsection{Witt's Cancellation Theorem}\label{subsec_Witt}
The following theorem due to Witt \cite{Witt} is very important for the (classical) classification theory (cf. section \ref{sec_classification}) of symmetric matrices over fields $\mathbb{K}$ (as above).

\begin{thm}\label{thm_witt}
For a quadratic form $q:W \to \mathbb{K}$ that is regular on a subspace $U \subseteq W$ and an isomorphism $l:U \to V \subseteq W$ with $q(l(u)) = q(u)$ for all $u \in U$ there is an automorph of $q$ that coincides with $l$ on $U$ and maps $U^{\perp}$ onto $V^{\perp}$.
\end{thm}
\begin{proof}
See \cite{Cassels}, ch.2, thm.4.1 for the first assertion. The latter assertion follows from Remark \ref{rem_directsum}c).
\end{proof}

This can be interpreted in terms of equivalence of symmetric matrices.

\begin{cor}\label{cor_witt}
For $P,Q \in \textnormal{Sym}_m(\mathbb{K})$ and $R,S \in \textnormal{Sym}_n(\mathbb{K})$ with $|R| \ne 0$ the equivalence of $R$ with $S$ and of the two block matrices
\begin{equation*}
\left(\begin{matrix} P \ &\ O \\ O &\ R \end{matrix}\right) , \left(\begin{matrix} Q &\ O \\ O &\ S \end{matrix}\right) \in \textnormal{Sym}_{m+n}(\mathbb{K})
\end{equation*}
with $O$ denoting zero-matrices implies the equivalence of $P$ with $Q$. Clearly, the same assertion holds for permuted diagonal blocks of each block matrix.
\end{cor}
\begin{proof}\footnote{Using matrices in the proof would be rather cumbersome. The perspective of \texttt{quadratic spaces} will reveal its power here.}
By hypothesis the quadratic forms $s,t:\mathbb{K}^{m+n} \to \mathbb{K}$ corresponding to the given block matrices (with respect to the canonical unit basis) are equivalent, i.e. $s = t \circ l$ for an automorphism $l$ of $W := \mathbb{K}^{m+n}$. Since $s$ is regular on $U := \lbrace (0,...,0) \rbrace \times \mathbb{K}^n \subset W$ so is $t$ on $l(U)$. Hence there is an automorph $\tau$ of $t$ with $\tau \circ l(U) = U$ and $\tau \circ l(U^{\perp}) = \tau \left(l(U)^{\perp}\right) = U^{\perp}$ due to Theorem \ref{thm_witt} and Remark \ref{rem_directsum}c). Because of regularity on $U$ it holds $U^{\perp} = \mathbb{K}^m \times \lbrace (0,...,0) \rbrace$. Therefore we have an automorphism $\sigma$ of $\mathbb{K}^m$ defined by $\sigma(x) := \pi \circ \tau \circ l(x,0,...,0)$ where $\pi$ denotes the projection onto the first $m$ coordinates. For the quadratic forms $p,q:\mathbb{K}^m \to \mathbb{K}$ corresponding to $P,Q$, respectively, it follows now $p(x) = s(x,0,...,0) = t(\sigma(x),0,...,0) = q \circ \sigma(x)$ for all $x \in \mathbb{K}^m$. Hence $P$ and $Q$ are equivalent.
\end{proof}

So, for quadratic forms $p,q:\mathbb{K}^m \to \mathbb{K}$ and a regular quadratic form $r:\mathbb{K}^n \to \mathbb{K}$ s.t. $p(x_1,...,x_m)+r(x_{m+1},...,x_{m+n})$ and $q(x_1,...,x_m)+r(x_{m+1},...,x_{m+n})$ are equivalent $p$ is already equivalent to $q$.

\begin{ex}\label{ex_witt}
The quadratic forms $2 x^2 + 6 x y + 3 y^2 + z^2$ and $x^2 + 4 x y + y^2 + 4 z^2$ are rationally inequivalent. Otherwise $2 x^2 + 6 x y + 3 y^2$ and $x^2 + 4 x y + y^2$ would be equivalent according to the Corollary since $4 z^2$ and $z^2$ are equivalent. But the latter two binary quadratic forms are inequivalent as shown in Example \ref{ex_geomEquiv}.
\end{ex}

Now, we answer the question of subsection \ref{subsec_rational}.

\begin{thm}\label{thm_ratequiv}
Two regular quadratic forms with rational coefficients are equivalent over $\mathbb{Q}$ if and only if they are equivalent over $\mathbb{R}$ and every local field.
\end{thm}
\begin{proof}
The necessity of the condition is clear since $\mathbb{Q}$ is contained in $\mathbb{Q}_\infty := \mathbb{R}$ and in every local field $\mathbb{Q}_p$ ($p$ an element of the set $\mathbb{P}$ of primes; s. Remark \ref{rem_local}). Now, we consider quadratic forms $q,r$ that are equivalent over $\mathbb{Q}_p$ for all $p \in \mathbb{P} \cup \lbrace \infty \rbrace$. By Proposition \ref{prop_quadraticforms} the zero-form is the only quadratic form whose coefficient matrix is zero. Since $q$ is regular it rationally represents a non-zero rational number $a$. By hypothesis $r$ represents $a$ over every $\mathbb{Q}_p$ because equivalent forms represent the same elements due to Remark \ref{rem_equiv}e). Therefore, it represents $a$ also rationally due to Corollary \ref{cor_minkowski_hasse}. According to Corollary \ref{cor_orthogonalisation} the coefficient matrices of $q$ and $r$ are equivalent to symmetric matrices $Q = (q_{i j})$ and $R = (r_{i j})$, respectively, with $q_{1 1} = r_{1 1} = a$ and $q_{1 j} = r_{1 j} = 0$ for $j > 1$. Now, we proceed by induction on the dimension $n$ of the underlying vectorspace. For $n=1$ the assertion is true since then $Q = R$. Otherwise $Q$ and $R$ are block matrices with $(a)$ as an upper left diagonal 'block'. From the equivalence of $Q$ and $R$ over $\mathbb{Q}_p$ it follows the equivalence over $\mathbb{Q}_p$ of the lower right diagonal blocks $\tilde{Q}$ and $\tilde{R}$ of $Q$ and $R$ due to Corollary \ref{cor_witt}. Therefore, by the induction hypothesis we may assume that $\tilde{Q}$ and $\tilde{R}$ are equivalent over $\mathbb{Q}$. But then $Q$ and $R$ are also equivalent.
\end{proof}

\subsection{Automorphs of integral binary forms}\label{subsec_automIntBin}
In order to deepen our knowledge of the group in Theorem \ref{thm_composition} we study the \textit{orthogonal group} $\textnormal{O}_q := \lbrace A \in \textnormal{GL}_2(\mathbb{Z}) : A^t P A = P \rbrace$ of its representing forms $q(x,y) = (x,y)P(x,y)^t$.

\begin{rem}\label{rem_automIntBin}
It holds either $\textnormal{O}_q = \textnormal{O}_q^{+} := \lbrace A \in \textnormal{SL}_2(\mathbb{Z}) : P.A = P \rbrace$ or the disjoint union $\textnormal{O}_q = \textnormal{O}_q^{+} \cup A \textnormal{O}_q^{+}$ where $A \in \textnormal{GL}_2(\mathbb{Z})$ is an arbitrary automorph of $q$ with negative determinant, i.e. $|A|=-1$. This is clear since for another automorph $B$ of negative determinant we have $A B^{-1} \in \textnormal{SL}_2(\mathbb{Z})$.
\end{rem}

\begin{lem}\label{lem_improperAutom}
A primitive integral binary form has an automorph of determinant $-1$ if and only if the square of its proper class is the neutral element.
\end{lem}
\begin{proof}
For arbitrary ring elements $\alpha,\beta,\gamma$ it holds
\begin{equation*}
I \left(\begin{matrix} \alpha &\ \beta \\ \beta &\ \gamma \end{matrix} \right) I = \left(\begin{matrix} \gamma &\ \beta \\ \beta &\ \alpha \end{matrix} \right) \textnormal{ with } I := \left(\begin{matrix} 0 &\ 1 \\ 1 &\ 0 \end{matrix} \right) .
\end{equation*}
The square of a proper class $F$ of a form $[\alpha,\beta,\gamma]$ is the neutral element if and only if $F = F^{-1}$, i.e. $[\alpha,\beta,\gamma]$ is properly equivalent to $[\gamma,\beta,\alpha]$ due to the last assertion of Theorem \ref{thm_composition}. According to Remark \ref{rem_automIntBin} any automorph of negative determinant is of the form $I A$ for some automorph $A \in \textnormal{SL}_2(\mathbb{Z})$. The above equation shows that for such an $IA \in \textnormal{O}_q \setminus \textnormal{O}_q^{+}$ the proper class of $A^t I^t P I A$ is the inverse of the proper class of $P$. This proves the first direction. Conversely, if the proper class of $P$ equals the proper class of $I^t P I$ then there is some $A \in \textnormal{SL}_2(\mathbb{Z})$ s.t. $A^t I^t P I A = P$. So $IA$ is an automorph of negative determinant.
\end{proof}

\begin{ex}
For the primitive form $n_\Delta := \left[1,\Delta,(\Delta^2 - \Delta)/4\right]$ of discriminant $\Delta \equiv 0 \textnormal{ or } 1 \mod{4}$ the matrix
\begin{equation*}
\left(\begin{matrix} 1 &\ \Delta \\ 0 &\ -1 \end{matrix}\right)
\end{equation*}
is an automorph of determinant $-1$. And indeed, that form represents the neutral element $E = E^2$ of the (proper) class group of discriminant $\Delta$.
\end{ex}

\begin{prop}\label{prop_classNo}
The number of classical equivalence classes of primitive forms of discriminant $\Delta$ is
\begin{equation*}
\frac{g^{+}(\Delta)+h^{+}(\Delta)}{2}
\end{equation*}
where $h^{+}$ denotes the proper class number\footnote{i.e. the order of the group in Theorem \ref{thm_composition}} and $g^{+}$ the number\footnote{It is called the \textit{proper} \textit{genus number}. The analogous number $g$ for geometric classes is called the (\textit{geometric}) \textit{genus number}. The number of \textit{linear equivalence classes}, defined by the equations $q' = \pm q.A  (A \in \textnormal{GL}_2(\mathbb{Z}))$, of primitive forms $q$ is $(g+h)/2$.} of proper classes whose squares are neutral.
\end{prop}
\begin{proof}
A classical class is the union of two proper classes if and only if its forms do not have automorphs of determinant $-1$. Hence, due to Lemma \ref{lem_improperAutom}, the number in question is
\begin{equation*}
g^{+} + \frac{h^{+}-g^{+}}{2} = \frac{g^{+}+h^{+}}{2} .
\end{equation*}
\end{proof}

In section \ref{sec_crypto} we are interested in group elements of high order only. Therefore, with regard to Remark \ref{rem_automIntBin} and Lemma \ref{lem_improperAutom}, we now concentrate on the \textit{proper orthogonal group} $\textnormal{O}_q^{+}$. The following example focuses on special ring units that correspond with that group elements in the sense of the next proposition.

\begin{ex}\label{ex_units}
By Remark \ref{rem_normform}a) the units $x + y \omega$ of the quadratic order $\mathbb{O}_\Delta$ of Example \ref{ex_quadrOrder}a) are defined by $n_\Delta(x,y) \in \lbrace -1 , 1 \rbrace$, and they correspond bijectively with the integral solutions $(t,u)$ of $t^2 - \Delta u^2 \in \lbrace -4 , 4 \rbrace$. So for $\Delta < -4$ there are only the two units $\pm 1$ of $\mathbb{O}_\Delta$.
\end{ex}

\begin{prop}\label{prop_SO}
For a primitive integral form $q = [\alpha,\beta,\gamma]$ of non-square discriminant $\Delta$ the map
\begin{equation*}
(x,y) \mapsto \left(\begin{matrix} x + y \frac{\Delta - \beta}{2} &\ - \gamma y \\ \alpha y &\ x + y \frac{\Delta + \beta}{2} \end{matrix}\right)
\end{equation*}
is bijective between the set of integral solutions $(x,y)$ of $n_\Delta(x,y)=1$ and $\textnormal{O}_q^{+}$. With multiplication in $\mathbb{O}_\Delta^{\times}$ it defines even an isomorphism.
\end{prop}
\begin{proof}
With the bijective correspondence declared in Example \ref{ex_units} the proof is shown in \cite{Zagier}, ch.8, thm.2.
\end{proof}

\begin{rem}\label{rem_SO}
By Remark \ref{rem_normform}a) the norm function yields a homomorphism from $\mathbb{O}_\Delta^{\times}$ to $\lbrace \pm 1 \rbrace$. There is an isomorhism between $\mathbb{O}_\Delta^{\times}$ and the \textit{geometric automorphism group} $\lbrace A \in \textnormal{GL}_2(\mathbb{Z}) : q.A = q \rbrace$ (not to be confused with $\textnormal{O}_q$).\footnote{For details see \cite{KahlDiplThesis}, ch.6.} Hence in case there is some $(x,y) \in \mathbb{Z}^2$ with $n_\Delta(x,y) = -1$ any proper equivalence class of a primitive form $q$ of discriminant $\Delta$ equals a geometric equivalence class since then $q$ has a geometric automorph of negative determinant. In the other case every geometric equivalence class of a primitive form of discriminant $\Delta$ decomposes into two proper equivalence classes, because $q.A = r$ with $|A| = -1$ and $r.B = q$ with $|B| = 1$ imply $q.AB = q$ with $|AB| = -1$. Hence the corresponding class numbers fulfill $h^{+}(\Delta) = h(\Delta)$ or $h^{+}(\Delta) = 2 h(\Delta)$ with equality if and only if  $\mathbb{O}_\Delta$ has a unit of norm $-1$.
\end{rem}

\begin{ex}
a) For a primitive integral form $q$ of discriminant $\Delta < -4$ it holds $\textnormal{O}_q = \lbrace \pm E_2 \rbrace$ (the trivial orthogonal group) and therefore $h^{+}(\Delta) = 2 h(\Delta)$. This is in accordance with Remark \ref{rem_definiteCase} which implies the latter equation for all negative discriminants.

b) It holds $n_{12}(x,y) = x^2 + 12 x y + 33 y^2 \ne -1$ for all $x,y \in \mathbb{Z}$.\footnote{This implies $h(12)=1$ (s. also Example \ref{ex_intGeomEquiv}) by Example \ref{ex_content}a) which shows $h^{+}(12)=2$.} Hint: Assume the contrary and reduce the questionable equation modulo three.
\end{ex}

\begin{rem}\label{rem_genusNumber}
The \textit{genus numbers} $g^{+}(\Delta),g(\Delta)$ (s. Proposition \ref{prop_classNo}) corresponding to the proper and the geometric class group, respectively, of non-square discriminant fulfill $g^{+}(\Delta)=g(\Delta)$ or $g^{+}(\Delta)=2g(\Delta)$ depending on the solvability of $\Delta = x^2 + 4 y^2$ in coprime integers $x,y$. This condition is equivalent with the existence of rational numbers $x,y \in \mathbb{Q}$ with $n_\Delta(x,y)=-1$. Hence in case $\Delta < 0$ it holds $g^{+}(\Delta)=2g(\Delta)$ which is due to the partition into positive and negative proper classes (s. Remark \ref{rem_definiteCase}). The assertion
\begin{equation*}
g^{+}(\Delta)=2g(\Delta) \Leftrightarrow (x,y \in \mathbb{Z},\gcd(x,y)=1 \Rightarrow \Delta \ne x^2 + 4 y^2)
\end{equation*}
and the next formula for non-square discriminants $\Delta > 0$ are shown in \cite{KahlGenusNumber}\footnote{which is originated in \cite{KahlDiplThesis}}:
\begin{equation*}
g(\Delta) = \left \lbrace \begin{matrix}
2^{m-2} &\ \textnormal{ if } q \textnormal{ divides } \Delta \textnormal{ and } (\Delta \textnormal{ or } \Delta/4 \equiv 1 \mod{4}) \\ 
2^m \hfill\null &\ \textnormal{ if } \Delta = 8 \Pi \textnormal{ or } \Delta \equiv 0 \mod{32} \hfill\null \\
2^{m-1} &\ \textnormal{ otherwise } \hfill\null \end{matrix} \right.
\end{equation*}
Hereby $m$ denotes the number of odd prime divisors of $\Delta$, $q$ a prime with $q \equiv 3 \mod{4}$, and $\Pi$ may be $1$ or a product of primes $p \equiv 1 \mod{4}$. This formula follows from the above criterion for $g^{+}=g$ by help of Gauss' formula in \cite{Gauss}, art.257-259:
\begin{equation*}
g^{+}(\Delta) = \left \lbrace \begin{matrix}
2^{m-1} &\ \textnormal{ if } \Delta \textnormal{ is odd or } \Delta/4 \equiv 1 \mod{4} \\ 
2^{m+1} &\ \textnormal{ if } \Delta \equiv 0 \mod{32} \hfill\null \\
2^m \hfill\null &\ \textnormal{ otherwise } \hfill\null \end{matrix}  \right.
\end{equation*}
\end{rem}

\begin{ex}\label{ex_genusNo}
Show $g^{+}(5)=g(5)=g^{+}(20)=g(20)=1 , g^{+}(80)=2g(80)=2$.
\end{ex}

\section{Linear algebraic application: linear equation systems}\label{sec_LAappl}
A very fundamental question of linear algebra is the solvability of the system of linear equations $A x = b$ in (the coordinates of) the vector $x \in \mathbb{R}^{n \times 1}$ for given $A \in \mathbb{R}^{m \times n}$ and $b \in \mathbb{R}^{m \times 1}$. Often it is described by help of $\textnormal{rk}(A)$. But it can be characterised also by a certain matrix equation.

\begin{rem}\label{rem_LESsolvability}
a) For $A \in \mathbb{R}^{m \times n}$ the following theorem guarantees the existence of a matrix $\tilde{A} \in \mathbb{R}^{n \times m}$ s.t. $A \tilde{A} A = A$. Then for $B \in \mathbb{R}^{m \times l}$ and $X \in \mathbb{R}^{n \times l}$ with $A X = B$ it follows $B = A \tilde{A} A X = A \tilde{A} B$. And vice versa, the equation $A \tilde{A} B = B$ yields $X := \tilde{A} B$ as a solution of $A X = B$.

b) The matrix equation $A X = B A$ implies $A X^n = B^n A$ by induction on $n \in \mathbb{N}$. This fact concerns e.g. the theory of stochastic matrices.

c) The matrix $\tilde{A}$ of the following theorem is called the \textit{pseudoinverse} or \textit{Moore-Penrose inverse} of $A$. It has the property that for $x := \tilde{A} b$ the euclidean norm $\| \cdot \|$ of $A x - b$ is at minimum; cf. subsection \ref{subsec_lsf}. This can be seen by the construction of $\tilde{A} =  V \tilde{D} U$ in the proof of the theorem via orthogonal matrices $U , V$ s.t. $D = (d_{i j}) := U A V$ is \textit{quasi-diagonal}, i.e. $d_{i j} = 0$ for $i \ne j$. First realise that $\|D x - b\|$ is at minimum for $x := \tilde{D} b$ whereby $\tilde{D}$ arises from $D$ by transpositon and substituting the non-zero elements $d_{i i}$ by $1 / d_{i i}$. So $y := \tilde{D} U b$ minimises $\|D y - U b\|$, i.e. $x := V \tilde{D} U b = \tilde{A} b$ minimises $\|D V^t x - U b\|$. Then the assertion follows by Proposition \ref{prop_orthogonal} which implies $\|D V^t x - U b\| = \|A x - b\|$.
\end{rem}

\begin{thm}\label{thm_genInverse}
For $A \in \mathbb{R}^{m \times n}$ there is one and only one $\tilde{A} \in \mathbb{R}^{n \times m}$ with
\begin{equation*}
A \tilde{A} A = A , \tilde{A} A \tilde{A} = \tilde{A} , A \tilde{A} \in \textnormal{Sym}_m(\mathbb{R}) , \tilde{A} A \in \textnormal{Sym}_n(\mathbb{R}) .
\end{equation*}
\end{thm}
\begin{proof}
According to \cite{SerreD}, thm.11.4 (about 'singular value decomposition') there are orthogonal matrices $U \in  \mathbb{R}^{m \times m} , V \in \mathbb{R}^{n \times n}$ s.t.
\begin{equation*}
U A V = \left(\begin{matrix} S \ O \\ O \ O \end{matrix}\right) =: D
\end{equation*}
with an invertible diagonal matrix $S$ and zero matrices $O$ of appropiate dimension. By transposing $D$ and substituting $S$ by $S^{-1}$ we obtain a matrix $\tilde{D} \in \mathbb{R}^{n \times m}$ which fulfills the four properties with $D$ instead of $A$. That quasi-diagonal matrix is uniquely determined by these properties (s. the proof of \cite{SerreD}, thm.11.5). The former assertion implies that $\tilde{A} := V \tilde{D} U$ possesses these properties, and the latter assertion implies $\tilde{A} = V \tilde{D} U$ for any matrix $\tilde{A}$ with these properties (hence uniqueness), since then $V^t \tilde{A} U^t = \tilde{D}$.
\end{proof}

\begin{ex}\label{ex_inverse}
For $A \in \mathbb{R}^{m \times n}$ with $\textnormal{rk}(A) = n$ it holds $\tilde{A} = (A^t A)^{-1} A^t$, especially $\tilde{A} = A^{-1}$ for invertible $A \in \mathbb{R}^{n \times n}$.
\end{ex}

\section{Analytic applications}\label{sec_analyticAppl}
The following subsections represent a short list of analytic applications of symmetric matrices. The first item is very well-known. The other three items are also well known and useful in numerical analysis. Some facts of subsection \ref{subsec_analysis} are used. As already declared there vectors are written in column form.

\subsection{Finding local extrema with the Hessian matrix}
In this subsection $f:D \to \mathbb{R}$ denotes a two times differentiable function on an open set $D \subseteq \mathbb{R}^n$. The first fact is a celebrated result of H.A. Schwarz (1834-1921).

\begin{prop}\label{prop_symmetricHessian}
The matrix $\textnormal{H}f$ is symmetric at each point of continuity.
\end{prop}
\begin{proof}
This is standard in any textbook about analysis, e.g. \cite{Rudin}.
\end{proof}

The following is a useful criterion on local extremum points.

\begin{prop}\label{prop_hessian}
For a two times continuosly differentiable function $f:D \to \mathbb{R}$ a point $x_0 \in D$ with $\nabla f(x_0) = o^t$ and positive or negative definite $\textnormal{H}f(x_0)$ is an isolated local minimum or maximum point, respectively. When $x_0$ is a local minimum or maximum point then $\nabla f(x_0) = o^t$ and $\textnormal{H}f(x_0)$ is positive or negative semidefinite, respectively.\footnote{See Remark \ref{rem_posdef}c) which can be transferred to semidefinite matrices, i.e. the defining inequality for all $x$ can be reduced to inequality for all $x$ in a 'neighbourhood' of the zero vector.}
\end{prop}
\begin{proof}
When $\nabla f(x_0) = o^t$ then for an $x \in D$ with line segment $L$ between $x$ and $x_0$ being a subset of $D$ there is some $\xi \in L$ with
\begin{equation*}
f(x) = f(x_0) + \frac{1}{2}(x-x_0)^t \textnormal{H}f(\xi) (x-x_0)
\end{equation*}
due to Theorem \ref{thm_meanValue}. Because of continuity of $\textnormal{H}f$ the matrix $\textnormal{H}f(\xi)$ is also definite if $x$ is near enough to $x_0$. So for $x \ne x_0$ and, let us say, positive definite $\textnormal{H}f(x_0)$ we have $f(x) > f(x_0)$ in a neighbourhood of $x_0$. This shows the first assertion. When $x_0$ is a local extremum point then it holds $\nabla f(x_0) = o^t$ due to the first assertion of Theorem \ref{thm_meanValue}, hence again the above equation. By standard arguments of continuity the semidefiniteness follows.
\end{proof}

The following Remark considers criterions of (semi-)definiteness.

\begin{rem}\label{rem_definiteness}
a) For a positive (semi-) definite matrix $A \in \textnormal{Sym}_n(\mathbb{R})$ every matrix $A(I)$ resulting from $A$ by deleting all rows and colums of index in $I \subset \mathbb{N}_n$ is positive (semi-) definite. This is clear because of $\tilde{x}^t A(I) \tilde{x} = x^t A x$ for all $x \in \mathbb{R}^n$ whose entries of index in $I$ vanish and $\tilde{x}$ resulting from $x$ by deleting all entries of index in $I$. With $-A$ instead of $A$ we obtain an anlogue criterion for negative (semi-) definiteness.

b) A useful criterion from Jacobi (1804-1851) is the following: An $A = (a_{i j}) \in \textnormal{Sym}_n(\mathbb{R})$ is positive definite if and only if $|A_k| > 0$ for all $k \in \mathbb{N}_n$ with $A_k := (a_{i j})_{i,j \in \mathbb{N}_k}$.\footnote{For negative definiteness we need $(-1)^k |A_k|$ istead of $|A_k|$.} The necessity of this condition follows from Remark a) by induction on $n$. We prove the converse also by induction on $n$. The case $n=1$ is trivial. For the induction step $n \to n+1$ we may assume by Remark \ref{rem_equiv}d) that the given $(n+1) \times (n+1)$-matrix is of the form
\begin{equation*}
\left( \begin{matrix} A &\ o \\ o^t &\ \alpha \end{matrix} \right)
\end{equation*}
for some $\alpha \in \mathbb{R}$ and  $A \in \textnormal{Sym}_n(\mathbb{R})$. From hypothesis it follows $\alpha > 0$ and from induction hypothesis that $A$ is positive definite. This implies the assertion.

c) By the same argument as in Remark b) it follows that an analogue condition of semi-definiteness holds with $A(I)$ (s. in Remark a)!) instead of $A_k$ for all $I \subset \mathbb{N}_n$. The submatrices $A_{k,l} := (a_{i j})_{k \le i,j \le l}$ with $k \le l \in \mathbb{N}_n$ do not suffice as shown by the example
\begin{equation*}
A := \left( \begin{matrix} 0 &\ 0 &\ -1 \\ 0 &\ 0 &\ 0 \\ -1 &\ 0 &\ 0 \end{matrix} \right)
\end{equation*}
of a matrix that is not positive semidefinite but fulfills $|A_{k,l}| \ge 0$ for all $k \le l \in \mathbb{N}_n$.

d) Remark c) implies that an $A \in \textnormal{Sym}_2(\mathbb{R})$ is semi-definite if and only if $|A| \ge 0$. Then the product of its main diagonal entries is non-negative. By Remark b) it is even definite if and only if $|A| > 0$. And then the sign $\sigma$ of either of its main diagonal entries determines the kind of definiteness: $\sigma = 1$ means positive definiteness, and $\sigma = -1$ negative definiteness.
\end{rem}

\begin{ex}\label{ex_localExtr}
Find the local extremum points of $f(x,y):=x^2 - y^2 - \left(x^2 + y^2 \right)^2$ in $\mathbb{R}^2$.
\end{ex}

\subsection{Numerically stable Least Squares Fit}\label{subsec_lsf}
In 1794 Gauss solved the problem to find an $x \in \mathbb{R}^n$ s.t. $\|A x - b\|$ is as small as possible for given $A \in \mathbb{R}^{m \times n}$ and $b \in \mathbb{R}^m$. Here $\| \cdot \| := \| \cdot \|_2$ denotes the euclidean norm (s. Example \ref{ex_norm}c)). Therefore, this fundamental task of linear algebra is called 'linear squares fit'.\footnote{With this method Gauss could retrieve the position of the planetoid 'Ceres' in 1801. This success gave him so much reputation amongst the astronomers of his time that he became the director of the observatory in G\"ottingen in 1807. In 1809 he published his 'Theoria motus corporum coelestium sectionibus conicis solem ambientium' about celestial mechanics which also describes the theory of his 'least squares'.} The naive approach of multiplying the (unsolvable) linear equation system $A x = b$ by $A^t$ from the left (thus making it solvable\footnote{e.g. by the method of subsection \ref{subsec_GaussSeidel} under certain condition on $A$}) often yields bad condition of the symmetric coefficient matrix $A^t A$. In order to proceed numerically stable one restricts to orthogonal transformations $Q \in \textnormal{O}_m(\mathbb{R})$ of the original coefficient matrix $A$, so that we have $\|Q A x - Q b\| = \|A x - b\|$ due to Proposition \ref{prop_orthogonal}.\footnote{And when $A$ is invertible $Q A$ has the same euclidean condition as $A$ (s. \cite{HJ}, eq.(5.8.3)).} For sake of simplicity, we restrict to the case that $A$ has at least as many rows as columns, as common in practice. We construct $Q$ s.t. $Q A =: R$ is an upper right triangle matrix since then the corresponding linear equation system can be solved easily by gaussian 'backwards elimination'.

\begin{rem}
We can do so by choosing suitable symmetries (s. Proposition \ref{prop_symmetry})\footnote{Rotations have the disadvantage that rotation angles are calculated by help of transcendental functions like $\arccos$.}: First we take a 'symmetry matrix' $S_y := E_n - 2 y y^t / \|y\|^2$ that maps the first column $x$ of $A$ to a scalar multiple of the first unit vector $e_1$, i.e. $S_y x = \pm \|x\| e_1$ with $y := x \mp \|x\| e_1$.\footnote{The sign in this vector addition can be chosen s.t. there is no digit deletion (in the first coordinate) since that would cause serious problems concerning rounding errors.} Then we restrict to the hyperplane perpendicular to $e_1$ and do the same for the second column of $S_y A$ but only from index $i=2$ on. And so we proceed until the last column (from index $i=n$ on). The corresponding transformation matrices are block matrices with entries one in the upper left diagonal and symmetry matrices as the lower right blocks. The product of all these transformation matrices is the demanded $Q$. In each iteration the part $x$ of the column vector in question must not be the zero vector. Otherwise we skip the iteration. So we end up with the linear equation system $R x = Q b$ in $x$. It might be unsolvable. But in case $\textnormal{rk}(A) = n$ the first $n$ rows of the coefficient matrix $R$ conform an invertible upper right triangular matrix. So by restricting the equation system $R x = Q b$ to the first $n$ equations we obtain a unique solution $x$. The rows of $R$ with index $> n$ equal the zero vector of $\mathbb{R}^n$. The euclidean norm $\sqrt{c_{n+1}^2 + ... + c_m^2}$ of $Q b =: (c_1, ... , c_m)^t$ from index $n+1$ on tells us the minimal error $\|A x - b\|$.
\end{rem}

We illustrate the procedure of this remark by the following

\begin{ex}\label{ex_lsf}
For the full rang matrix
\begin{equation*}
A := \left(\begin{matrix} -4 &\ 1 \\ 0 &\ 1 \\ 3 &\ 1 \end{matrix}\right)
\end{equation*}
we take $y := (-4-5,0,3) = (-9,0,3)$ with $\|y\|^2 = 90$, so that
\begin{equation*}
S_y = \left(\begin{matrix} 1 &\ 0 &\ 0 \\ 0 &\ 1 &\ 0 \\ 0 &\ 0 &\ 1 \end{matrix}\right) - \frac{2}{90} \left(\begin{matrix} 81 &\ 0 &\ -27 \\ 0 &\ 0 &\ 0 \\ -27 &\ 0 &\ 9 \end{matrix}\right) = \frac{1}{5} \left(\begin{matrix} -4 &\ 0 &\ 3 \\ 0 &\ 5 &\ 0 \\ 3 &\ 0 &\ 4 \end{matrix}\right).
\end{equation*}
Then the first and second column of $S_y A$ are $(5,0,0)^t$ and $(-1/5,1,7/5)^t$, respectively. For the next symmetry we take $x := (1,7/5)^t$ which gives us the new $y := (1+\sqrt{74}/5,7/5)$ with $\|y\|^2 = (148+10 \sqrt{74})/25$, whence
\begin{equation*}
S_y = \frac{1}{74 + 5 \sqrt{74}}\left(\begin{matrix} -25 - 5 \sqrt{74} &\ -35 - 7 \sqrt{74} \\ -35 - 7 \sqrt{74} &\ 25 + 5 \sqrt{74} \end{matrix}\right) \approx \left(\begin{matrix} -0.581 &\ -0.814 \\ -0.814 &\ 0.581 \end{matrix}\right).
\end{equation*}
Therefore, an orthogonal matrix $Q$ that transforms $A$ to an upper triangle matrix $R$ is approximately
\begin{equation*}
\left(\begin{matrix} 1& \ 0 &\ 0 \\ 0 &\ -0.581 &\ -0.814 \\ 0 &\ -0.814 &\ 0.581 \end{matrix}\right) \left(\begin{matrix} -0.8 &\ 0 &\ 0.6 \\ 0 &\ 1 &\ 0 \\ 0.6 &\ 0 &\ 0.8 \end{matrix}\right).
\end{equation*}
Hence we have
\begin{equation*}
Q \approx \left(\begin{matrix} -0.8 &\ 0 &\ 0.6 \\ -0.488 &\ -0.581 &\ -0.651 \\ 0.349 &\ -0.814 &\ 0.465 \end{matrix}\right) , Q A \approx \left(\begin{matrix} 5 &\ -0.2 \\ 0 &\ -1.720 \\ 0 &\ 0 \end{matrix}\right).
\end{equation*}
For $b := (1 , 2 , 3)^t$ we obtain $Q b \approx (1 , -3.604 , 0.116)^t$. Hence the solution $(x_1,x_2) \approx (0.284,2.095)$ of
\begin{equation*}
\left(\begin{matrix} 5 &\ -0.2 \\ 0 &\ -1.720 \end{matrix}\right) \left(\begin{matrix} x_1 \\ x_2 \end{matrix}\right) = \left(\begin{matrix} 1 \\ -3.604 \end{matrix}\right)
\end{equation*}
approximates the minimum point $x \in \mathbb{R}^2$ of $\|A x - b\|$ with approximate minimum value $0.116$. In other words: The linear function $l(x) := 0.284 x + 2.095$ of $x \in \mathbb{R}$ fits the points $(-4,1),(0,2),(3,3)$ as good as possible. The 'gaussian error sum' $(l(-4)-1)^2 + (l(0)-2)^2 + (l(3)-3)^2$ is approximately $0.116^2 \approx 0.014$.
\end{ex}

\begin{ex}\label{ex_HHtransformation}
Find an orthogonal matrix $Q$ s.t. $QA$ is upper right triangular for
\begin{equation*}
A := \left(\begin{matrix} -4 &\ 1 \\ 0 &\ 0 \\ 3 &\ 2 \end{matrix}\right) .
\end{equation*}
\end{ex}

\begin{rem}\label{rem_QRdecomposition}
For $Q,R$ as above the equation $A = Q R$ is called a \textit{QR-decomposition} of $A$. But even in case of a symmetric matrix $A$ the matrix $Q^t A Q$ may not be diagonal; s. the following example! Hence the method of this section does not yield eigenvalues of $A \in \textnormal{Sym}_n(\mathbb{R})$ (cf. subsection \ref{subsec_eigen}).
\end{rem}

\begin{ex}\label{ex_QRdecomposition}
Find an approximate QR-decomposition of
\begin{equation*}
\left(\begin{matrix} 1 &\ 2 \\ 2 &\ 1 \end{matrix}\right) .
\end{equation*}
\end{ex}

\subsection{Gauss-Seidel iteration with relaxation}\label{subsec_GaussSeidel}
Iterative methods for approximating the vector $x$ that solves $A x = b$ (cf. subsection \ref{subsec_lsf}) for an invertible \textit{coefficient matrix} $A \in \mathbb{R}^{m \times m}$ and a \textit{right side} $b \in \mathbb{R}^m$ uses an additive decomposition $A = B + C$ where $B$ is an invertible (triangular) matrix. The given equation is equivalent with $x = \varphi(x) := B^{-1}(b - C x)$, and the \textit{iteration} is defined recursively by $x_n := \varphi(x_{n-1})$ for all $n \in \mathbb{N}$ \textit{starting} from a certain $x_0 \in \mathbb{R}^m$.\footnote{In fact, $B$ is not to be inverted. But $x_n$ may be calculated as the solution $x$ of $B x = b - C x_{n-1}$ by 'gaussian (forward) elimination' if the coefficient matrix $B$ of this system is a lower left triangle matrix. For the Gauss-Seidel-iteration described below there is a more efficient algorithm where each coordinate of $x_n$ is computed in a double loop in dependence of the new coordinates of lower index and the old coordinates of higher index (s. \cite{SerreD}, ch.12.2!).} The iteration is called \textit{convergent} when this sequence converges (towards the solution). It is called \textit{globally convergent} when the convergence does not depend on the starting point $x_0$ and not on $b$. In case of the \textit{Gauss-Seidel iteration} $B$ is of the form $L + D / \omega$ with $L = (l_{i j})$ as the \textit{lower left triangle part} of $A=(a_{i j})$, i.e. $l_{i j} := 0$ for $i \le j$ and $l_{i j} := a_{i j}$ for $i > j$, $D := \textnormal{diag}(a_{1 1},...,a_{m m})$ and $\omega \in \mathbb{R}$ the \textit{relaxation parameter}.\footnote{This parameter is used to accelerate convergence. For $\omega = 1$ it is the classical Gauss-Seidel iteration (with no relaxation).} For a coefficient matrix of the form $A^t A$ (instead of $A$) with $A \in \mathbb{R}^{m \times n}$ of full rank $n \le m$ the Gauss-Seidel iteration is well-defined according to Example \ref{ex_posdef} and Remark \ref{rem_posdef}a).

\begin{lem}\label{lem_posdef}
For a symmetric positive definite matrix $A$ and an invertible matrix $B$ s.t. $B^t - C$ with $C := A - B$ is also positive definite it holds $\|B^{-1} C x\|_A < \|x\|_A$ for all $x \ne o$.\footnote{Recall the definition of $\| \cdot \|_A$ in Proposition \ref{prop_euclideanNorm}.}
\end{lem}
\begin{proof}
It holds $\|B^{-1} C x\|_A^2 = \|x\|_A^2 - y^t (B^t - C) y$ for the non-zero vector $y := B^{-1} A x$ (cf. \cite{SerreD}, ch.12.3.2, Lemma 20). Since $B^t - C$ is symmetric according to Remark \ref{rem_posdef}b) this implies the assertion.
\end{proof}

The following criterion of convergence (s. \cite{SerreD}, thm.12.1.) is from D.M. Young.

\begin{thm}\label{thm_GaussSeidel}
The Gauss-Seidel iteration converges globally for symmetric, positive definite coefficient matrices if the relaxation parameter is in the open interval between zero and two. It does not converge for any starting point and any right side if the relaxation parameter is not in that open interval.
\end{thm} 
\begin{proof}
With notation as above we have $B^t - C = R + D / \omega - (R + D - D / \omega) = (2 / \omega - 1) D = (2 - \omega) D / \omega$ which is positive definite if and only if $\omega \in \left] 0 , 2 \right[$ according to Remark \ref{rem_posdef}a). Now the first assertion follows from Lemma \ref{lem_posdef} and Corollary \ref{cor_Banach}. The second assertion is \cite{SerreD}, prop.12.2.
\end{proof}

\begin{ex}\label{ex_GaussSeidel}
Approximate the solution $x$ of the \textit{normal equation} $A^t A x = A^t b$, where $A$ and $b$ are defined in Example \ref{ex_lsf}, with help of two classical Gauss-Seidel iterations with starting point $(0,2)^t$.
\end{ex}

\subsection{Perturbation of eigenvalues of symmetric matrices}\label{subsec_eigen}
As mentioned in Remark \ref{rem_matrixnorm}b) the eigenvalues $\lambda \in \mathbb{C}$ of an $A \in \textnormal{Sym}_n(\mathbb{R})$ are even real, i.e. for an eigenvector $x \in \mathbb{C}^n \setminus \lbrace o \rbrace$ with $A x = \lambda x$ for some $\lambda \in \mathbb{C}$ the imaginary part of $\lambda$ vanishes. This follows from Remark \ref{rem_orthogonalisation_real} which says that there is some $Q \in \textnormal{O}_n(\mathbb{R})$ with $Q^t A Q = D := \textnormal{diag}(\lambda_1,...,\lambda_n)$ for some $\lambda_1,...,\lambda_n \in \mathbb{R}$. By definition of $\textnormal{O}_n$ in Example \ref{ex_orthogGroup} the equation is equivalent with $A Q = Q D$ which shows that the $j$-th column of $Q$ is an eigenvector of $A$ with eigenvalue $\lambda_j$. By suitable permutation of these columns we can order the eigenvalues by magnitude: $\lambda_1 \le ... \le \lambda_n$. In this ordering we set $\lambda(A) := (\lambda_1,...,\lambda_n)$. Now we look at the change of $\lambda(A)$ in dependence on additive change of $A$ by another symmetric matrix $E$. Then $B := A + E$ is also symmetric. According to Hoffman and Wielandt (s. \cite{HJ}, cor.6.3.8) we have

\begin{thm}\label{thm_spectrumDiff}
For $A , B \in \textnormal{Sym}_n(\mathbb{R})$ it holds $\| \lambda(A) - \lambda(B) \|_2 \le \| A - B \|_2$.
\end{thm}

\begin{ex}
It holds $\lambda(A) = (1 , 1)$ and $\lambda(A+E) = (1-\varepsilon , 1 + \varepsilon)$ for
\begin{equation*}
A := \left(\begin{matrix} 1 &\ 0 \\ 0 &\ 1 \end{matrix}\right) , E := \left(\begin{matrix} 0 &\ \varepsilon \\ \varepsilon &\ 0 \end{matrix}\right) , \varepsilon > 0 .
\end{equation*}
Hence, for $B := A + E$ it holds $\| \lambda(A) - \lambda(B) \|_2 = \varepsilon$ and $\| A - B \|_2 = \| E \|_2 = \varepsilon$. So, in this example both sides of the inequality in  Theorem \ref{thm_spectrumDiff} are equal.
\end{ex}

\begin{cor}
For $A , E \in \textnormal{Sym}_n(\mathbb{R})$ and an eigenvalue $\tilde{\lambda} \in \mathbb{R}$ of $A + E$ there is an eigenvalue $\lambda \in \mathbb{R}$ of $A$ with $|\tilde{\lambda} - \lambda| \le \| E \|_2$.
\end{cor}
\begin{proof}
For the eigenvalues $\lambda_1 \le ... \le \lambda_n$ of $A$ and the eigenvalues $\tilde{\lambda}_1 \le ... \le \tilde{\lambda}_n$ of $B := A + E$ it holds $|\tilde{\lambda}_j - \lambda_j| \le \|(\tilde{\lambda}_1 - \lambda_1 , ... , \tilde{\lambda}_n - \lambda_n)\|_2 = \| \lambda(B) - \lambda(A) \|_2$ for all $j \in \mathbb{N}_n$. In particular there is some eigenvalue $\lambda$ of $A$ with $|\tilde{\lambda} - \lambda| \le \| \lambda(B) - \lambda(A) \|_2$. Now Theorem \ref{thm_spectrumDiff} implies the assertion.
\end{proof}

An allied result can be shown independently (cf. \cite{HJ}, thm.6.3.14).

\begin{prop}
For  $A \in \textnormal{Sym}_n(\mathbb{R}) , x \in \mathbb{R}^n , \mu \in \mathbb{R}$ there is some eigenvalue $\lambda \in \mathbb{R}$ of $A$ with $|\lambda - \mu| \| x \|_2 \le \| A x - \mu x \|_2$.
\end{prop}
\begin{proof}
We may assume without loss of generality that $\mu$ is different from every eigenvalue $\lambda_j$ of $A$. Then $D - \mu E_n$ is invertible for $D := \textnormal{diag}(\lambda_1,...,\lambda_n)$. So, according to Remark \ref{rem_matrixnorm}a) and Proposition \ref{prop_orthogonal} we have
\begin{equation*}
\|x\|_2 = \| Q (D - \mu E_n)^{-1} Q^t r \|_2 \le \|(D - \mu E_n)^{-1}\|_2 \|r\|_2
\end{equation*}
with $r := A x - \mu x$ and $Q \in \textnormal{O}_n(\mathbb{R})$ with $Q^t A Q = D$. The eigenvalues of $D - \mu E_n$ are the numbers $\lambda_j - \mu \ne 0$. Hence the eigenvalues of $(D - \mu E_n)^{-1}$ are the numbers $(\lambda_j - \mu)^{-1}$. It follows (s. Remark \ref{ex_matrixnorm}b)!)
\begin{equation*}
\|(D - \mu E_n)^{-1}\|_2 = \max \lbrace |(\lambda_j - \mu)^{-1}| : j \in \mathbb{N} \rbrace = \left(\min \lbrace |\lambda_j - \mu| : j \in \mathbb{N} \rbrace \right)^{-1} .
\end{equation*}
Now, the above equation implies the assertion.
\end{proof}

\begin{ex}
Find the eigenvalue $\lambda$ of the Proposition for
\begin{equation*}
A := \left( \begin{matrix} 2 &\ 3 \\ 3 &\ 3 \end{matrix} \right) , x := (1,1)^t , \mu := 5.5
\end{equation*}
\end{ex}

\section{Geometric applications}
The following two subsections represent applications of symmetric matrices to geometry. They are not well-known although the topics reach far into the ancient history of geometry. Theorem \ref{thm_euclDist} deals with the symmetric matrix $A^t A$ for an arbitrary real matrix $A$. Theorem \ref{thm_area} concerns plane quadrics with external symmetry centre from which we know by section \ref{sec_quadrics} that they are defined by symmetric, real $2 \times 2$-matrices.

\subsection{Euclidean distance via determinants}
In many application fields it is a fundamental task to determine the euclidean distance $d$ between a point $b \in \mathbb{R}^m$ and the linear subspace $\langle a_1 , ... , a_n \rangle \subseteq \mathbb{R}^m$ generated by vectors $a_1 , ... , a_n \in \mathbb{R}^m$:
\begin{equation*}
	d = \min \left\lbrace \| a - b \| : a \in \langle a_1 , ... , a_n \rangle \right\rbrace
\end{equation*}
For the matrix $A$ with columns $a_j$ it means $d = \| A x - b \|$ for a solution $x$ of the problem in subsection \ref{subsec_lsf}.

\begin{lem}\label{lem_euclDist}
	For $A \in \mathbb{R}^{m \times n}$ it holds $\textnormal{rk}(A^t A) = \textnormal{rk}(A)$ and $|A^t A| \ge 0$. We have $|A^t A| = 0$ if and only if $\textnormal{rk}(A) < n$.
\end{lem}
\begin{proof}
	If it holds $A^t A x = o$ for some $x \in \mathbb{R}^n$ the vector $A x$ is orthogonal to all columns of $A$. But $A x$ is a linear combination of those columns. It follows $A x = o$. This proves that the linear space of solutions $x$ of $A x = o$ equals the linear space of solutions of $A^t A x = o$. In particular these linear spaces have same dimension $k$. So the fundamental dimension formula of linear algebra tells us $n - \textnormal{rk}(A) = k = n - \textnormal{rk}(A^t A)$. This implies the first assertion. In case $m < n$ it follows that $A^t A$ does not have full rank. So in this case the determinant vanishes by Remark \ref{rem_module}c). In case $m \ge n$ we use a QR-decomposition of $A$ to see that $A^t A = R^t R$ for some $n \times n$-matrix $R$ with $\textnormal{rk}(R) = \textnormal{rk}(A)$.\footnote{Take the first $n$ rows of $R$ in Remark \ref{rem_QRdecomposition}.} It follows $|A^t A| = |R|^2 \ge 0$ and equality if and only if the rank is not full.
\end{proof}

The following is \cite{KahlLossVal}, Thm. 1, proven via QR-decomposition (s. Remark \ref{rem_QRdecomposition}).

\begin{thm}\label{thm_euclDist}
	For the euclidean distance $d$ between a point $b \in \mathbb{R}^m$ and the subspace generated by the columns of a matrix $A \in \mathbb{R}^{m \times n}$ it holds
	\begin{equation*}
		d \sqrt{|A^t A|} = \sqrt{|(A|b)^t (A|b)|} .
	\end{equation*}
	Here $(A|b) \in \mathbb{R}^{m \times (n+1)}$ is the matrix $A$ extended by $b$ as an extra column.
\end{thm}

\begin{cor}\label{cor_eucl}
	The euclidean distance between a point $b \in \mathbb{R}^m$ and the subspace generated by the columns of a matrix $A \in \mathbb{R}^{m \times n}$ of full rank $n$ is
	\begin{equation*}
		\sqrt{|(A|b)^t (A|b)| / |A^t A|} .
	\end{equation*}
\end{cor}
\begin{proof}
	By Lemma \ref{lem_euclDist} the matrix $A^t A$ has full rank, too. Hence the determinant of $A^t A$ differs from zero. So the assertion follows from Theorem \ref{thm_euclDist}.
\end{proof}

\begin{ex}\label{ex_eucl}
	a) With help of the \textit{Lagrangian identity}
	\begin{equation*}
		\|x\|^2 \|y\|^2 - (x \circ y)^2 = \|x \times y\|^2  \textnormal{ for } x, y \in \mathbb{R}^3
	\end{equation*}
	we may derive from Corollary \ref{cor_eucl} the well-known term $\|a \times b\| / \|a\|$ for the distance between $b \in  \mathbb{R}^3$ and the line $\langle a \rangle$ generated by $a \in \mathbb{R}^3 \setminus \lbrace o \rbrace$ and the term $|(a_1 \times a_2) \circ b| / \|a_1 \times a_2\|$ for the distance between $b$ and the plane generated by linearly independent $a_1 , a_2 \in \mathbb{R}^3$.\footnote{The numerator of the latter term equals the absolute value of the determinant of the matrix with columns $a_1 , a_2 , b$.}
	
	b) For $A \in \mathbb{R}^{(n+1) \times n}$ let $A_i \in \mathbb{R}^{n \times n}$ be the matrix that evolves from $A$ by deleting the $i$-th row. Show that the vector $b := \left((-1)^i |A_i| \right)_{i \in \mathbb{N}_{n+1}}$ is orthogonal to the columns $a_j$ of $A$ and that $\|b\| = \sqrt{|A^t A|}$, i.e.
	\begin{equation*}
		b \circ a_j = 0 \textnormal{ for all } j \in \mathbb{N}_n \textnormal{ and } \sum_{i=1}^{n+1} |A_i|^2 = |A^t A| .
	\end{equation*}
	Hint: Consider $|(A|a_j)|$ via development by the last column and $|(A|b)|$ in Theorem \ref{thm_euclDist}.
\end{ex}

\subsection{Plane area measurement}
In geodesy lengths and angles are measured in order to derive more entities like heights, areas, volumes etc. We concentrate on plane areas.\footnote{Volumes are also treated in \cite{KahlQuadSect}.} A common method for approximating plane areas with curved boundary is \textit{triangulisation}, i.e. summing up triangle areas that cover the area 'as good as possible'. This kind of first order approximation can be improved to a second order approximation by choosing sectors at centre of quadrics: Just take a 'central point of view' in the plane region to be measured and sum up the sector areas under the angular fields that cover the region. With the origin being a fixed centre of symmetry a plane quadric is uniquely determined by three pairwise linearly independent vectors as points of the quadric; s. Theorem \ref{thm_sectorCoef}.

\begin{thm}\label{thm_area}
	Let $a := (a_1,a_2), b := (b_1,b_2), (c_1,c_2) \in \mathbb{R}^2$ pairwise linearly independent vectors lying - as points - on a plane quadric externally centred at the origin. For the triangle area $\Delta$ between $a$ and $b$ and the analytic function $f:\left]-1 , \infty\right[ \to \mathbb{R}_0^{+}$ defined by
	\begin{equation*}
		f(t) := \left\lbrace \begin{matrix} \arccos(t)/ \sqrt{1-t^2} &\ \textnormal{ for } |t| < 1 \\ 1 &\ 	\textnormal{ for } t=1 \\ \textnormal{arcosh}(t)/\sqrt{t^2-1} &\ \textnormal{ for } t > 1 \end{matrix}\right.
	\end{equation*}
	with $\delta := \left(\gamma^2 - \alpha^2 - \beta^2\right) / (2 \alpha \beta)$ and
	\begin{equation*}
		\alpha := \left| \begin{matrix} b_1 &\ b_2 \\ c_1 &\ c_2 \end{matrix}\right| , \beta := \left| \begin{matrix} c_1 &\ c_2 \\ a_1 &\ a_2 \end{matrix}\right| , \gamma := \left| \begin{matrix} a_1 &\ a_2 \\ b_1 &\ b_2 \end{matrix}\right|
	\end{equation*}
	the sector area between $a$ and $b$ equals $\Delta f(\delta)$.
\end{thm}

\includegraphics[width=11.6cm,height=8.2cm]{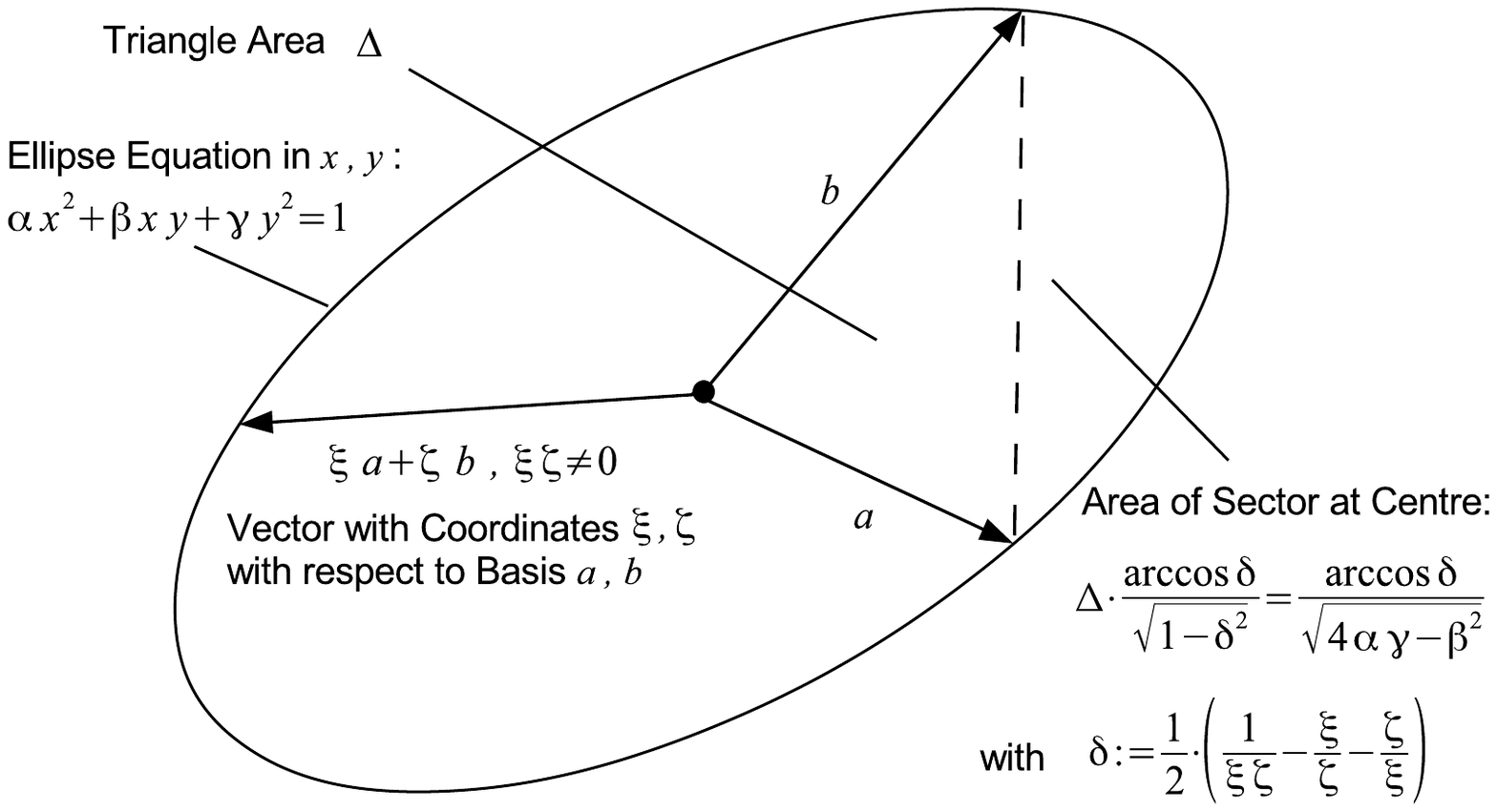}

For the elliptic case\footnote{The hyperbolic case $\delta > 1$ is treated analogously; s. \cite{KahlQuadSect}, sect.3!.} $|\delta| < 1$ the proof of this area formula is sketched in the figure below: It relies on the analytical fact that an area changes under a transformation by the absolute value of the functional determinant and on a linear algebraic formula for the coordinates $\xi, \zeta$ of $c$ with respect to the basis $a,b$ in dependence of the determinants $\alpha,\beta,\gamma$.

\begin{center}
\includegraphics[width=11.6cm,height=8.2cm]{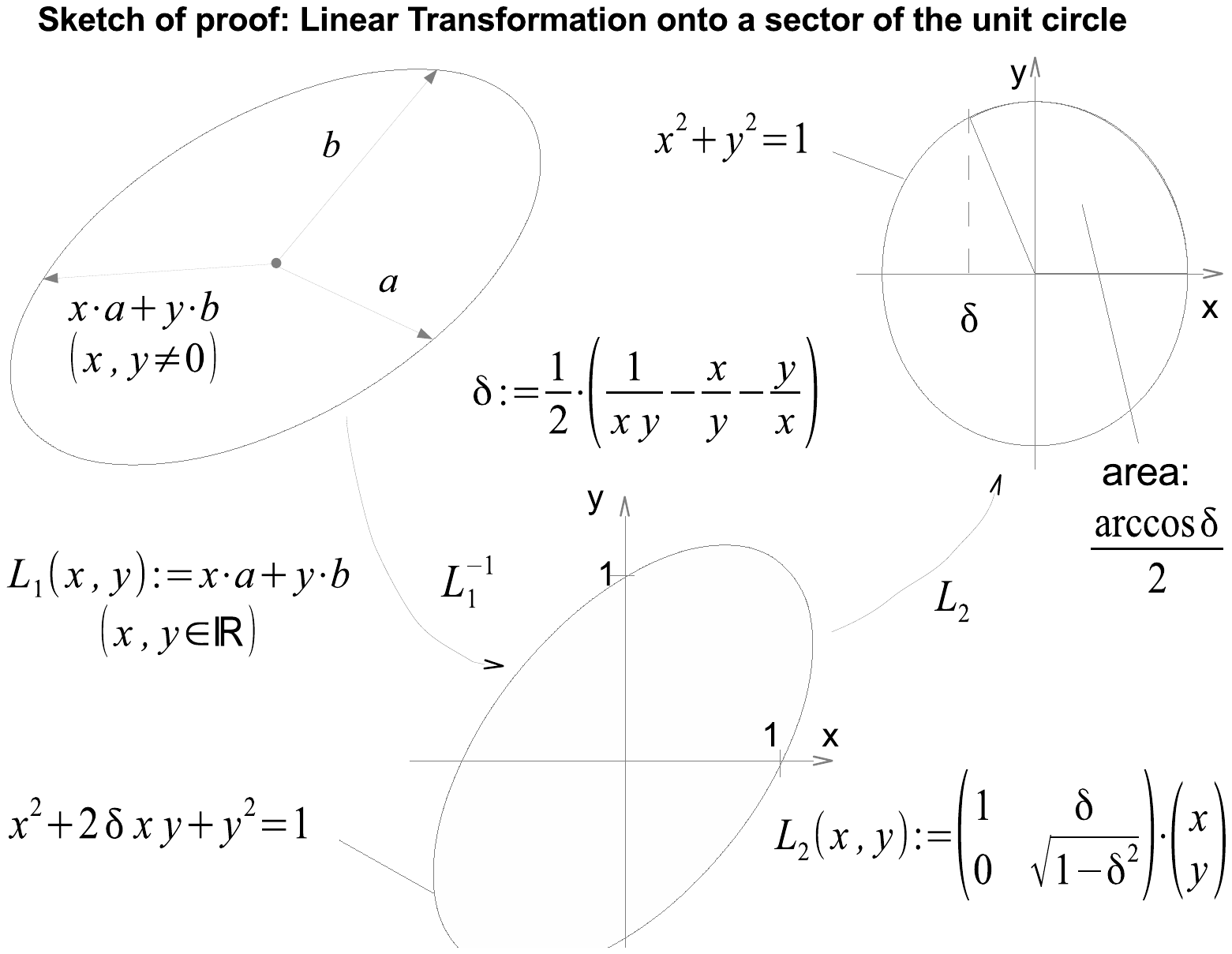}
\end{center}

\begin{rem}
In this proof of the area formula the sector (in the first quadrant) $\lbrace (x,y) \in \mathbb{R}^2 | x,y \ge 0 , x^2 + 2 \delta x y + y^2 = 1 \rbrace$ is used. By rotating this area around the centre by an angle of $\pi/4$ the bounding arc becomes a function of $x$, namely
\begin{equation*}
x \mapsto \sqrt{\frac{1 + x^2 (\delta-1)}{\delta+1}} , \frac{-1}{\sqrt{2}} \le x \le \frac{1}{\sqrt{2}} .
\end{equation*}
So the measure of that area can be computed by integrating this function. By help of $L_1$ and some complex analysis it follows that $f$ is analytically continuable in $1$:
\begin{equation*}
f(t) = \sum\limits_{n=0}^{\infty} a_n (t-1)^n , |t-1| < r
\end{equation*}
with Taylor-coefficients $a_n \in \mathbb{R}$ and a radius $r > 0$ of convergence. On the other hand $f$ fulfills the differential equation $(t^2 - 1) y'(t) + t y(t) = 1$ in $y$ for $|t| < 1$ with initial condition $y(1)=1$. By setting in the above power series it turns out that there is only one analytic solution and
\begin{equation*}
a_n = (-1)^n n! / \prod\limits_{k=1}^{n} (2 k + 1) .
\end{equation*}
Hence the radius of convergence is $r=2$ and we have
\begin{equation*}
f(t) = \sum\limits_{n=0}^{\infty} \prod\limits_{k=1}^{n} \frac{-k}{2 k + 1} (t-1)^n , -1 < t < 3 .
\end{equation*}
So we can evaluate $f$ efficiently with high precision around 1. And for good approximation of a plane area we need small angles of the angular fields that comprise the area, so that we evaluate $f(\delta)$ for arguments $\delta$ nearby $1$ only. When $\varepsilon > 0$ is the given fault tolerance we obtain the error estimation
\begin{equation}\label{eq_formFactor}
\left| f(\delta) - \sum\limits_{m=0}^{n} (1-\delta)^m \prod\limits_{k=1}^{m} \frac{k}{2 k + 1} \right| < \varepsilon \textnormal{ for } n \ge \frac{\ln \left(\varepsilon (1-|\delta-1|/2)\right)}{\ln \left(|\delta-1|/2)\right)} - 1
\end{equation}
by help of the geometric series.
\end{rem}

\begin{ex}\label{ex_sectorArea}
Measure (with compass and ruler) and calculate the area of the elliptic sector region in the figure below Theorem \ref{thm_area}. Use formula \ref{eq_formFactor}, let's say for $\varepsilon := 10^{-2}$.
\end{ex}

From the area formula of a quadric sector at centre follows a generalisation of the concept 'angle' (s. \cite{KahlQuadSect}, sect.4 for details!): For the \textit{sector coefficient}
\begin{equation*}
\delta := \delta(a,b;c) := \frac{1}{2} \left(\frac{1}{x y} - \frac{x}{y} - \frac{y}{x}\right)
\end{equation*} 
of linearly independent $a,b \in \mathbb{R}^n$ and a vector $c$ with $c = x a + y b$ for some $x,y \ne 0$ the \textit{angle}
\begin{equation*}
\angle(a,b;c) := \left\lbrace \begin{matrix} \arccos(\delta) \textnormal{ in case } |\delta| < 1 \\ \textnormal{arcosh}(\delta) \textnormal{ in case } \delta \ge 1 \end{matrix} \right.
\end{equation*}
\textit{between $a$ and $b$ with respect to $c$} fulfills (s. \cite{KahlQuadSect}, cor.4.2) in case of\footnote{Otherwise take $-c$ instead of $c$. Also the condition $\delta > -1$ (of boundedness of the sector region in question) can be achieved by suitable permutation of the three points $a,b,c$: Just take two points $a,b$ of the same component of connectedness of the quadric.} $-c$ \textit{lying between} $a$ and $b$, i.e.  $-c = x a + y b$ for some $x,y > 0$, the equation
\begin{equation*}
\angle(a,-c;b) + \angle(-c,b;a) = \angle(a,b;\pm c)
\end{equation*}
and in (the elliptic) case $|\delta| < 1$ also
\begin{equation*}
\angle(a,b;c) + \angle(b,c;a) + \angle(c,a;b) = 2 \pi .
\end{equation*}
In case $a,b,c$ lying on a circle centred at the origin the positive number $\angle(a,b;c)$ is the usual angle between $a$ and $b$. In case $a,b,c$ lying on a line (not through the origin), i.e. $\delta = 1$, this angle is zero. In general, $\angle(a,b;c)$ is the sector area between $a$ and $b$ times $\sqrt{|\beta^2 - 4 \alpha \gamma|}$ where $\alpha x^2 + \beta x y + \gamma y^2 = 1$ is the defining equation of the quadric that is determined by $a,b,c$.

\begin{ex}
Compute $\angle((2,-1),(2,3);(-3,0))$. Compare the result with the corresponding value of Example \ref{ex_sectorArea}.
\end{ex}

\section{Statistical application: loss value and correlation of multiple linear regression}
In \textit{multiple linear regression} so called \textit{regression coefficients} $\alpha_0, \alpha_1 , ... , \alpha_n$ of the \textit{fitting hyperplane} (in $\mathbb{R}^{n+1}$) $y = \alpha_0 + \alpha_1 x_1 + ... + \alpha_n x_n$ as a function of variables $x_1 , ... , x_n \in \mathbb{R}$ are computed from given (\textit{empirical}) data points
\begin{equation*}
(x_{1 1} , ... , x_{1 n} , y_1) , ... , (x_{m 1} , ... , x_{m n} , y_m) \in \mathbb{R}^{n+1} , m \in \mathbb{N}
\end{equation*}
s.t. the \textit{loss value}
\begin{equation*}
d := \left(\sum_{i=1}^{m} (\alpha_0 + \alpha_1 x_{i 1} + ... + \alpha_n x_{i n} - y_i)^2\right)^{1/2}
\end{equation*}
is at minimum. For the matrix $(1|X)$ that we obtain from $X := (x_{i j})_{i \in \mathbb{N}_m, j \in \mathbb{N}_n}$ by prepending $(1 , ... , 1)^t \in \mathbb{R}^m$ as an extra column (of index $0$) we have $d = \|(1|X) a - y\|$ with $a := (\alpha_0, \alpha_1 , ... , \alpha_n)^t$ and $y := (y_1 , ... , y_m)^t$. I.e.: $a$ must be a solution of the 'least squares fit'-problem of subsection \ref{subsec_lsf}, and the corresponding value of $d$ is nothing else than the euclidean distance between $y$ and the linear space generated by the columns of $(1|X)$. In statistics it is common to express empirical values of expectation with the help of the arithmetic mean $\bar{y} := (y_1 + ... + y_m)/m$ of a (data) vector like $y$ above. We denote by $\hat{y} := (y_1 - \bar{y}, ... , y_m - \bar{y})^t$ the \textit{centering of} $y$ and by $\hat{X}$ the $m \times n$-matrix obtained from $X$ by centering all its columns. Then $\textnormal{cov}(X) := (\hat{X}^t \hat{X})/(m-1)$ is the \textit{sample covariance matrix of the data matrix} $X$. It serves as an estimator of the \textit{covariance matrix of the random vector} $(X_1 , ... , X_n)$ whose $m$ samples are given by $X$, row by row. With the additional random variable $Y$ whose samples are represented by $y$ the \textit{mean squared loss value} $d^2/(m-1)$ of $(X|y)$ is an estimator of the expected value of the random variable $(Y - \alpha_0 - \alpha_1 X_1 - ... - \alpha_n X_n)^2$. Due to \cite{KahlLossVal}, Thm. 3 we have the following formula.

\begin{thm}
	The loss value $d$ of the data matrix $(X|y)$ is
	\begin{equation*}
		d = \sqrt{\left|\left(\hat{X}|\hat{y}\right)^t \left(\hat{X}|\hat{y}\right)\right| / \left|\left(\hat{X}\right)^t \hat{X}\right|}
	\end{equation*}
	in case $\textnormal{rk}(\hat{X}) = n$, i.e. $\textnormal{rk}(1|X) = n+1$.\footnote{This condition of full rank is common in practice where $m$ is often much bigger than $n$.}
\end{thm}

\begin{ex}
	Compute the loss value of the data matrix with four samples
	\begin{equation*}
		\left(\begin{matrix}
			26 &\ 943 &\ 303 \\ 45 &\ 880 &\ 263 \\ 30 &\ 835 &\ 369 \\ 17 &\ 850 &\ 408
		\end{matrix}\right) .
	\end{equation*}
\end{ex}

\begin{rem}
	In terms of sample covariance matrices the mean squared loss value of $(X|y)$ is
	\begin{equation*}
		d^2/(m-1) = |\textnormal{cov}(X|y)| / |\textnormal{cov}(X)| .
	\end{equation*}
	The \textit{sample variance} $\textnormal{v}(y) := \textnormal{cov}(y)$ of a (column) vector $y$ vanishes iff $\hat{y} = o$. So in case $\hat{y} \ne o$ the \textit{multiple correlation coefficient}
	\begin{equation*}
	\rho := \sqrt{1 - |\textnormal{cov}(X|y)| / (|\textnormal{cov}(X)| \textnormal{v}(y))}
	\end{equation*}
	between $y$ and $X$ is well-defined. It holds $\rho = \sqrt{1 - d^2 / (\hat{y}^t \hat{y})} \in [0 , 1]$.
\end{rem}

\section{Cryptographic application: efficient group composition}\label{sec_crypto}
In public key cryptography the major tasks are encryption of rather short secret information (like a secret symmetric key), agreement of a secret (symmetric) key and digital signature. In any case the fundamental function is $(g,n) \mapsto g^n := g \cdot ... \cdot g$ for some group element $g$ of high order and some $n \in \mathbb{N}$. Hereby the base $g$ is a public system parameter. It should be easy to evaluate the function in order to be practical. But for security reasons it must be hard to compute $n$ from $g$ and $g^n$ ('Discrete Logarithm Problem'). In this section we consider the group $Cl(\Delta)$ described in subsection \ref{subsec_integral} for negative discriminants $\Delta$. Due to a theorem of Siegel \cite{Siegel} the digit number of its order $h(\Delta)$ is about half of that of $\Delta$. Currently, a discriminant of $128$ byte length is assumed to be secure enough if the system parameter $g \in Cl(\Delta)$ generates a group of order not much smaller than $h(\Delta)$.  For illustrating how to compute $g^n$ in that group, first remind that each $g \in Cl(\Delta)$ is uniquely represented by a reduced form $[\alpha,\beta,\gamma]$. By regarding $\Delta$ as a system parameter it suffices to store $(\alpha,\beta)$ since $\gamma = (\beta^2-\Delta)/(4 \alpha)$ is determined by the other entities.

\begin{rem}
	At this point it's time for a summary of the composition algorithm resulting from Lemma \ref{lem_composition} and Remark \ref{rem_reduction}. As input we take $(\alpha,\beta), (\alpha',\beta') \in Cl(\Delta)$. The algorithm will overwrite $(\alpha,\beta)$ several times. The corresponding third coefficient will be denoted by $\gamma$ as explained above. At the end $(\alpha,\beta)$ will be the composition of the two input group elements.
	\begin{itemize}
		\item compute the greatest divisor $x$ of $\alpha'$ coprime with $\gamma$
		\item compute the greatest divisor $y$ of $\alpha'$ coprime with $\alpha x$
		\item choose $w , z \in \mathbb{Z}$ s.t. $w x - y z = 1$
		\item substitute $(\alpha,\beta)$ by $(\alpha x^2 + \beta x y + \gamma y^2 , 2 \alpha x z + \beta (w x + y z) + 2 \gamma w y)$
		\item choose $n \in \mathbb{Z}$ s.t. $2 \alpha n \equiv \beta' - \beta \mod{\alpha'}$
		\item substitute $(\alpha,\beta)$ by $(\alpha \alpha' , \beta + 2 \alpha n)$
		\item while $[\alpha,\beta,\gamma]$ is not reduced:
		\subitem compute the greatest integer $n \le (\beta+\gamma)/(2 \gamma)$
		\subitem substitute $(\alpha,\beta)$ by $(\gamma,2 \gamma n - \beta)$
	\end{itemize}
\end{rem}

As an example we represent the Diffie-Hellman key exchange (s. \cite{BV}, algo.12.1) with very small numbers (too small for cryptographic security).

\begin{ex}
We take $g := (2,1) \in Cl(-167)$ (s. Example \ref{ex_classComposition}) for the system parameter. Both parties choose their own secret natural number\footnote{in real life at least of $16$ byte length}, say $a:=4$ and $b:=7$. Then each party computes $g^a = (3,1)$ and $g^b = (3,-1)$, respectively. Then they send their results to each other. Now, both can compute their common (secret) key $(g^a)^b = g^{a b} = (g^b)^a = (6,-1)$.
\end{ex}

\section{Appendix: some analytic and algebraic basics}\label{sec_appendix}
This section presents some standard facts of analysis and algebra.

\subsection{Basic Analysis}\label{subsec_analysis}
This subsection is not meant to be a 'crash course' on calculus. It stresses the fundamental concept of norm (of a matrix) which is used for declaring convergence of sequences in vectorspaces like $\mathbb{R}^n$. In this subsection vectors are identified with column vectors, e.g. $o := (0,...,0)^t$.

\begin{defn}\label{def_norm}
A function $\| \cdot \| : V \to \mathbb{R}_0^{+}:= \lbrace x \in \mathbb{R} : x \ge 0 \rbrace$ is called a \textit{norm} on a vectorspace $V$ over $\mathbb{R}$ when 
\begin{itemize}
\item $\| x \| = 0 \: \Rightarrow x = o$ \ \ \ \ \ (\textit{non-degeneracy})
\item $\| \lambda x \| = |\lambda| \| x \|$ \ \ \ \ \ \ \ \ \ (\textit{homogeneity})
\item $ \| x + y \| \le  \| x \| +  \| y \|$ \ (\textit{triangle inequality})
\end{itemize}
for all $x,y \in V , \lambda \in \mathbb{R}$. Then $V$ (or more exactly: $(V,\| \cdot \|)$) is called a \textit{normed space}. A \textit{sequence} $x : \mathbb{N} \to V$ of vectors $x_n := x(n)$ \textit{converges} to a \textit{limit} vector $\xi \in V$ when for all $\varepsilon > 0$ there is a $k \in \mathbb{N}$ s.t. $\| x_n - \xi \| < \varepsilon$ for all $n \ge k$. A sequence that converges to zero is called a \textit{zero sequence}. A sequence $x : \mathbb{N} \to V$ is called \textit{Cauchy-convergent} or \textit{fundamental} when for all $\varepsilon > 0$ there is a $k \in \mathbb{N}$ s.t. $\| x_n - x_m \| < \varepsilon$ for all $m,n \ge k$. For normed spaces $V , W$ a function $f : M \to W$ is called \textit{continuous at a point} $\xi \in M \subseteq V$ when for every sequence $x : \mathbb{N} \to M$ that converges to $\xi$ the sequence $f \circ x : \mathbb{N} \to W$ converges. When $f$ is continuous at all points of its definition set it is called \textit{continuous}. A subset $M$ of a normed space is called \textit{bounded} when there is a constant $\kappa$ s.t. $\| x \| < \kappa$ for all $x \in M$. It is called \textit{closed} when every convergent sequence $x : \mathbb{N} \to M$ possesses a limit in $M$. A subset of a finite-dimensional normed space is called \textit{compact} when it is bounded and closed. A subset of a normed space is called \textit{open} when it is the complement of a closed subset. For a norm $\| \cdot \|$ on $\mathbb{R}^{n+1}$ ($n \in \mathbb{N}_0$) the set $S_n := \lbrace x \in \mathbb{R}^{n+1} : \| x \| = 1 \rbrace$ is called the ($n$-{\it dimensional}) \textit{unit sphere}.
\end{defn}

\begin{ex}\label{ex_norm}
a) The absolute value or modulus $| \cdot |$ as a function on $\mathbb{R}$ defines a norm (s. Example \ref{ex_ideal}b)). A norm function $\| \cdot \| : V \to \mathbb{R}$ is continuous with respect to the norm $| \cdot |$ on $\mathbb{R}$ because the triangular inequality implies $\left| \|x_n\| - \|\xi\|\right| \le \| x_n - \xi \|$ for $x_n, \xi \in V$.

b) A function  $f : M \to W$ between normed spaces $V \supseteq M$ and $W$ with a \textit{Lipschitz} (1832-1903) constant $\lambda \in \mathbb{R}$ s.t. $\|f(x)-f(y)\| \le \lambda \|x-y\|$ for all $x , y \in M$ is continuous. That is clear by the definitions.

c) For $1 \le p \le \infty$ the function $\|(x_1,...,x_n)\|_p := (|x_1|^p + ... + |x_n|^p)^{1/p}$ defines a norm, called $p$-\textit{norm}, on $\mathbb{R}^n$ according to Minkowski's inequality (s. \cite{SerreD}, prop.7.1). In case $p = \infty$ it is called also the \textit{maximum norm}: $\|(x_1,...,x_n)\|_\infty = \max \lbrace |x_1| , ... , |x_n| \rbrace$. In case $p=2$ it is called the \textit{euclidean norm}.

d) The unit sphere is compact.
\end{ex}

\begin{rem}\label{rem_convergence}
a) A limit vector $\lim\limits_{n \to \infty} x_n$ is uniquely determined by its sequence $(x_n)_n$ because of non-degeneracy and the triangle inequality.

b) Every convergent sequence is fundamental. But not vice versa; E.g.: The sequence $x:\mathbb{N}_0 \to \mathbb{Q}$ recursively defined by $x_n := x_n/2 + 1/x_n , n \in \mathbb{N} , x_0 := 1$ is fundamental but not convergent (with respect to $|\cdot|$; s. Example \ref{ex_norm}a)!). When it is regarded as a sequence of real numbers then it is convergent (with limit $\sqrt{2} \in \mathbb{R} \setminus \mathbb{Q}$).

c) When a function $f:M \to W$ on a subset $M$ of a normed space $V$ fulfills the property of continuity at a point $\xi \in V \setminus M$ then it is called \textit{continuously continuable at/in} $\xi$. Then for every series $x:\mathbb{N} \to M$ that converges to $\xi$ the limit of $f \circ x$ in the normed space $W$ is the same. It is denoted by $\lim\limits_{x \to \xi} f(x)$. When we say that this \textit{limit exists} we mean the continuous continuability. So a function $f$ is continuous at $\xi$ iff $\lim\limits_{x \to \xi} f(x)$ exists and equals $f(\xi)$.

d) An analogon of the $p$-norm of Example \ref{ex_norm}c) can be used to define a norm on a cartesian product of normed spaces $V,W,...$: $\|(x,y,...)\|_p := (\|x\|^p + \|y\|^p + ...)^{1/p}$ for $x \in V, y \in W , ...$.
\end{rem}

\begin{defn}\label{def_matrixnorm}
For $A \in \mathbb{R}^{m \times n}$ the least upper bound $\| A \|$ of $\lbrace \| A x \| : x \in S_{n-1} \rbrace$ is called the ({\it induced matrix}) \textit{norm} of $A$.
\end{defn}

\begin{rem}\label{rem_matrixnorm}
a) The induced matrix norm is a norm. Non-degeneracy and homogeneity are clear. The triangular inequality follows from $\|(A + B) x\| = \|A x + B x\| \le \|A x\| + \|B x\| \le \|A\| + \|B\|$ for all $A , B \in \mathbb{R}^{m \times n}, x \in S_{n-1}$. A matrix norm is \textit{compatible} with the vector norm that induces it: $\|A x\| \le \|A\| \|x\|$ since for $x \ne o$ it holds $\|A x\| / \|x\| = \|A x / \|x\|\| \le \|A\|$. This implies $\|A B\| \le \|A\| \|B\|$ (\textit{sub-multiplicativity}) for $A \in \mathbb{R}^{k \times m}, B \in \mathbb{R}^{m \times n}$ because of $\|A B x\| \le \|A\| \|B x\| \le \|A\| \|B\|$ for $x \in S_{n-1}$.

b) For $p \in \lbrace 1 , 2 , \infty \rbrace$ the $p$-norm of Example \ref{ex_norm}c) induces the matrix norm (s. \cite{SerreD}, ch.7.1.4, examples)\footnote{There, all assertions are shown for square matrices only. But the proofs remain true for non-square matrices, too.}
\begin{equation*}
\|A\|_1 = \max \left\lbrace \sum\limits_{i=1}^{m} |a_{i j}| : j \in \mathbb{N}_n \right\rbrace ,
\end{equation*}
\begin{equation*}
\|A\|_\infty = \max \left\lbrace \sum\limits_{j=1}^{n} |a_{i j}| : i \in \mathbb{N}_m \right\rbrace ,
\end{equation*}
\begin{equation*}
\|A\|_2 = \sqrt{\max \lbrace |\lambda| : \lambda \in \mathbb{C} , A^t A x = \lambda x \textnormal{ for some } x \ne o \rbrace}
\end{equation*}
for $A = (a_{i j}) \in \mathbb{R}^{m \times n}$.\footnote{For existence of the maximum see Remark \ref{rem_normMax}a)!} In particular, we have
\begin{equation*}
\|A\|_2 = \max \lbrace|\lambda| : \lambda \in \mathbb{R} , A x = \lambda x \textnormal{ for some } x \ne o \rbrace
\end{equation*}
for $A \in \textnormal{Sym}_n(\mathbb{R})$. That we may restrict to real numbers $\lambda$ is a consequence of Remark \ref{rem_orthogonalisation_real} and the fact that $B^{-1} A B$ is diagonal if and only if the columns of $B \in \textnormal{GL}_n(\mathbb{R})$ are \textit{eigenvectors} $x$ ($\ne o$ because of $|B| \ne 0$) of $A$, i.e. $ A x = \lambda x$ for some  \textit{eigenvalue} $\lambda \in \mathbb{R}$.
\end{rem}

\begin{ex}\label{ex_matrixnorm}
The non-symmetric matrix
\begin{equation*}
A := \left(\begin{matrix} 1 &\ 1 \\ 0 &\ 1 \end{matrix}\right) = \left(\begin{matrix} a &\ -b \\ b &\ a \end{matrix}\right) \left(\begin{matrix} c &\ 0 \\ 0 &\ 1/c \end{matrix}\right) \left(\begin{matrix} b &\ a \\ -a &\ b \end{matrix}\right)
\end{equation*}
with $c := (1 + \sqrt{5})/2 , a := \sqrt{c / \sqrt{5}} , b := 1 / \sqrt{c \sqrt{5}}$ has euclidean norm $\|A\|_2 = c$ (the 'golden ratio') which is bigger than
\begin{equation*}
\frac{3}{2} = \max \lbrace |x^2 + x y + y^2| : x,y \in \mathbb{R} , x^2 + y^2 = 1 \rbrace .
\end{equation*}
The two non-diagonal matrices in the above \textit{singular value} decomposition of $A$ are (orthogonal) rotation matrices like in Remark \ref{rem_rotation} because of $a^2 + b^2 = 1$. Thus $\|A\|_2 = c$ can be seen also by the following Proposition that characterises orthogonal matrices over $\mathbb{R}$ with help of the euclidean norm $\| \cdot \| := \| \cdot \|_2$.
\end{ex}

\begin{prop}\label{prop_orthogonal}
A quadratic matrix $Q$ over $\mathbb{R}$ is orthogonal if and only if $\|Q x\| = \|x\|$ for all  $x \in \mathbb{R}^n$.
\end{prop}
\begin{proof}
If $Q^t Q = E_n$ then $\|Q x\|^2 = x^t Q^t Q x = x^t x = \|x\|^2$ for all $x$. This proofs one direction. For $S = (s_{i j}) := Q^t Q$ the latter equations show $s_{i i} = 1$ for all $i \in \mathbb{N}_n$. It holds
\begin{equation*}
2 x^t S y = (x+y)^t S (x+y) - x^t S x - y^t S y
\end{equation*}
for all $x,y \in \mathbb{R}^n$ (cf. Remark \ref{rem_polarform}). By hypothesis, the right side of the equation equals $\|x + y\|^2 - \|x\|^2 - \|y\|^2$. When we choose $x:=e_i$ and $y:=e_j$ with $i \ne j$ (as orthogonal unit vectors) it vanishes according to the theorem of Pythagoras. This shows $s_{i j} = e_i^t S e_j = 0$ 
\end{proof}

The following proposition shows that all norms are \textit{equivalent} in some sense.

\begin{prop}\label{prop_equivNorms}
For norms $N_1,N_2 : V \to \mathbb{R}_0^{+}$ on a finite dimensional vectorspace $V$ there are constants $\kappa_1,\kappa_2 \in \mathbb{R}$ s.t. $N_1 < \kappa_1 N_2$ and $N_2 < \kappa_2 N_1$.
\end{prop}
\begin{proof}
s. \cite{SerreD}, prop.7.3!
\end{proof}

\begin{rem}\label{rem_completeness}
a) A sequence of vectors $x_k = (x_{k 1},...,x_{k n}) \in \mathbb{R}^n$ converges if and only if for every $j \in \mathbb{N}_n$ the sequence of coordinates $x_{k j}$ converges. This follows from Proposition \ref{prop_equivNorms} and the above examples: For all $\xi = (\xi_1,...,\xi_n) \in \mathbb{R}^n$ there are $\kappa_1,\kappa_2 \in \mathbb{R}$ s.t.
\begin{equation*}
|x_{k j}-\xi_j| \le \kappa_1 \|x_k - \xi\| \le \kappa_2 \|x_k - \xi\|_\infty .
\end{equation*}

b) An essential property of the field of real numbers is its \textit{completeness}: Every Cauchy-convergent sequence in $\mathbb{R}$ converges. By Remark a) this assertion generalises to $\mathbb{R}^n$. Thus the euclidean space is a \textit{Banach space} as any complete normed space is called. Also $\mathbb{R}^{m \times n}$ is complete with respect to any matrix norm since convergence of a sequence of matrices is equivalent with convergence of all corresponding sequences of entries. This follows by Remark \ref{rem_matrixnorm} and Proposition \ref{prop_equivNorms} with help of the norm $(a_{i j}) \mapsto \max \lbrace |a_{i j}| : i \in \mathbb{N}_m , j \in \mathbb{N}_n \rbrace$ on $\mathbb{R}^{m \times n}$.

c) For an $A \in \mathbb{R}^{n \times n}$ with induced norm $\|A\| < 1$ the matrix $E_n - A$ is invertible with
\begin{equation*}
(E_n - A)^{-1} = \sum\limits_{k=0}^{\infty} A^k .
\end{equation*}
The convergence of the latter series follows from Remark b) and
\begin{equation*}
\left \| \sum\limits_{k=l}^{m} A^k \right \| \le \sum\limits_{k=l}^{m} \|A^k\| \le \sum\limits_{k=l}^{m} \|A\|^k
\end{equation*}
for $l,m \in \mathbb{N}$. The equation follows from
\begin{equation*}
(E_n - A) \sum\limits_{k=0}^{m} A^k = E_n - A^{m+1} , m \in \mathbb{N} .
\end{equation*}
\end{rem}

The following theorem is standard in any textbooks about analysis, e.g. \cite{Rudin}.

\begin{thm}\label{thm_compact}
The image set $f(C)$ of a continuous function $f : C \to \mathbb{R}^m$ on a compact set $C \subset \mathbb{R}^n$ is compact. In case $m=1$ it has a maximum and a minimum. And in this case $f(\left[\alpha_1 , \beta_1 \right] \times ... \times \left[\alpha_n , \beta_n \right])$ is a compact interval for real numbers $\alpha_j < \beta_j , j \in \mathbb{N}_n$.
\end{thm}

The latter assertion is well known as the 'intermediate value theorem'.

\begin{rem}\label{rem_normMax}
a) As a consequence of the second assertion of the Theorem and Example \ref{ex_norm}d) we have $\|A\| = \max \lbrace \| A x \| : x \in S_{n-1} \rbrace$ for all $A \in \mathbb{R}^{m \times n}$.

b) Because of $\|x+y\|_2^2 = \|x\|_2^2 + \|y\|_2^2 + 2 x^t y$ the euclidean norm fulfills the \textit{Cauchy-Schwarz inequality}\footnote{a special case of \textit{H\"older's inequality} (s. \cite{SerreD}, prop.7.1)} $|x^t y| \le \|x\|_2 \|y\|_2$ for all $x,y \in \mathbb{R}^n$. Because of compatibility (s. Remark \ref{rem_matrixnorm}), for $A \in \mathbb{R}^{n \times n}$ it follows $|x^t A x| \le \|A x\|_2 \le \|A\|_2$ when $\|x\|_2 = \|x^t\|_2 = 1$. This shows
\begin{equation*}
\|A\|_2 \ge \max \lbrace |x^t A x| : x \in S_{n-1} \rbrace ,
\end{equation*}
whereby the maximum exists again because of the theorem. Equality does not hold in general as the following example will show. But in case $A \in \textnormal{Sym}_n(\mathbb{R})$ we have equality. This follows from Remark \ref{rem_orthogonalisation_real} which says that there is a $Q \in \textnormal{O}_n(\mathbb{R})$ s.t. $Q^t A Q$ is a diagonal matrix $\textnormal{diag}(\lambda_1,...,\lambda_n)$ for some $\lambda_j \in \mathbb{R}$. Hence for the columns $q_j$ of $Q$ it holds $q_j^t A q_j = \lambda_j , j \in \mathbb{N}_n$. Then
\begin{equation*}
\|A\|_2 = \max \lbrace |x^t A x| : x \in S_{n-1} \rbrace
\end{equation*}
follows according to Remark \ref{ex_matrixnorm}b).
\end{rem}

\begin{ex}\label{ex_continuous}
A (multi-)linear map $V \times ... \times V \to W$ (s. Definition \ref{def_module}) of normed spaces $V,W$ of finite dimension is continuous. In particular it holds
\begin{itemize}
\item $\lim\limits_{n \to \infty} l(x_n) = l \left(\lim\limits_{n \to \infty} x_n \right)$ for a linear map $l:V \to W$ and a convergent sequence $x : \mathbb{N} \to V$
\item $\lim\limits_{n \to \infty} (x_n^t y_n) = \left(\lim\limits_{n \to \infty} x_n \right)^t \left(\lim\limits_{n \to \infty} y_n \right)$ for convergent $x,y : \mathbb{N} \to V$
\end{itemize}
As an exercise verify these two special cases for $V = \mathbb{R}^n , W = \mathbb{R}^m$. Hint: Use the compatibility of the matrix norm (s. Remark \ref{rem_matrixnorm}a)!) for the former formula and the Cauchy-Schwarz inequality (s. Remark \ref{rem_normMax}b)!) for the latter formula.
\end{ex}

A natural generalisation of the euclidean norm is given by some special kind of quadratic matrices.

\begin{defn}\label{def_posdef}
A real quadratic matrix $A$ is called \textit{positive definite} when $x^t A x > 0$ for all column vectors $x \ne o$. It is called \textit{negative definite} when $-A$ is positive definite. Analogously one defines (positive and negative) \textit{semidefiniteness} with '$\ge$' instead of '$>$'.
\end{defn}

\begin{ex}\label{ex_posdef}
For an $A \in \mathbb{R}^{m \times n}$ of full rank $n \le m$ the (symmetric) matrix $A^t A$ is positive definite because the columns of $A$ are linearly independent which implies $A x \ne o$ for every $x \in \mathbb{R}^n \setminus \lbrace o \rbrace$, hence $0 < \|A x\|_2^2 = x^t A^t A x$. In any case $A^t A$ is positive semidefinite.
\end{ex}

\begin{prop}\label{prop_euclideanNorm}
For a positive definite matrix $A \in \mathbb{R}^{n \times n}$ the function
\begin{equation*}
\|x\|_A := \sqrt{x^t A x}
\end{equation*}
of column vectors $x \in \mathbb{R}^n$ defines a vector norm.
\end{prop}
\begin{proof}
The triangle inequality follows from the Cauchy-Schwarz inequality (s.  Remark \ref{rem_normMax}b): $\|x+y\|_A^2 \le \|x\|_A^2 + \|y\|_A^2 + 2 |x^t A y| \le \|x\|_A^2 + \|y\|_A^2 + 2 \|x\|_2 \|A y\|_2 = (\|x\|_A + \|y\|_A)^2$ for all $x,y \in \mathbb{R}^n$. For the latter equation we have assumed without loss of generality that $A$ is symmetric since $(A + A^t)/2$ is so in general. The other two norm properties are easy to show.
\end{proof}

\begin{rem}\label{rem_posdef}
a) A positive definite matrix $A$ is invertible since otherwise there would be a vector $x \ne o$ with $A x = o$ and so $x^t A x = x^t o = 0$. And its main diagonal elements are positive: Just choose the canonical unit vectors for $x$ in the definition.

b) For quadratic matrices $B,C$ s.t. $B+C$ is symmetric the quadratic matrix $B^t - C = B^t + B - (B + C)$ is also symmetric.

c) A matrix $A \in \mathbb{R}^{n \times n}$ is positive definite when $x^t A x > 0$ for all $x \in S_{n-1}$. This is clear by the definition of multiplication with a scalar $\lambda \ne 0$ and the fact $\lambda^2 > 0$. It is positive definite when it fulfills the inequality on an arbitrary open set containing the zero vector. This follows by the topologic property of an open set $D$ of a normed space $V$ that for every $x_0 \in D$ there is a $\delta > 0$ s.t. $\lbrace x \in V : \|x - x_0\| \le \delta \rbrace \subset D$.
\end{rem}

The following fixed point theorem of S. Banach (1892-1945) is useful for iterative approximation methods.

\begin{thm}\label{thm_Banach}
A function $\varphi : C \to C$ on a closed set $C \subseteq \mathbb{R}^m$ with a Lipschitz constant $\kappa < 1$ has a unique \textit{fixed point} $\xi \in C$, i.e. $\varphi(\xi)=\xi$. And for all \textit{starting points} $x_0 \in C$ we have the inequalities
\begin{equation*}
\|x_n-\xi\| \le \left \lbrace \begin{matrix} \kappa^n \|x_0-\xi\| \hfill\null \\ \frac{\kappa^n}{1-\kappa} \|x_1-x_0\| \hfill\null \; \textnormal{ (a-priori-estimation)} \\ \frac{\kappa}{1-\kappa} \|x_n-x_{n-1}\| \hfill\null \; \textnormal{ (a-posteriori-estimation)} \end{matrix} \right.
\end{equation*}
whereby $x_n := \varphi(x_{n-1}) , n \in \mathbb{N}$. In particular, the sequence of iteratives $x_n$ converges to the fixed point.
\end{thm}
\begin{proof}
By induction on $n \in \mathbb{N}$ we obtain $\|x_n-x_{n-1}\| \le \kappa^n \|x_1-x_0\|$. It follows $\|x_{n+k}-x_n\| \le (\kappa^{n+k-1} + ... + \kappa^n) \|x_1-x_0\|  \le \frac{\kappa^n}{1-\kappa} \|x_1-x_0\|$ for all $k,n \in \mathbb{N}$. Because of $\kappa < 1$ this shows Cauchy-convergence, whence convergence according to Remark \ref{rem_completeness}b). Since $C$ is complete the limit $\xi$ is an element of $C$. Since $\varphi$ is continuous due to Example \ref{ex_norm}b) it holds $\varphi(\xi) = \xi$. For another fixed point $\tilde{\xi}$ we have $\|\tilde{\xi} - \xi\| = \|\varphi(\tilde{\xi})-\varphi(\xi)\| \le \kappa \|\tilde{\xi} - \xi\|$, hence $\tilde{\xi} = \xi$ because of $\kappa < 1$. The a-posteriori-estimation follows from $\|x_n-\xi\| \le \kappa \|x_{n-1}-\xi\| \le \kappa (\|x_{n-1}-x_n\| + \|x_n-\xi\|)$. Therefrom follows the a-priori-estimation by induction on $n$. The first inequality follows easily by induction on $n$, too.
\end{proof}

\begin{cor}\label{cor_Banach}
For a real quadratic matrix $A \in \mathbb{R}^{m \times m}$ of norm less than one and a real column vector $b \in \mathbb{R}^{m \times 1}$ the affine function $\varphi(x) := A x + b$ defines an iteration  $x_n := \varphi(x_{n-1}) , n \in \mathbb{N}$ that converges for every starting point $x_0$ to a point independent of $x_0$.
\end{cor}
\begin{proof}
For $x,y \in \mathbb{R}^m$ we have $\| \varphi(y)-\varphi(x) \| = \| A(y-x) \| \le \|A\| \|y-x\|$ due to compatibility of the matrix norm (s. Remark \ref{rem_matrixnorm}). So with $C:= \mathbb{R}^m$ and $\kappa := \|A\|$ the presuppositions of Theorem \ref{thm_Banach} are fulfilled.
\end{proof}

\begin{ex}\label{ex_iteration}
The function $\varphi(x,y):= y(-3/25,-41/50)+(1/5,11/10)$ defines a globally convergent iteration on $\mathbb{R}^2$. Why?
\end{ex}

In the final part of this subsection we recall some facts of differentiation theory. We restrict to the euclidean space for sake of simplicity.

\begin{defn}\label{def_diff}
A function $f:M \to \mathbb{R}$ is called \textit{differentiable at} $\xi \in M \subseteq \mathbb{R}$ when the \textit{derivative}
\begin{equation*}
f'(\xi) := \lim\limits_{x \to \xi} \frac{f(x)-f(\xi)}{x-\xi}
\end{equation*}
of $f$ in $\xi$ exists (s. Remark \ref{rem_convergence}c)!). It is called \textit{differentiable} when it is so at all points of $M$. Then the function $f':M \to \mathbb{R}$ might be differentiable again, and so on. By this way there can be defined recursively the $n$-\textit{th derivative} $f^{(n)} := (f^{(n-1)})', n \in \mathbb{N}$ with $f^{(0)}:=f$. A function $f:M \to \mathbb{R}$ is called \textit{partially differentiable at} $\xi =(\xi_1,...,\xi_n) \in M \subseteq \mathbb{R}^n$ when all the (\textit{partial}) derivatives $\partial f / \partial x_j (\xi)$ of the functions $x \mapsto f(\xi_1,...,\xi_{j-1},x,\xi_{j+1},...,\xi_n) , j \in \mathbb{N}_n$ at $\xi_j$ exist. The (row) vector $\nabla f(\xi) := \left(\partial f / \partial x_1 (\xi) , ... , \partial f / \partial x_n (\xi) \right)$ is called the \textit{gradient} of $f$ at $\xi$. A function $f=(f_1,...,f_m)^t:M \to \mathbb{R}^m$ is called \textit{partially differentiable at} $\xi =(\xi_1,...,\xi_n) \in M \subseteq \mathbb{R}^n$ when all its (real valued) \textit{components} $f_i , i \in \mathbb{N}_m$ are so. Then the matrix $\nabla f(\xi) := (\partial f_i / \partial x_j (\xi))_{i \in \mathbb{N}_m , j \in \mathbb{N}_n}$ with gradient $\nabla f_i(\xi)$ as $i$-th row is called the \textit{Jacobian} (\textit{matrix}) of $f$ at $\xi$. The matrix $\textnormal{H}f := (\frac{\partial^2 f}{\partial x_i \partial x_j})_{i,j \in \mathbb{N}_n} = \nabla(\nabla f)^t$ of the twofold partial derivatives is called the \textit{Hessian} (\textit{matrix}) of a real-valued function $f$.
\end{defn}

\begin{ex}\label{ex_Jacobian}
For $A \in \mathbb{R}^{m \times n} , b \in \mathbb{R}^{m \times 1}$ the affine function $f(x) := Ax+b$ is partially differentiable (overall in $\mathbb{R}^n$) with constant Jacobian $\nabla f = A$.
\end{ex}

\begin{defn}\label{def_localExtremum}
A point $x_0$ of an open set $U$ is called a \textit{local minimum point} or \textit{local maximum point} of a function $f:U \to \mathbb{R}$ when there is a $\delta > 0$ s.t. $f(x) \ge f(x_0)$ or $f(x) \le f(x_0)$, respectively, for all $x \in U$ with $\|x - x_0\| < \delta$. A local etremum point is called \textit{isolated} when its defining inequality is strict for $x \ne x_0$.
\end{defn}

The following three facts are fundamental for applications of differentiation theory. The first one is from Fermat, the second one from Cauchy, the third one from Lagrange.

\begin{prop}\label{prop_meanValue}
a) For a local extremum point $\xi$ of a differentiable function $f:]a,b[ \to \mathbb{R}$ it holds $f'(\xi)=0$.

b) A continuous function $f:[a,b] \to \mathbb{R}$ that is differentiable in $]a,b[$ there is some $\xi \in ]a,b[$ with $f(b)-f(a)=f'(\xi)(b-a)$.

c) For an $n$-times differentiable function $f:[a,b] \to \mathbb{R}$ and $x,x_0 \in [a,b]$ there is a number $\xi$ between $x$ and $x_0$ s.t.
\begin{equation*}
f(x) = \sum\limits_{k=0}^{n-1} \frac{f^{(k)}(x_0)}{k!} (x-x_0)^k + \frac{f^{(n)}(\xi)}{n!}(x-x_0)^n .
\end{equation*}
\end{prop}
\begin{proof}
a) Without loss of generality we assume $\xi$ being a local maximum point. It follows $(f(\xi + h)-f(\xi))/h \le 0 \textnormal{ or } \ge 0$ for $0 < h < \delta$ or $0 > h > - \delta$, respectively. This implies
\begin{equation*}
f'(\xi) = \lim\limits_{h \to 0} \frac{f(\xi + h)-f(\xi)}{h} = 0 .
\end{equation*}

b) The function $g(x):=(b-a)f(x)-(f(b)-f(a))x$ has two extremum points in $[a,b]$ due to Theorem \ref{thm_compact}. If they are equal to $a$ and $b$, respectively, $g$ is constant on $[a,b]$, hence $g'(x)=0$ for all $x \in [a,b]$. Otherwise there is some local extremum point $\xi \in ]a,b[$ of $g$. This implies $g'(\xi)=0$ due to a). Hence in any case there is some $\xi \in ]a,b[$ with $0 = (b-a)f'(\xi) - (f(b)-f(a))$.

c) For the function $r(x):=f(x)-p_{n}(x)$ with $p_n$ denoting the polynomial function in the formula it holds $r^{(k)}(x_0)=0$ for $k \in \lbrace 0,1,...,n-1 \rbrace$. With help of b) this implies that $\xi \mapsto (x-x_0)^{n-k} r^{(k)}(\xi) - r^{(k)}(x)(\xi-x_0)^{n-k}$ has a zero between $x$ and $x_0$. Since we may assume $x\ne x_0$ without loss of generality it follows
\begin{equation*}
\frac{r(x)}{(x-x_0)^n} = \frac{r'(\xi_1)}{n(\xi_1-x_0)^{n-1}} = ...  = \frac{r^{(n)}(\xi_n)}{n!}
\end{equation*}
for some $\xi_1,...,\xi_n$ between $x$ and $x_0$. Because of $p_n^{(n)} \equiv 0$ this shows the assertion with $\xi := \xi_n$.
\end{proof}

\begin{rem}\label{rem_differentiation}
a) A real valued function $f:M \to \mathbb{R}$ of one real variable $x \in M \subseteq \mathbb{R}$ that is differentiable (at $\xi \in M$) is continuous (at $\xi$) because
\begin{equation*}
0 = f'(\xi) \lim\limits_{x \to \xi}(x - \xi) =  \lim\limits_{x \to \xi} (f(x)-f(\xi)) =  \lim\limits_{x \to \xi} f(x) - f(\xi) .
\end{equation*}

b) From Proposition \ref{prop_meanValue}b) it can be derived that for a partially differentiable function $f : U \to \mathbb{R}^m$ on an open set $U \subseteq \mathbb{R}^n$ with continuity of $\nabla f$ at a point $\xi \in U$ it holds
\begin{equation}\label{eq_diff}
\lim\limits_{x \to \xi} \frac{\| f(x)-f(\xi)-\nabla f(\xi)(x-\xi) \|}{\| x-\xi \|} = 0 .
\end{equation}
In particular, then also $f$ is continuous at $\xi$.
\end{rem}

\begin{defn}
A function that fulfills Equation \ref{eq_diff} is called (\textit{totally}) \textit{differentiable} at $\xi$.
\end{defn}

The condition of continuous partial derivatives in Remark \ref{rem_differentiation}b) is not superfluous as shown by the following example.

\begin{ex}\label{ex_partDiff}
The function $f:\mathbb{R}^2 \to \mathbb{R}$ defined by $f(0,0):=0$ and
\begin{equation*}
f(x,y) := \frac{x y}{x^2+y^2} \textnormal{ for } (x,y) \ne (0,0)
\end{equation*}
is not continuous at $(0,0)$ as shown by the sequence $n \mapsto (1/n,1/n)$. But $f$ is partially differentiable. Show that $(0,0)$ is not a local extremum point in spite of $\nabla f(0,0)=(0,0)$ (s. the next Theorem!).
\end{ex}

\begin{thm}\label{thm_meanValue}
For a partially differentiable function $f : U \to \mathbb{R}$ on an open set $U \subseteq \mathbb{R}^n$ and a local extremum point $\xi \in U$ of $f$ it holds $\nabla f(\xi) = o^t$. If $f$ is continuous at two different points $a,b \in U$ with line segment $L := \lbrace (1-\lambda)a + \lambda b : 0 < \lambda < 1 \rbrace \subset U$ and if $f$ is totally differentiable on $L$ it holds
\begin{equation*}
f(b)-f(a)=\nabla f(\xi) (b-a)
\end{equation*}
for some $\xi \in L$.\footnote{This is called the 'mean value theorem'. For $n=1$, i.e. $U \subseteq \mathbb{R}$, it is the assertion of Proposition \ref{prop_meanValue}b) again.} If $f$ is even two times differentiable on $L$ it holds
\begin{equation*}
f(b) = f(a) + \nabla f(a)(b-a) + \frac{1}{2}(b-a)^t \textnormal{H}f(\xi) (b-a)
\end{equation*}
for some $\xi \in L$.\footnote{This formula yields a criterion on local extrema of $f$; s. Proposition \ref{prop_hessian}! For $U \subseteq \mathbb{R}$ it is the assertion of Proposition \ref{prop_meanValue}c) with $n:=2$.}
\end{thm}
\begin{proof}
The assertions follow from Proposition \ref{prop_meanValue}a),b),c), respectively.
\end{proof}

\subsection{Basic Algebra}\label{subsec_algebra}
We recall some basic algebraic notions and facts as usual in any elementary textbook about algebra, like e.g. \cite{vdWaerden1}.

\begin{defn}\label{def_equivClass}
A family of non-empty subsets $M_j, j \in J$ of a set $M$ with
\begin{equation*}
M = \underset{j\in J} \cup M_j \textnormal{ and } M_i \cap M_j = \emptyset
\end{equation*}
is called a \textit{disjoint union} of $M$. An $M_j$ is called an \textit{equivalence class} of $M$.\footnote{The union of the $M_j \times M_j$ is called an \textit{equivalence relation} of $M$.}
\end{defn}

\begin{ex}
For $n \in \mathbb{N}$ the sets $M_j := \lbrace j + m n : m \in \mathbb{N} \rbrace$ for $j \in \mathbb{N}_n$ define $n$ equivalence classes of $\mathbb{N}$. In case $n := 12$ they represent the hours of an analogue clock.
\end{ex}

\begin{defn}\label{def_group}
A function $f : G \times G \to G$ is called a \textit{group} (of set $G$) when
\begin{itemize}
\item it is \textit{associative}, i.e. $f \left(x , f(y,z)\right) = f \left(f(x,y),z \right)$ for all $x,y,z \in G$;
\item there is a \textit{neutral element} $e \in G$, i.e. $f(x,e) = x$ for all $x \in G$;
\item every $x \in G$ has an \textit{inverse} $y \in G$, i.e. $f(x,y)$ is neutral.
\end{itemize}
It is called \textit{commutative} or also \textit{abelian} when $f(x,y)=f(y,x)$ for all $x,y \in G$. If not being mistaken we use the notation $x y := f(x,y)$ and just write $G$ instead of $f : G \times G \to G$.
\end{defn}

\begin{prop}\label{prop_group}
In a group $G$ there is exactly one neutral element $e$. It holds $ex=x$ for all $x \in G$.\footnote{These two assertions follow from the first two group properties.} There is exactly one inverse $x^{-1} := y$ of $x \in G$. It fulfills the identity $x^{-1}x=e$. In case $x y = x$ or $x y = y$ it follows $y = e$ or $x = e$, respectively.
\end{prop}
\begin{proof}
Because of associativity we may omit brackets. Then for $x,y,z \in G$ and neutral elements $d,e \in G$ with $x y = d$ and $y z = e$ it follows $y x = y x e = y x y z = y d z = y z = e$, hence $d x = x y x = x e = x$ and, in particular, $e = d e = d$. For another inverse $\tilde{y}$ of $x$ it follows $\tilde{y} = e \tilde{y} = y x \tilde{y} = y e = y$. The identity $x y = x$ implies $y = e y = z x y = z x = e$ for the inverse $z$ of $x$. The other case of the last assertion follows analogously.
\end{proof}

\begin{ex}
Show that the third condition of the definition (about inverses) can not be formulated with $f(y,x)$ instead of $f(x,y)$. Hint: Consider the \textit{projection} $f(x , y) := x$.
\end{ex}

\begin{defn}\label{def_subgroup}
For a group $f : G \times G \to G$ and a subset $H$ of $G$ with $f(H \times H) \subseteq H$ the restriction $\tilde{f}$ of $f$ to $H \times H$ is called a \textit{subgroup} of $f$ when $\tilde{f} : H \times H \to H$ is a group. For a subgroup $H$ of $G$ and an element $g \in G$ the set $gH := \lbrace f(g,h) : h \in H \rbrace$ and $Hg := \lbrace f(h,g) : h \in H \rbrace$ is called a \textit{left coset} and \textit{right coset}, respectively, of $g$ with respect to $H$. A subgroup $H$ of $G$ is called \textit{normal} when $gH=Hg$ for all $g \in G$.
\end{defn}

\begin{prop}\label{prop_subgroup}
A non-empty set $H \subseteq G$ is a subgroup of a group $f : G \times G \to G$ if and only if $f(x,y^{-1}) \in H$ for all $x,y \in H$. Then the set of all left cosets defines a set of equivalence classes of $G$. The same holds for the right cosets. In case of a finite group $G$ all these cosets have the same number of elements, namely the number of elements of $H$. For a normal subgroup $H$ of $G$ the map $(gH,hH) \mapsto ghH$ defines a group of the set $G/H$ of left cosets, the so-called factor group of $G$ by $H$.
\end{prop}
\begin{proof}
If $H \subseteq G$ is a group then for every $y \in H$ the inverse $y^{-1} \in G$ is also an element of $H$. This shows already one direction of the first assertion. By assumption of the other direction the neutral element $e = x x^{-1}$ and the inverse $x^{-1} = e x^{-1}$ are in $H$ for all $x \in H$. In particular, we have also $x y = x (y^{-1})^{-1} \in H$ for all $x,y \in H$ by the assumption. This proves the other direction. For every $g \in G$ the function $h \mapsto g h$ defines a bijection between $H$ and the left coset $g H$ due to Proposition \ref{prop_group}. This proves the last assertion. For $h_1,h_2 \in H$ with $g_1 h_1 = g_2 h_2$ it holds $g_1 = g_2 h_2 h_1^{-1}$. This implies $g_1 H = g_2 H$ by the latter bijection. So different left cosets are even disjoint. And that the union of the left cosets is $G$ follows from $g \in g H$ for every $g \in G$ since $H$ contains the neutral element. An analogous argumentation holds for right cosets. For proving the last assertion we presuppose $gH=\tilde{g}H , hH=\tilde{h}H$ for $g,\tilde{g},h,\tilde{h} \in G$. Then, by assumption, for every $k \in H$ there are $l,m,n,p \in H$ s.t.$g h k = g \tilde{h} l = g m \tilde{h} = \tilde{g} n \tilde{h} = \tilde{g} \tilde{h} p$. This shows $ghH \subseteq \tilde{g}\tilde{h}H$. By symmetry of argumentation the other inclusion follows also. Hence the map is well-defined. That $G/H$ is a group with neutral element $H$ and inverse element $g^{-1}H$ of $gH$ is easy to verify.
\end{proof}

\begin{rem}\label{rem_subgroup}
A consequence of this proposition is: In case of a group $G$ with a finite number $|G|$ of elements the number of left cosets equals the number of right cosets and equals $|G| / |H|$ for every subgroup $H$ of $G$. It is called the \textit{index} of $H$ in $G$. In particular: $|H|$ divides $|G|$, and in case of a normal subgroup $H$ of $G$ the index equals $|G/H|$.
\end{rem}

\begin{ex}\label{ex_cyclic}
In a finite group $G$ every element $g \in G$ \textit{generates} a \textit{cyclic} subgroup $\lbrace g^n : n \in \mathbb{N} \rbrace$ of $G$ whereby $g^n := g g ... g$. Its number $\textnormal{ord}(g)$ of elements equals the smallest $n \in \mathbb{N}$ s.t. $g^n=e$ is the neutral element $e$. The \textit{order} $\textnormal{ord}(g)$ of $g$ divides the \textit{order} $|G|$ of $G$ and every other $n \in \mathbb{N}$ with $g^n=e$. Especially, it holds $g^{|G|}=e$ for all $g \in G$.
\end{ex}

\begin{rem}\label{rem_cyclic_criterion}
A sufficient condition for a finite group $G$ being cyclic is that for all divisors $n \in \mathbb{N}$ of $|G|$ the equation $g^n=e$ has at most $n$ solutions $g \in G$. For a proof s. \cite{Lorenz}, ch.9, sect.3, p.100/101!
\end{rem}

The left (or right) cosets are equivalence classes of the underlying group. This will turn out by help of the following general notion. (S. Example \ref{ex_subgroupAction})

\begin{defn}\label{def_action}
A group $G$ with neutral element $e$ \textit{acts} on a set $M$ \textit{from left} when there is a function $a : G \times M \to M$ whose values $g m := a(g,m)$ fulfill
\begin{itemize}
\item $e m = m$ for all $m \in M$ ,
\item $(g h) m = g (h m)$ for all $g,h \in G$ and $m \in M$ .
\end{itemize}
The set $G m := \lbrace g m : g \in G \rbrace$ is called the \textit{orbit} and $G_m := \lbrace g \in G : g m = m \rbrace$ the \textit{fix group} of $m \in M$. Analogous is the \textit{group action from right} with \textit{orbits} $m G$ and \textit{fix groups} ${_m}G$.
\end{defn}

\begin{rem}\label{rem_action}
a) The orbits are equivalence classes of $M$ because their union is obviously $M$ and $g m = h n$ implies $m = g^{-1} h n$, hence $G m = G n$, for $g,h \in G , m,n \in M$.

b) According to Proposition \ref{prop_subgroup} a fix group is indeed a group since $g m = m$ implies $g^{-1}m = g^{-1} (g m) = (g^{-1}g) m = m$ for all $g \in G , m \in M$.

c) For a fixed $m \in M$ the identity $g m = h m$ is equivalent with $g^{-1} h \in G_m$. By Remark b) and Proposition \ref{prop_subgroup} the latter is equivalent with the identity $g G_m = h G_m$ of left cosets. This shows that $g m \mapsto g G_m$ is a well-defined bijection from the orbit $G m$ onto the set of left cosets of $G_m$. Hence, in case of finiteness, the number of elements of an orbit $Gm$ equals the index of the fix group $G_m$ in $G$.\footnote{Together with Remark a) this yields a \textit{class formula} for the number of elements of $M$.}
\end{rem}

\begin{ex}\label{ex_subgroupAction}
Every subgroup $H$ of a group $G$, written multiplicatively, acts on $G$ from left by $a(h,g) := h g$ and from right by $a(g,h) := g h$ for $g \in G, h\in H$. Its orbits are the right cosets and left cosets, respectively.
\end{ex}

\begin{defn}\label{def_homomorphism}
For two groups $f : G \times G \to G$ and $g : H \times H \to H$ a function $h : G \to H$ is called a (\textit{group}) \textit{homomorphism} when $h(f(x,y)) = g(h(x),h(y))$ for all $x,y \in G$. A bijective homomorphism $h : G \to H$ is called \textit{isomorphism}. Then $G$ and $H$ are called \textit{isomorphic}. An injective homomorphism is called \textit{monomorphism}. A surjective homomorphism is called \textit{epimorphism}.
\end{defn}

\begin{ex}
a) The natural logarithm $\ln : \mathbb{R}^{+} \to \mathbb{R}$ is an isomorphism with respect to multiplication in $ \mathbb{R}^{+}$ and addition in $\mathbb{R}$.

b) For a normal subgroup $H$ of a group $G$ the \textit{canonical projection} $\pi : G \to G/H$ defined by $\pi(g) := gH$ is a surjective homomorphism with preimage $\pi^{-1}(H)=H$ of $H$.
\end{ex}

\begin{prop}\label{prop_homomorphism}
For a group homomorphism $h : G \to H$ and a subgroup $F$ of $G$ or $H$ the image $h(F)$ or preimage $h^{-1}(F)$ is a subgroup of $H$ or $G$, respectively. A homomorphism is injective if the preimage $\ker h := h^{-1}(e)$ of the neutral element $e \in H$ consists of the neutral element of $G$ only. For a normal subgroup $F$ of $H$ the group $h^{-1}(F)$ is normal in $G$. The preimage map $z \mapsto h^{-1}(z)$ defines an isomorphism between $h(G)$ and $G/\ker h$. Its inverse maps $x \ker h$ to $h(x)$ for all $x \in G$.
\end{prop}
\begin{proof}
The first assertion is easy to verify. For proof of the second assertion we assume $h(x)=h(y)$ for $x,y \in G$ and a homomorphism $h : G \to H$ between groups $f : G \times G \to G$ and $g : H \times H \to H$. Then $h(f(x,y^{-1})) = g(h(x),h(y^{-1})) = g(h(x),h(y)^{-1})$ is the neutral element of $H$. Therefore, by assumption $f(x,y^{-1})$ is the neutral element of $G$. This implies $x = y$ due to Proposition \ref{prop_group}, whence injectivity. For the last assertion we assume $h(x) \in F$ for some $x \in G$. Then it holds $h(yxy^{-1}) = h(y)h(x)h(y)^{-1} \in h(y) F h(y)^{-1} = F$, i.e. $yxy^{-1} \in h^{-1}(F)$, for all $y \in G$. Hereby we use $x y$ as composition of $x,y \in G$ or $H$. Thus it is proven $y h^{-1}(F) = h^{-1}(F) y$ for all $y \in G$, i.e. normality of $h^{-1}(F)$. Since $\lbrace e \rbrace$ is a normal subgroup of $H$ the factor group $G/\ker h$ exists. By definition of $\ker h$ it holds $h(x \ker h) = h(x)$ for all $x \in G$. If $x \ker h = y \ker h$ for $x,y \in G$ it holds $ x = y z$ for some $z \in \ker h$, hence $h(x) = h(y)h(z) = h(y)$. So, the map $x \ker h \mapsto h(x)$ is well-defined. It is a homorphism since $h$ is so. Since $\ker h$ is the neutral element of $G/\ker h$ it is injective. Because of $x \ker h = h^{-1}(h(x))$ this finishes the proof.
\end{proof}

\begin{rem}\label{rem_cyclic}
For $k,l \in \mathbb{N}_n$ we set $f(k,l) := k+l$ in case $k+l \le n$ and $f(k,l) := k+l-n$ otherwise. Then $f : \mathbb{N}_n \times \mathbb{N}_n \to \mathbb{N}_n$ is a group with neutral element $n$. It is cyclic with generator $1$. Thus there exists a cyclic group with $n \in \mathbb{N}$ elements. For another group $G = \lbrace g , g^2 , ... , g^n \rbrace$ with $n$ elements the function $h(k) := g^k$ defines a homomorphism $h : \mathbb{N}_n \to G$ because $h(n)=g^n$ is the neutral element of $G$ and therefore $h(f(k,l)) = g^{k+l} = g^k g^l$ for all $k,l \in \mathbb{N}_n$. Since the powers $g^k \in G$ are pairwise different the preimage of $g^n$ is only $n$. Hence $h$ is injective. Surjectivity is clear by definition of $h$. Thus, up to isomorphy, there is only one cyclic group $C_n$ with $n$ elements. The subgroups of $C_n$ are, up to isomorphy, all the $C_m$ for the divisors $m$ of $n$. Because in case $n = k m$ for $k,m \in \mathbb{N}_n$ the element $h := g^k$ generates a subgroup of $G$ with $m$ elements. Due to Proposition \ref{prop_subgroup} this shows the assertion.
\end{rem}

\begin{ex}
Every group of prime $p$ elements is isomorphic to $C_p$.
\end{ex}

\begin{rem}\label{rem_groupProduct}
For two groups $G,H$ the operation $(g,h)(g',h'):=(gg',hh')$ defines a group of the cartesian product $G \times H$ of set $G$ with set $H$.
\end{rem}

\begin{defn}\label{def_groupProduct}
The group of $G \times H$ like in this remark is called the \textit{direct product} of $G$ with $H$ and briefly denoted by $G \times H$.
\end{defn}

\begin{ex}
With neutral element $0$ and the other element $1$, called \textit{bits}, of $C_2 = \lbrace 0 , 1 \rbrace$ the eight-fold direct product $C_2^8 := C_2 \times ... \times C_2$ of $C_2$ is a group whose elements are called \textit{bytes}. Its neutral element is the \textit{zero-byte} $(0,0,0,0,0,0,0,0)$. Its group operation is called \textit{exclusive-or} with the property $1+1 = 0$ in each coordinate.
\end{ex}

\begin{defn}\label{def_ring}
A function $(f,g) : K \times K \to K \times K$ is called a (\textit{commutative}) \textit{field} when $f$ is a commutative group, $g$ restricted to $K^{\times} \times K^{\times}$ is also a commutative group and for all $x,y,z \in K$ we have \textit{distributivity} $g \left(f(x,y),z\right) = f\left(g(x,z),g(y,z)\right)$. Hereby we set $K^{\times} := K \setminus \lbrace 0 \rbrace$ where $0$ denotes the neutral element \textit{zero} of $f$. Then $f$ is called \textit{addition} and $g$ \textit{multiplication}. We write $x+y := f(x,y)$ and $x y := g(x,y)$.\footnote{So the axiom of distributivity reads $(x+y)z = x z + y z$, following the convention that multiplication precedes addition.} If all axioms of a field are fulfilled but the existence of inverses with respect to multiplication we call $(f,g)$ a \textit{commutative ring}. When its multiplication is not commutative but $g \left(z,f(x,y)\right) = f\left(g(z,x),g(z,y)\right)$ for all $x,y,z$ we just call it a (\textit{non-commutative}) \textit{ring}. Its neutral element $1$ with respect to multiplication is called \textit{one}. By abuse of notation we denote a ring by its set symbol $R$. An element $x \in R$ \textit{divides} an element $y \in R$ when there is some $z \in R$ with $x z = y$. Then $x$ is called a (\textit{left}) \textit{divisor} of $y$. In case an element divides $1$ it is called a \textit{unit}.\footnote{Remark \ref{rem_ring}a) will show that a unit is also a \textit{right} divisor of $1$.} An element $x \ne 0$ is called a \textit{zero divisor} when there is some $y \in R \setminus \lbrace 0 \rbrace$ with $x y = 0$ or $y x = 0$. A commutative ring is called an \textit{integral domain} when it does not have any zero divisors. A function $h$ from a ring $R$ to another ring is called a (\textit{ring}) \textit{homomorphism} when $h(x+y)=h(x)+h(y)$ and $h(x y) = h(x) h(y)$ for all $x,y \in R$. A bijective homomorphism is called an \textit{isomorphism}. 
\end{defn}

\begin{ex}\label{ex_intDomain}
a) The usual addition and multiplication in $\mathbb{Z}$ makes $\mathbb{Z}$ to an integral domain with neutral elements $0$ and $1$, respectively. Why is it not a field? The determinant function induces a group epimorphism from $\textnormal{GL}_n(\mathbb{Z})$ to $\lbrace \pm 1 \rbrace$. Hence $\textnormal{SL}_n(\mathbb{Z})$ has index two in $\textnormal{GL}_n(\mathbb{Z})$ according to Proposition \ref{prop_homomorphism}.

b) For the field $\mathbb{Q}$ of rational numbers (constructed out of $\mathbb{Z}$ via \textit{quotients}; s. Proposition \ref{prop_quotField}) the set of fundamental sequences $x:\mathbb{N} \to \mathbb{Q}$ (s. Definition \ref{def_norm}) becomes an integral domain with addition $(x+y)_n := x_n + y_n$ and multiplication $(x y)_n := x_n y_n$. Its zero and one element is the constant zero and one sequence, respectively.

c) For a ring $R$ the set $R^{n \times n}$ of all $n \times n$-matrices becomes a ring via matrix addition and multiplication. Since $R^{2 \times 2}$ is non-commutative (s. Example \ref{ex_nonCommutativity}) it is not an integral domain. Furthermore it has zero divisors:
\begin{equation*}
\left(\begin{matrix} 0 &\ 1 \\ 0 &\ 0 \end{matrix}\right) \left(\begin{matrix} 1 &\ 0 \\ 0 &\ 0 \end{matrix}\right) = \left(\begin{matrix} 0 &\ 0 \\ 0 &\ 0 \end{matrix}\right) .
\end{equation*}
This example shows that $x y = 0$ for ring elements $x , y$ does not imply $y x = 0$.
\end{ex}

\begin{rem}\label{rem_ring}
a) In a ring $R$ it holds $0 x = 0 = x 0$ for all $x \in R$ because for an additive inverse $-y$ of $y \in R$ it holds $-y x = -(y x)$ and therefore $0 x = (y - y) x = y x - y x = 0$. In a ring $x y = 1$ implies $y x = 1$. This follows by the same argumentation as in the proof of Proposition \ref{prop_group}.

b) A field $\mathbb{K}$ is an integral domain since $x y = 0$ for some $x \in K , y \in K^{\times}$ implies $x = x y y^{-1} = 0$ due to Remark a). And every non-zero element of $\mathbb{K}$ is a unit.

c) A function from a field $\mathbb{K}$ to a ring is a homomorphism if and only if it is the zero function or a group homomorphism from $\mathbb{K}$ with respect to addition and from $\mathbb{K}^{\times}$ with respect to multiplication. In the latter case the image of the homomorphism is a field isomorphic to $\mathbb{K}$. This follows from Proposition \ref{prop_homomorphism} and the fact that a non-zero homomorphism $h$ from a field $\mathbb{K}$ maps every $x \in \mathbb{K}^{\times}$ to a non-zero element since otherwise $h(x y) = h(x) h(y) = 0$ for all $y \in \mathbb{K}$ due to Remark a).

d) The set $R^{\times}$ of units of a ring $R$, e.g. $\textnormal{GL}_n(\mathbb{R}) = (R^{n \times n})^{\times}$ (s. Proposition \ref{prop_GL}), represents a group with respect to multiplication. None of its elements is a zero divisor. Because of uniqueness of inverses (s. Proposition \ref{prop_group}) one writes $1/r$ or $r^{-1}$ for an element $\varepsilon \in R$ s.t. $r \varepsilon = 1$.

e) For a ring homomorphism $h : R \to S$ the image set $h(R^{\times})$ is a subgroup of $S^{\times}$. To see this first realise $h(1) = 1$ since $h(1)h(r) = h(1 r) = h(r)$ for $r \in R$. And for $e \in R^{\times}$ there is some $e' \in R^{\times}$ with $e e' = 1$ whence $h(e)h(e') = h(e e') = h(1) = 1$. This shows $h(e) \in S^{\times}$. And that $h(R^{\times})$ is a group follows from Proposition \ref{prop_homomorphism}.

f) For a ring homomorphism $h : R \to R$ (called \textit{endomorphic}) of a commutative ring $R$ the function $N(x) := x h(x)$ fulfills
\begin{equation*}
N(x y) = x y h(x) h(y) = N(x) N(y)
\end{equation*}
for all $x , y \in R$. The argumentation of Remark e) shows $N(R^{\times}) \subset R^{\times}$. And every $x \in R$ with $N(x) = N(u)$ for some $u \in R^{\times}$ is also a unit. Because $x h(x) = u h(u)$ implies $x h(x) h(u)^{-1} u^{-1} = 1$. So we have $R^{\times} = N^{-1}(N(R^{\times}))$.

g) For two rings $R,S$ the two operations $(r,s)+(r',s'):=(r+r',s+s')$ and $(r,s)(r',s'):=(rr',ss')$ define a ring of the cartesian product $R \times S$.
\end{rem}

\begin{defn}\label{def_ringProduct}
The latter ring $R \times S$ is called the \textit{direct product} of ring $R$ with ring $S$.
\end{defn}

\begin{ex}\label{ex_ringProduct}
For two rings $R,S$ it holds $(R \times S)^{\times} = R^{\times} \times S^{\times}$.
\end{ex}

\begin{defn}\label{def_module}
For a ring $(R,+,\cdot)$ and a commutative group $(M,+)$ a function $(r,m) \mapsto rm$ from $R \times M$ to $M$ is called a \textit{module over} $R$ or an $R$-\textit{module} (symbolised by $M$) when it holds
\begin{equation*} (r+s)m=rm+sm , r(m+n)=rm+rn,r(sm)=(rs)m,1m=m \end{equation*} for all $r,s \in R , m,n \in M$.\footnote{When $M$ is even a ring with additional property $r(mn)=(rm)n$ for all $r \in R, m,n \in M$ then it is called an \textit{algebra}.} In case $R$ is a field $M$ is called an $R$-\textit{vectorspace}. A group homomorphism $l:M \to N$ between $R$-modules $M,N$ is called \textit{linear} when $l(\lambda m) = \lambda l(m)$ for all $m \in M , \lambda \in R$. For $R$-modules $M,N$ a map $M \times ... \times M \to N$ is called \textit{multilinear} when it is linear in each variable. A multlinear function $M \times M \to N$ is also called \textit{bilinear}. In case of $N = R$ a linear/bilinear/multilinear function is called a \textit{linear/bilinear/multilinear form}. A subset of an $R$-module $M$ is called an $R$-\textit{submodule of} $M$ when it is an $R$-module. When the ring $R$ is clear from context we just say \textit{submodule}. Analogously a \textit{subspace} of a vectorspace is defined.
\end{defn}

\begin{ex}\label{ex_module}
a) Every ring is a module over itself.

b) With $(x_1,...,x_n)+(y_1,...,y_n):=(x_1+y_1,...,x_n+y_n)$ and $\lambda (x_1,...,x_n) := (\lambda x_1,...,\lambda x_n)$ for $\lambda,x_i,y_i \in R$ the set $R^n$ of $n$-tuples becomes a module over any ring $R$. For $A \in R^{m \times n}$ the function $x \mapsto A x , x \in R^{n \times 1}$ defines a linear map from $R^n$ to $R^m$. 
\end{ex}

The following fact shows that $\mathbb{Q}$, constructed in the usual way from the integral domain $\mathbb{Z}$, is a field that contains $\mathbb{Z}$.

\begin{prop}\label{prop_quotField}
For an integral domain $\mathbb{O}$ the set $\mathbb{K}$ of sets $a/b := \lbrace (c,d) \in \mathbb{O} \times \mathbb{O} : a d = b c , d \ne 0 \rbrace$ for $a,b \in \mathbb{O}$ with $b \ne 0$ becomes a field with addition $a/b + c/d := (a d + b c)/(b d)$ and multiplication $a/b \cdot c/d := (a c)/(b d)$. The function $a \mapsto a/1$ defines an injective homomorphism $\mathbb{O} \to \mathbb{K}$.
\end{prop}
\begin{proof}
The equation $a/c = b/d$ implies $ad=bc$ because of $cd=dc$. This shows already the last assertion. And vice versa: $a d = b c$ for $b,d \ne 0$ implies $a/c = b/d$. For proving this we suppose additionly $a \tilde d = b \tilde c$. Then it follows $b (c \tilde{d} - d \tilde{c}) = a d \tilde{d} - d a \tilde{d} = 0$. This implies $c \tilde{d} = d \tilde{c}$ because there are no zero divisors. So we have shown $a/c \subseteq b/d$. The other inclusion follows analogously. Now, the well-definition of addition and multiplication and the other assertion follow easily.
\end{proof}

\begin{defn}\label{def_quotField}
Field $\mathbb{K}$ in the Proposition is called the \textit{quotient field} of $\mathbb{O}$.
\end{defn}

\begin{ex}\label{ex_polynomials}
For an integral domain $\mathbb{O}$ the set $\mathbb{O}[x]$ of all \textit{polynomials} $a:\mathbb{N}_0 \to \mathbb{O}$, i.e. $a(n)=0$ for all but finitely many $n \in \mathbb{N}$, is an integral domain with addition $(a+b)(n):=a(n)+b(n)$ and multiplication (as \textit{convolution} of sequences)\footnote{The zero-element is the zero function $(0 , 0 , ...)$ and the one-element $(1 , 0 , 0 , ...)$.}
\begin{equation*}
(a b)(n) := \sum\limits_{j=0}^{n} a(j)b(n-j) .
\end{equation*}
For the quotient field $\mathbb{K}$ of $\mathbb{O}$ the quotient field $\mathbb{K}(x)$ of $\mathbb{K}[x]$ is isomorphic to the quotient field of $\mathbb{O}[x]$. It is called the field of \textit{rational functions} (with coefficients in $\mathbb{K}$). In general, its elements must not be confused with functions
\begin{equation*}
f_{a,b}(x):=\frac{a_0 + a_1 x + ... + a_n x^n}{b_0 + b_1 x + ... + b_n x^n}
\end{equation*}
where $a=(a_k)_{k \in \mathbb{N}_0}, b=(b_k)_{k \in \mathbb{N}_0} \in \mathbb{K}[x]$ with $a_k=b_k=0$ for $k > n \in \mathbb{N}_0$ and where the polynomial function in the denominator is not the zero-function.\footnote{Then $f$ is defined for all bot finitely many $x \in \mathbb{K}$.} But for infinite integral domains $\mathbb{O}$ the function $a/b \mapsto f_{a,b}$ defines an isomorphism between the rational functions and those functions (with the usual addition and multiplication). Hereby, $\mathbb{O}[x]$ is mapped onto the integral domain of polynomial functions with coefficients in $\mathbb{O}$. To see this one has to realise that a non-zero polynomial function over an integral domain has only finitely many zeroes.
\end{ex}

\begin{defn}\label{def_ideal}
An additive subgroup $I$ of a commutative Ring $R$ (with one) is called an \textit{ideal} when $R I := \lbrace \alpha \beta : \alpha \in R, \beta \in I \rbrace \subseteq I$.\footnote{The fundamental concept of an ideal stems from R. Dedekind (1831-1916). It is of special importance for number theory.} We denote by $(\beta_1,...,\beta_k) := \lbrace \alpha_1 \beta_1 + ... + \alpha_k \beta_k : \alpha_i \in R \rbrace$ the ideal \textit{generated by} the elements $\beta_i \in R , i \in \mathbb{N}_k$. An ideal $(\beta)$ generated by a single element $\beta \in R$ is called \textit{principal}. An integral domain is called a \textit{principal ideal domain} when all its ideals are principal. An element $\delta \ne 0$ of a commutative Ring $R$ is called a \textit{greatest common divisor} of given elements of $R$ when every common divisor of those elements divides $\delta$. In case $\delta$ is a unit those elements are called \textit{coprime}.
\end{defn}

\begin{rem}\label{rem_ideal}
a) Every field $\mathbb{K}$ is a principal ideal domain with $(0)$ and $(1)$ as its only ideals. That is clear since every $\kappa \in \mathbb{K} \setminus \lbrace 0 \rbrace$ is a unit of $\mathbb{K}$, implying $(\kappa) = (1)$.

b) Two greatest common divisors $\delta,\delta'$ differ only by a unit as a factor because $\delta = \alpha \delta' = \alpha \beta \delta$ for some $\alpha,\beta$ imply $\alpha \beta = 1$. Hence for coprime elements every common divisor is a unit.

c) For an ideal $I$ of a commutative ring $R$ the set $R/I$ of equivalence classes $\alpha + I := \lbrace \alpha + \beta : \beta \in I \rbrace$ is a commutative ring with respect to addition $(\alpha + I) + (\beta + I) := (\alpha + \beta)+I$ and multiplication $(\alpha + I)(\beta + I) := (\alpha \beta)+I$. It is called the \textit{factor ring} or \textit{quotient ring} of $R$ by $I$. (cf. Proposition \ref{prop_subgroup}) The map $\alpha \mapsto \alpha + I$ defines a ring epimorphism $R \to R/I$, called the \textit{canonical projection}. The notation $\alpha \equiv \beta \mod I$ means $\alpha + I = \beta + I$, i.e. $\alpha - \beta \in I$. In case $I = (\gamma)$ we just write $\alpha \equiv \beta \mod \gamma$ which means that $\gamma$ is a divisor of $\alpha - \beta$.

d) For a homomorphism $h:R \to S$ between commutative rings $R,S$ the \textit{kernel} $\ker h := \lbrace \alpha \in R : h(\alpha)=0 \rbrace$ is an ideal of $R$. The map $\alpha + \ker h \mapsto h(\alpha)$ defines an isomorphism between $R / \ker h$ and $h(R)$. (cf. Proposition \ref{prop_homomorphism})
\end{rem}

\begin{ex}\label{ex_ideal}
a) The \textit{sum} $I+J:= \lbrace \alpha+\beta : \alpha \in I , \beta \in J \rbrace$, the \textit{product}
\begin{equation*}
I J := \left\lbrace \sum\limits_{k=1}^{n} \alpha_k \beta_k : n \in \mathbb{N} , \alpha_k \in I , \beta_k \in J \right\rbrace
\end{equation*}
and the intersection $I \cap J$ \textit{of ideals} $I,J$ are also ideals with
\begin{equation}\label{eq_ideals}
(I \cap J) (I+ J) \subseteq I J \subseteq I \cap J .
\end{equation}
If all these ideals are principal it holds even $(I \cap J)(I+ J) = I J$. That the condition is not superfluous is shown by the example $I:=(x),J:=(2) \subset \mathbb{Z}[x]$. Show also $(\alpha,\beta) =(\alpha)+(\beta)$ for elements $\alpha,\beta$ of a commutative ring and that $(a) \subseteq (b)$ implies the existence of an element $c$ with $(a) = (b)(c)$ in a principal ideal domain.

b) For the integral domain $R$ of fundamental sequences of rational numbers (s. Example \ref{ex_intDomain}b)) the set $I$ of rational zero sequences is an ideal of $R$. The quotient ring $\mathbb{R} := R / I$ is even a field. Its elements are called \textit{real numbers}. A real number $\rho := q+I$ ($q \in R$) is called \textit{positive}, in symbols: $\rho > 0$, when there is a positive rational number $\varepsilon$ s.t. $q_n < \varepsilon$ for atmost finitely many $n \in \mathbb{N}$. Show that this property does not depend on the choice of  the fundamental sequence $q$. In case $\rho \ne 0$ is not positive we call it \textit{negative}, in symbols $\rho < 0$. In case $\rho = 0$ or $\rho > 0$ one writes $\rho \ge 0$ (analogously: $\rho \le 0$). Show that the function
\begin{equation*}
|\rho| := \left \lbrace \begin{matrix}\rho &\ \textnormal{ in case } \rho \ge 0 \\ -\rho &\ \textnormal{ otherwise}\end{matrix} \right.
\end{equation*}
defines a modulus function $\mathbb{R} \to \mathbb{R}_0^{+}$, i.e. it is non-negative and fulfills the three norm properties of Definition \ref{def_norm}.
\end{ex}

The following will be referred to as the \textit{Theorem of B\'ezout}.

\begin{thm}\label{thm_bezout}
For elements $\alpha,\beta,\gamma,\delta \in R$ of a commutative ring $R$ with $(\alpha)+(\beta)=(\delta)$ the element $\gamma \delta$ is a greatest common divisor of $\alpha \gamma$ and $\beta \gamma$, and every greatest common divisor of $\alpha,\beta$ equals $\delta$ up to a unit factor. An analogue statement holds for a finite sum of principal ideals.
\end{thm}
\begin{proof}
The assumption implies $(\alpha),(\beta) \subseteq (\delta)$, hence $\alpha = \lambda \delta, \beta = \mu \delta$ for some $\lambda,\mu \in R$. So $\delta$ is a common divisor of $\alpha,\beta$. Because of $\delta \in (\alpha)+(\beta)$ another divisor of $\alpha,\beta$ divides $\delta$. Therefore $\delta$ is even a greatest common divisor of $\alpha,\beta$. By multiplying the given equation by $(\gamma)$ we obtain $(\alpha \gamma)+(\beta \gamma)=(\gamma \delta)$. Then the first assertion follows by same reasoning. The second assertion is due to Remark \ref{rem_ideal}b). The last assertion follows by induction on the number of summands.
\end{proof}

\begin{cor}
For coprime elements $\alpha,\beta$ of a principal ideal domain $\mathbb{O}$ and an element $\gamma \in \mathbb{O}$ s.t. $\alpha$ divides $\beta \gamma$ the element $\alpha$ must divide already $\gamma$.
\end{cor}
\begin{proof}
Due to the theorem there are $\lambda,\mu \in \mathbb{O}$ s.t. $\lambda \alpha + \mu \beta = 1$, hence $\gamma = \lambda \gamma \alpha + \mu \beta \gamma$. This shows the assertion.
\end{proof}

The following proposition deals with so-called "euclidean domains".

\begin{prop}\label{prop_euclideanDomain}
An integral domain $\mathbb{O}$ with a function $d : \mathbb{O} \setminus \lbrace 0 \rbrace \to \mathbb{N}_0$ s.t. for all $\alpha,\beta \in \mathbb{O} \setminus \lbrace 0 \rbrace$ there are $\gamma,\delta \in \mathbb{O}$ with $\alpha = \beta \gamma + \delta$ and $\delta=0$ or $d(\delta)<d(\beta)$ is principal. More precise: A non-zero ideal of such an integral domain is generated by an element with minimal value of $d$.
\end{prop}
\begin{proof}
For an ideal $I \ne (0)$ of $\mathbb{O}$ choose an element $\beta \in I \setminus \lbrace 0 \rbrace$ with minimal $d(\beta)$. For every $\alpha \in I$ there are $\gamma,\delta \in \mathbb{O}$ like in the assertion. Since $\delta = \alpha - \beta \gamma \in I$ it must hold $\delta = 0$ because of the minimality. This shows $I = (\beta)$.
\end{proof}

\begin{ex}\label{ex_euclideanDomain}
a) The most popular principal ideal domain is $\mathbb{Z}$. This can be shown fairly easy by help of integer division with remainder whereby $d$ in Proposition \ref{prop_euclideanDomain} is chosen as the absolute value function. According to B\'ezout's theorem it holds  $(\beta_1,...,\beta_k) = (\textnormal{gcd}(\beta_1,...,\beta_k))$ for integers  $\beta_1,...,\beta_k$ not all zero. Hereby 'gcd' means the greatest natural number dividing the given arguments.

b) For a field $\mathbb{K}$ the \textit{degree} $d(a) := \max \lbrace n \in \mathbb{N}_0 : a_n \ne 0 \rbrace$ of a non-zero polynomial $a =(a_n)_n \in \mathbb{K}[x]$ (s. Example \ref{ex_polynomials}) can be used for polynomial division, thus fulfilling the presuppositions of Proposition \ref{prop_euclideanDomain}. So $\mathbb{K}[x]$ is a principal ideal domain, and it follows that a polynomial equation
\begin{equation*}
a_n x^n + a_{n-1} x^{n-1} + ... + a_1 x + a_0 = 0
\end{equation*}
of degree $n \in \mathbb{N}_0$ in $x \in \mathbb{K}$ has at most $n$ solutions. From Remark \ref{rem_cyclic_criterion} it follows that for a finite field $\mathbb{K}$, like e.g. $\mathbb{K} = \mathbb{Z} / p \mathbb{Z}$ for a prime number $p \in \mathbb{Z}$, the multiplicative group $\mathbb{K} \setminus \lbrace 0 \rbrace$ is cyclic.

c) Also the ring $\mathbb{Z}[i] = \lbrace x + i y \, | \, x,y \in \mathbb{Z} \rbrace$ of 'Gaussian integers' is a principal ideal domain. This can be shown analogously to Example a) whereby the remainder $\delta$ of the integer division of $\alpha$ by $\beta$ fulfills $|\delta| \le |\beta| / \sqrt{2}$.
\end{ex}

The following will be referred to as the \textit{Chinese Remainder Theorem}.

\begin{thm}\label{thm_crt}
For ideals $I,J$ of a commutative ring $R$ with $I+J=R$ the function $h(\alpha):=(\alpha+I,\alpha+J)$ defines a surjective homomorphism $h : R \to R/I \times R/J$ with kernel $IJ$. In particular, $R/(IJ)$ is isomorphic to $R/I \times R/J$. An analogue statement holds for a finite product of quotient rings.
\end{thm}
\begin{proof}
According to Remark \ref{rem_ideal}c) $h$ is a homomorphism of commutative rings. By assumption there are $\alpha \in I , \beta \in J$ with $\alpha + \beta = 1$. Then for all $\gamma,\delta \in R$ it holds $\gamma \alpha + \delta \beta - \delta = \gamma \alpha - \delta \alpha = (\gamma - \delta) \alpha \in I$ and, analogously, $\gamma \alpha + \delta \beta - \gamma \in J$. That means $\gamma \alpha + \delta \beta \in (\delta + I) \cap (\gamma + J)$, hence $h(\gamma \alpha + \delta \beta) = (\delta + I , \gamma + J)$. This shows surjectivity. Because of Equation \eqref{eq_ideals} it holds $\ker h = I \cap J = I J$. So the second assertion follows from Remark \ref{rem_ideal}d). The last assertion follows by induction on the number of factors.
\end{proof}

\begin{ex}\label{ex_crt}
For coprime numbers $p,q \in \mathbb{N}$ the class $p + (q)$ is a unit in $\mathbb{Z}/(q)$ due to B\'ezout's theorem. When we denote by $p^{-1} \mod{q}$ an element of its inverse then for all $a,b \in \mathbb{Z}$ the integer $c := a + p \left((b-a) p^{-1} \mod{q} \right)$ fulfills $c + (p) = a +(p)$ and $c + (q) = b + (q)$.
\end{ex}

\begin{defn}\label{def_prime}
An element $a \notin R^{\times}$ of a commutative ring $R$ is called \textit{irreducible} when $a = b c$ for $b,c \in R$ implies $b \in R^{\times}$ or $c \in R^{\times}$. It is called \textit{prime} when for all $b,c \in R$ the divisibility of $b c$ by $a$ implies that $a$ divides $b$ or $c$.
\end{defn}

\begin{rem}\label{rem_prime}
a) A field does not have any irreducible elements due to Remark \ref{rem_ring}b) and $0 = 0 \cdot 0$. But in a ring zero is always a prime since it is divisible by every ring element.

b) A non-zero prime $\pi$ of an integral domain $\mathbb{O}$ is always irreducible since $1 \pi = \pi = \beta \gamma$ for $\beta,\gamma \in \mathbb{O}$ shows that $\pi$ is a divisor of $\beta$ or $\gamma$ by assumption; say $\beta = \alpha \pi$ for some $\alpha \in \mathbb{O}$. Then we have $\pi = \alpha \pi \gamma$. Since there are no zero divisors we can cancel out $\pi \ne 0$ and obtain $1 = \alpha \gamma$ which shows that $\gamma$ is a unit.
\end{rem}

\begin{prop}\label{prop_pid}
In a principal ideal domain the irreducible elements are exactly the non-zero primes.
\end{prop}
\begin{proof}
Because $0=0 \cdot 0$ is not a unit an irreducible element is not zero. It remains to show that an irreducible element $\pi$ of a principal ideal domain $\mathbb{O}$ is prime. When $\pi$ does not divide $\beta \in \mathbb{O}$ then $\pi$ and $\beta$ are coprime because of irreducibility. Then by B\'ezout's theorem there are $\lambda,\mu \in \mathbb{O}$ s.t. $\lambda \pi + \mu \beta = 1$, hence $\gamma = \gamma \lambda \pi + \mu \beta \gamma$ for arbitrary $\gamma \in \mathbb{O}$. If $\pi$ divides $\beta \gamma$ then also $\gamma$.
\end{proof}

\begin{ex}
a) The irreducible elements of the principal ideal domain $\mathbb{Z}$ (s. Example \ref{ex_euclideanDomain}a)!) are those integers $p$ that have exactly two positive divisors, namely one and $|p|$. So, according to Remark \ref{rem_prime}b) and Proposition \ref{prop_pid} an element of $\mathbb{Z}$ is prime if and only if it is either zero or $\pm p$ for a positive irreducible integer $p$.

b) In the principal ideal domain $\mathbb{Z}[i] = \lbrace x + i y \, | \, x,y \in \mathbb{Z} \rbrace$ (s. Example \ref{ex_euclideanDomain}c)!) an element $x + i y$ whose \textit{norm} $x^2+y^2$ is a natural prime number is irreducible. Such a natural prime (as a norm value) is either $2$ or congruent to $1$ modulo $4$. Any other irreducible element $p u$ is a natural prime number $p \equiv 3 \mod 4$ times a unit $u \in \mathbb{Z}[i]^{\times} = \lbrace -1 , 1 , -i , i \rbrace$.
\end{ex}

The following is called the \textit{Fundamental Theorem of Number Theory}.

\begin{thm}\label{thm_fundamentalNTh}
In a principal ideal domain every non-zero element is a unit or a product of finitely many primes. This product is unique up to a unit factor and up to the order of prime factors.
\end{thm}
\begin{proof}
For a sequence $(\alpha_n)_{n \in \mathbb{N}}$ of elements of a principal ideal domain $\mathbb{O}$ with $(\alpha_n) \subseteq (\alpha_{n+1})$ for all $n \in \mathbb{N}$ the union $I$ of all $(\alpha_n)$ is also an ideal of $\mathbb{O}$. Hence by assumption there is some $\beta \in \mathbb{O}$ s.t. $I = (\beta)$. By definition of $I$ there is some $n \in \mathbb{N}$ with $\beta \in (\alpha_n)$. It follows $I = (\alpha_n)$. This implies that a non-zero element is a unit or a product of $n \in \mathbb{N}$ irreducible divisors $\pi_1,...,\pi_n$. Due to Proposition \ref{prop_pid} this shows the first assertion. Now assume that such a product is divisible by a prime $\pi$. Then $\pi$ divides at least one of the $\pi_k$. Because $\pi_k$ is irreducible it holds $\pi_k = \varepsilon \pi$ for a unit $\varepsilon$. So cancelling out $\pi$ from the product yields the product of $n-1$ of the factors $\pi_k$ up to a unit factor. Hence the uniqueness follows by induction on $n$.
\end{proof}

\begin{ex}\label{ex_nonPID}
The following examples a) and b) show that the integral domain $\mathbb{O} := \left\lbrace x + i \sqrt{5} y : x,y \in \mathbb{Z} \right\rbrace \subset \mathbb{C}$ is not principal due to Proposition \ref{prop_pid}.

a) Show that $2$ is irreducible in $\mathbb{O}$.

b) Show that $2$ is not prime in $\mathbb{O}$. Hint: $2 \cdot 3 = (1 + i \sqrt{5})(1 - i \sqrt{5})$.

c) Conclude that $2$ is not a product of primes of $\mathbb{O}$.

d) Show that the ideal $(2 , 1 + i \sqrt{5}) \subset \mathbb{O}$ is not principal.
\end{ex}

\begin{rem}\label{rem_fundamentalThm}
As a consequence of Theorem \ref{thm_crt} and Theorem \ref{thm_fundamentalNTh} every non-zero and non-unit integer $m$ induces a canonical isomorphism between $\mathbb{Z}/(m)$ and $\mathbb{Z}/(p_1^{e_1}) \times ... \times \mathbb{Z}/(p_k^{e_k})$ for some primes $p_1,...,p_k$ and natural numbers $k,e_1,...,e_k$ that are determined by $m = \pm p_1^{e_1} ... p_k^{e_k}$. And this ring isomorphism induces a group isomorphism between the corresponding unit groups due to Example \ref{ex_ringProduct}. For an odd prime $p \in \mathbb{N}$ and an exponent $e > 1$ the only natural numbers that are smaller than $p^e$ and not divisible by $p$ are the numbers $k p + 1, k p + 2, ... , k p + p - 1$ for $k \in \lbrace 0, 1 , ... , p^{e-2} \rbrace$. So the unit group $\left(\mathbb{Z}/(p^e)\right)^{\times}$ has $(p-1)p^{e-1}$ elements. According to Example \ref{ex_euclideanDomain}b) there exists a number $g \in \mathbb{Z}$ s.t. $\textnormal{ord}(g \mod p) = p-1$. If $g^{p-1} \equiv 1 \mod p^2$ then $(g+p)^{p-1} \equiv g^{p-1} + (p-1) g^{p-2} p \equiv 1 + a p \mod p^2$ for an integer $a$ not divisible by $p$. So we may assume without loss of generality $g^{p-1} \equiv 1 + a p \mod p^2$ for such an $a$. Then  $g^{(p-1)p^{e-2}} \equiv 1 + a p^{e-1} \mod p^e$ is not congruent to $1$ modulo $p^e$, nor is $g^d$ for any other proper divisor $d$ of $(p-1)p^{e-1}$. Thus we have shown that $\left(\mathbb{Z}/(p^e)\right)^{\times}$ is isomorphic to the cyclic group $C_{(p-1)p^{e-1}}$.
\end{rem}

\begin{ex}\label{ex_RSA}
For two different prime numbers $p,q \in \mathbb{N}$ the unit group of $\mathbb{Z}/(pq)$ is isomorphic to the direct product of the two cyclic groups $C_{p-1}$ and $C_{q-1}$. In case $p$ and $q$ are secret they can be used for RSA, a very popular asymmetric key scheme which was invented first (but not yet published) around 1970 by members of the British secret service GCHQ. Ignoring implementation and protocol details the RSA scheme consists of a non-secret 'public key' $(e,n := pq)$ and a secret 'private key' $(d,n)$ with the property $ed \equiv 1 \mod{\varphi}$ where $\varphi$ denotes a multiple of the least common multiple of $p-1$ and $q-1$, like e.g. $\varphi = (p-1)(q-1)$. Show that then $(m^e)^d = m^{ed} = m$ for all $m \in \mathbb{Z}$ (coprime with $n$).
\end{ex}

Due to Remark \ref{rem_ideal}c) for every polynomial $q$ of $n \in \mathbb{N}$ variables with integral coefficients a solution $(x_1,...,x_n) \in \mathbb{Z}^n$ of the \textit{diophantine equation} $q(x_1,...,x_n)=0$ implies solvability of the congruence $q(x_1,...,x_n) \equiv 0 \mod{m}$ for every module $m \in \mathbb{N}$. In order to investigate the latter congruences K. Hensel invented the $p$-adic numbers in 1897 (s. e.g. \cite{Hensel}). Its axiomatic characterisation and properties are collected in \cite{Cassels}, ch.3. (See there for more details!)

\begin{defn}\label{def_local}
For a prime number $p \in \mathbb{N} \subset \mathbb{Z}$ and a sequence $(a_0,a_1,...)$ of integers with $0 \le a_i < p$ the sequence\footnote{similarly constructed as a series out of a sequence} $(a_0,a_0 + a_1 p, a_0 + a_1 p + a_2 p^2,...)$ is called a $p$\textit{-adic integer}.
\end{defn}

\begin{rem}\label{rem_local}
a) The set $\mathbb{Z}_p$ of all $p$-adic integers is an integral domain under coordinate-wise addition and multiplication whereby the $j$-th coordinate of the operation result must be reduced modulo $p^{j+1}$ for all $j \in \mathbb{N}_0$. The integers are embedded via $z := (z \mod{p} , z \mod{p^2} , ...)$ for all $z \in \mathbb{Z}$. A $p$-adic integer $(a_0,...)$ ($0 \le a_0 < p$) is a unit if and only if $a_0 \ne 0$.

b) The set $\mathbb{Q}_p$ of $p$\textit{-adic numbers} $p^m \alpha$ with $m \in \mathbb{Z}$ and $\alpha \in \mathbb{Z}_p$ is a (\textit{local}) field isomorphic to the quotient field of $\mathbb{Z}_p$. It contains $\mathbb{Q}$ since it contains $\mathbb{Z}_p$ which contains $\mathbb{Z}$.

c) Every non-zero $p$-adic number $\kappa$ has a unique representation $p^m \varepsilon$ for some $m \in \mathbb{Z}$ and $\varepsilon \in \mathbb{Z}_p^{\times}$. The exponent $\nu(\kappa) := m$ has the three properties:
\begin{itemize}
\item $\nu(\kappa \lambda) = \nu(\kappa) + \nu(\lambda)$
\item $\nu(\kappa + \lambda) \ge \min(\nu(\kappa),\nu(\lambda))$
\item $\nu(\kappa + \lambda) = \min(\nu(\kappa),\nu(\lambda)) \textnormal{ in case } \nu(\kappa) \ne \nu(\lambda)$
\end{itemize}
for all $\kappa, \lambda \in \mathbb{Q}_p^{\times}$. With $\nu(0):= \infty$ it holds $\mathbb{Z}_p = \lbrace \kappa \in \mathbb{Q}_p : \nu(\kappa) \ge 0 \rbrace$. For the (\textit{non-archimedean}) \textit{valuation} $|\kappa|_p := p^{-\nu(\kappa)}$ the above properties imply the norm (or modulus) properties of Definition \ref{def_norm} with $V := \mathbb{Q}_p$ and $|0|_p := 0$. According to \cite{Cassels}, ch.3, lem.1.2 the field $\mathbb{Q}_p$ is complete (in the sense of Remark \ref{rem_completeness}b)) with respect to that norm. And from Theorem \ref{thm_fundamentalNTh} it follows
\begin{equation*}
\prod\limits_{p \in \mathbb{P} \cup \lbrace \infty \rbrace}^{} |r|_p = 1
\end{equation*}
for every $r \in \mathbb{Q}^{\times}$ whereby $\mathbb{P}$ denotes the set of natural primes and $|\cdot|_\infty$ the usual modulus function.

d) For every non-zero ideal $I$ of $\mathbb{Z}_p$ there is an element $\mu \in I$ with minimal exponent $m := \nu(\mu) \in \mathbb{N}_0$. And then it holds $I = (\mu) = (p^m) = (p)^m$. In particular $\mathbb{Z}_p$ is a principal ideal domain.
\end{rem}

\begin{thm}\label{thm_congruence}
For $p \in \mathbb{P}$ and a polynomial $q$ of $n \in \mathbb{N}$ variables with integral coefficients the congruence $q(x_1,...,x_n) \equiv 0 \mod{p^k}$ is solvable in $\mathbb{Z}^n$ for every $k \in \mathbb{N}$ if and only if $q(x_1,...,x_n)=0$ is solvable in $\mathbb{Z}_p^n$.
\end{thm}
\begin{proof}
See \cite{BS}, ch.1.5, thm.1!
\end{proof}

\begin{cor}\label{cor_congruence}
For $p \in \mathbb{P}$ and a quadratic form $q$ of $n \in \mathbb{N}$ variables with integral coefficients the equation $q(x_1,...,x_n)=0$ has a non-trivial solution in $\mathbb{Z}_p^n$ if and only if for every $k \in \mathbb{N}$ the congruence $q(x_1,...,x_n) \equiv 0 \mod{p^k}$ has a solution $(x_1,...,x_n) \in \mathbb{Z}^n$ with $p$ not dividing $x_j$ for some $j \in \mathbb{N}_n$.
\end{cor}
\begin{proof}
See \cite{BS}, ch.1.5, thm.2!
\end{proof}

\begin{rem}
Remark \ref{rem_fundamentalThm} and the latter Theorem (or its Corollary about non-trivial solutions) show that $q(x_1,...,x_n) \equiv 0 \mod{m}$ is solvable in $\mathbb{Z}^n$ for every module $m \in \mathbb{N}$ if and only if $q(x_1,...,x_n)=0$ is solvable in $\mathbb{Z}_p^n$ for every prime $p > 0$.
\end{rem}

\begin{ex}
The congruence $x^2 \equiv 2 \mod{3}$ has no integral solution. Hence there is no $x \in \mathbb{Z}_3$ with $x^2 = 2$.
\end{ex}


\end{document}